\documentclass[lectures]{pcms-l}
\usepackage[alphabetic,backrefs,nobysame]{amsrefs}
\usepackage{amssymb}
\usepackage[all]{xy}

\numberwithin{subsection}{section}

\numberwithin{equation}{section}

\theoremstyle{plain}
\newtheorem{thm}[equation]{Theorem}
\newtheorem{prop}[equation]{Proposition}
\newtheorem{cor}[equation]{Corollary}
\newtheorem{lemma}[equation]{Lemma}
\newtheorem{conj}[equation]{Conjecture}

\theoremstyle{definition}
\newtheorem{defn}[equation]{Definition}
\newtheorem{defns}[equation]{Definitions}

\theoremstyle{remark}
\newtheorem{rem}[equation]{Remark}
\newtheorem{rems}[equation]{Remarks}
\newtheorem{exer}[equation]{Exercise}
\newtheorem{exers}[equation]{Exercises}
\newtheorem{rem-exer}[equation]{Remark/Exercise}
\newtheorem{rem-exers}[equation]{Remark/Exercises}

\newcommand{\CC}{\mathcal{C}}
\newcommand{\CCbar}{\overline{\mathcal{C}}}
\newcommand{\DD}{\mathcal{D}}
\newcommand{\DDbar}{\overline{\mathcal{D}}}

\newcommand{\XX}{\mathcal{X}}
\newcommand{\XXbar}{\overline{\mathcal{X}}}
\newcommand{\YY}{\mathcal{Y}}
\newcommand{\YYbar}{\overline{\mathcal{Y}}}
\newcommand{\ZZ}{{\mathcal{Z}}}
\newcommand{\WW}{{\mathcal{W}}}

\renewcommand{\O}{\mathcal{O}}
\newcommand{\FF}{\mathcal{F}}
\newcommand{\EE}{\mathcal{E}}
\newcommand{\EEbar}{\overline{\mathcal{E}}}
\newcommand{\LL}{\mathcal{L}}
\newcommand{\MM}{\mathcal{M}}

\newcommand{\OO}{\mathcal{O}}
\newcommand{\F}{\mathbb{F}}
\newcommand{\Fp}{{\mathbb{F}_p}}

\newcommand{\Fq}{{\mathbb{F}_q}}
\newcommand{\Fqn}{{\mathbb{F}_{q^n}}}

\newcommand{\Fpbar}{{\overline{\mathbb{F}}_p}}
\newcommand{\Fqbar}{{\overline{\mathbb{F}}_q}}

\newcommand{\kbar}{{\overline{k}}}
\newcommand{\Kbar}{{\overline{K}}}
\newcommand{\ratto}{{\dashrightarrow}}
\newcommand{\Ql}{{\mathbb{Q}_\ell}}
\newcommand{\Zl}{{\mathbb{Z}_\ell}}
\newcommand{\Qlbar}{{\overline{\mathbb{Q}}_\ell}}
\newcommand{\Qp}{{\mathbb{Q}_p}}
\newcommand{\Zp}{{\mathbb{Z}_p}}

\newcommand{\Z}{\mathbb{Z}}

\newcommand{\Q}{\mathbb{Q}}

\newcommand{\C}{\mathbb{C}}
\newcommand{\A}{\mathbb{A}}
\renewcommand{\P}{\mathbb{P}}

\newcommand{\m}{\mathfrak{m}}
\newcommand{\n}{\mathfrak{n}}

\newcommand{\T}{\mathcal{T}}

\newcommand{\sha}{{\hbox to 10pt{\rlap{\hskip2.8pt\vrule
height6pt\hskip1.6pt\vrule height6pt\hskip1.6pt
\vrule height6pt}\hskip1pt\vrule height0.8pt width 8pt\hskip1pt}}}

\newcommand{\<}{\langle}
\renewcommand{\>}{\rangle}
\newcommand{\into}{\hookrightarrow}
\newcommand{\onto}{\twoheadrightarrow}
\newcommand{\isoto}{\tilde{\to}}
\newcommand{\tensor}{\otimes}
\newcommand{\compose}{\circ}

\newcommand{\nodiv}{\not|}
\def\nodiv{\mathrel{\mathchoice{\not|}{\not|}{\kern-.2em\not\kern.2em|}{\kern-.2em\not\kern.2em|}}}

\newcommand{\SL}{\mathrm{SL}}
\newcommand{\GL}{\mathrm{GL}}

\newcommand{\G}{\mathbb{G}}

\DeclareMathOperator{\coker}{coker}
\DeclareMathOperator{\tr}{Tr}

\DeclareMathOperator{\ind}{Ind}

\DeclareMathOperator{\cond}{Cond}
\DeclareMathOperator{\ord}{ord}
\DeclareMathOperator{\rk}{Rank}
\DeclareMathOperator{\dvsr}{div}
\DeclareMathOperator{\RP}{Re}

\DeclareMathOperator{\Hom}{Hom}
\DeclareMathOperator{\Aut}{Aut}

\DeclareMathOperator{\Pic}{Pic}
\DeclareMathOperator{\DivCorr}{DivCorr}
\DeclareMathOperator{\ps}{\underline{Pic}}
\DeclareMathOperator{\Picvar}{PicVar}
\DeclareMathOperator{\NS}{NS}
\DeclareMathOperator{\Num}{Num}
\DeclareMathOperator{\Ho}{Homol}
\DeclareMathOperator{\Br}{Br}

\DeclareMathOperator{\gal}{Gal}
\DeclareMathOperator{\spec}{Spec}
\DeclareMathOperator{\en}{End}
\DeclareMathOperator{\mor}{Mor}

\DeclareMathOperator{\divr}{\rm Div}

\DeclareMathOperator{\Fr}{Fr}

\def\clap#1{\hbox to 0pt{\hss#1\hss}} 
\def\mathllap{\mathpalette\mathllapinternal}

\def\mathllapinternal#1#2{%
\llap{$\mathsurround=0pt#1{#2}$}}

\begin{document}

\frontmatter
\tableofcontents

\mainmatter

\LectureSeries{Elliptic curves over function fields \author{Douglas Ulmer}}
\address{School of Mathematics, Georgia Institute of Technology,
Atlanta, GA~~30332}
\email{douglas.ulmer@math.gatech.edu}

\section*{Introduction}

These are the notes from a course of five lectures at the 2009 Park
City Math Institute.  The focus is on {\it elliptic curves\/} over
function fields over {\it finite fields\/}.  In the first three
lectures, we explain the main classical results (mainly due to Tate)
on the Birch and Swinnerton-Dyer conjecture in this context and its
connection to the Tate conjecture about divisors on surfaces.  This
is preceded by a ``Lecture~0'' on background material.  In the
remaining two lectures, we discuss more recent developments on
elliptic curves of large rank and constructions of explicit points in
high rank situations.

A great deal of this material generalizes naturally to the context of
curves and Jacobians of {\it any genus\/} over function fields over
{\it arbitrary ground fields\/}.  These generalizations were discussed
in a course of 12 lectures at the CRM in Barcelona in February, 2010,
and will be written up as a companion to these notes, see
\cite{UlmerCRM}.  Unfortunately, theorems on unbounded ranks over
function fields are currently known only in the context of finite
ground fields.

Finally, we mention here that very interesting theorems of
Gross-Zagier type exist also in the function field context.  These
would be the subject of another series of lectures and we will not say
anything more about them in these notes.

It is a pleasure to thank the organizers of the 2009 PCMI for the
invitation to speak, the students for their interest, enthusiasm, and
stimulating questions, and the ``elder statesmen''---Bryan Birch, Dick
Gross, John Tate, and Yuri Zarhin---for their remarks and
encouragement.  Thanks also to Keith Conrad for bringing the
fascinating historical articles of Roquette \cite{Roquette06} to my
attention.  Last but not least, thanks are due as well to Lisa Berger,
Tommy Occhipinti, Karl Rubin, Alice Silverberg, Yuri Zarhin, and an
anonymous referee for their suggestions and \TeX{}nical advice.

\setcounter{lecture}{-1}
\lecture{Background on curves and function fields}
This ``Lecture~0'' covers definitions and notations that are
probably familiar to many readers and that were reviewed very quickly
during the PCMI lectures.  Readers are invited to skip it and refer
back as necessary.  

\section{Terminology}
Throughout, we use the language of schemes.  This is necessary to be
on firm ground when dealing with some of the more subtle aspects
involving non-perfect ground fields and possibly non-reduced group
schemes.  However, the instances where we use any hard results from this
theory are isolated and students should be able to follow readily the main
lines of discussion, perhaps with the assistance of a friendly
algebraic geometer. 

Throughout, a {\it variety\/} over a field $F$ is a separated, reduced
scheme of finite type over $\spec F$.  A {\it curve\/} is a variety
purely of dimension 1 and a {\it surface\/} is a variety purely of
dimension 2.

\section{Function fields and curves}\label{s:ffs}
Throughout, $p$ will be a prime number and $\Fq$ will denote the field
with $q$ elements with $q$ a power of $p$.  We write $\CC$ for a
smooth, projective, and absolutely irreducible curve of genus $g$ over
$\Fq$ and we write $K=\Fq(\CC)$ for the function field of $\CC$ over
$\Fq$.  The most important example is when $\CC=\P^1$, the projective
line, in which case $K=\Fq(\CC)=\Fq(t)$ is the field of rational
functions in a variable $t$ over $\Fq$.

We write $v$ for a closed point of $\CC$, or equivalently for an
equivalence class of valuations of $K$.  For each such $v$ we write
$\O_{(v)}$ for the local ring at $v$ (the ring of rational functions
on $\CC$ regular at $v$), $\m_v\subset\O_{(v)}$ for the maximal ideal
(those functions vanishing at $v$), and $\kappa_v=\O_{(v)}/\m_v$ for
the residue field at $v$.  The extension $\kappa_v/\Fq$ is finite and
we set $\deg(v)=[\kappa_v:\Fq]$ and $q_v=q^{\deg(v)}$ so that
$\kappa_v\cong\F_{q_v}$.

For example, in the case where $\CC=\P^1$, the ``finite'' places of
$\CC$ correspond bijectively to monic irreducible polynomials
$f\in\Fq[t]$.  If $v$ corresponds to $f$, then $\O_{(v)}$ is the set
of ratios $g/h$ where $g,h\in\Fq[t]$ and $f$ does not divide $h$.  The
maximal ideal $\m_v$ consists of ratios $g/h$ where $f$ does divide
$g$, and the degree of $v$ is the degree of $f$ as a polynomial in
$t$.  There is one more place of $K$, the ``infinite'' place
$v=\infty$.  The local ring consists of ratios $g/h$ with
$g,h\in\Fq[t]$ and $\deg(g)\le\deg(h)$.  The maximal ideal consists of
ratios $g/h$ where $\deg(g)<\deg(h)$ and the degree of $v=\infty$ is
1.  The finite and infinite places of $\P^1$ give all closed points of
$\P^1$.

We write $K^{sep}$ for a separable closure of $K$ and let
$G_K=\gal(K^{sep}/K)$.  We write $\Fqbar$ for the algebraic closure of
$\Fq$ in $K^{sep}$.  For each place $v$ of $K$ we have the
decomposition group $D_v$ (defined only up to conjugacy), its normal
subgroup the inertia group $I_v\subset D_v$, and $\Fr_v$ the
(geometric) Frobenius at $v$, a canonical generator of the quotient
$D_v/I_v\cong\gal(\Fqbar/\Fq)$ that acts as $x\mapsto x^{q_v^{-1}}$ on
the residue field at a place $w$ dividing $v$ in a finite extension
$F\subset K^{sep}$ unramified over $v$.

General references for this section and the next are
\cite{GoldschmidtAFPC}, \cite{RosenNTFF}, and \cite{StichtenothAFFC}.

\section{Zeta functions}\label{s:zetas}
Let $\XX$ be a variety over the finite field $\Fq$.  Extending the
notation of the previous section, if $x$ is a closed point of $\XX$,
we write $\kappa_x$ for the residue field at $x$, $q_x$ for its
cardinality, and $\deg(x)$ for $[\kappa_x:\Fq]$.

We define the $Z$ and $\zeta$ functions of $\XX$ via Euler products:
$$Z(\XX,T)=\prod_x\left(1-T^{\deg(x)}\right)^{-1}$$
and 
$$\zeta(\XX,s)=Z(\XX,q^{-s})=\prod_x\left(1-q_x^{-s}\right)^{-1}$$
where the products are over the closed points of $\XX$.
It is a standard exercise to show that
$$Z(\XX,T)=\exp\left(\sum_{n\ge1}N_n\frac{T^n}{n}\right)$$
where $N_n$ is the number of $\F_{q^n}$-valued points of $\XX$.  
It follows from a crude estimate for the number of $\F_{q^n}$ points
of $\XX$ that the Euler product defining $\zeta(\XX,s)$
converges in the half plane $\RP(s)>\dim \XX$.

If $\XX$ is smooth and projective, then it is known that
$Z(\XX,T)$ is a rational function of the form
$$\frac{\prod_{i=0}^{\dim\XX-1}P_{2i+1}(T)}{\prod_{i=0}^{\dim\XX}P_{2i}(T)}$$
where $P_0(T)=(1-T)$, $P_{2\dim\XX}(T)=(1-q^{\dim\XX}T)$, and for all
$0\le i\le2\dim\XX$ $P_i(T)$ is a polynomial with integer coefficients
and constant term 1.  We denote the inverse roots of $P_i$ by
$\alpha_{ij}$ so that
$$P_i(T)=\prod_j(1-\alpha_{ij}T)$$

The inverse roots $\alpha_{ij}$ of $P_i(T)$ are algebraic integers
that have absolute value $q^{i/2}$ in every complex embedding.  (We
say that they are {\it Weil numbers of size $q^{i/2}$\/}.)  It follows
that $\zeta(\XX,s)$ has a meromorphic continuation to the whole $s$
plane, with poles on the lines $\RP s\in\{0,\dots,\dim\XX\}$ and
zeroes on the lines $\RP s\in\{1/2,\dots,\dim\XX-1/2\}$.  This is the
analogue of the Riemann hypothesis for $\zeta(\XX,s)$.

It is also known that the set of inverse roots of $P_i(T)$ (with
multiplicities) is stable under $\alpha_{ij}\mapsto q/\alpha_{ij}$.
Thus $\zeta(\XX,s)$ satisfies a functional equation when $s$ is
replaced by $\dim\XX-s$.

In the case where $\XX$ is a curve, $P_1(T)$ has
degree $2g$ ($g=$ the genus of $\CC$) and has the form
$$P_1(T)=1+\cdots+q^gT^{2g}=\prod_{j=1}^{2g}(1-\alpha_{1j}T).$$
Thus $\zeta(\CC,s)$ has simple poles for $s\in\frac{2\pi i}{\log q}\Z$
and $s\in1+\frac{2\pi i}{\log q}\Z$ and its zeroes lie on the line
$\RP s=1/2$.

For a fascinating history of the early work on zeta functions and the
Riemann hypothesis for curves over finite fields, see
\cite{Roquette06} and parts I and II of that work.

\section{Cohomology}\label{s:cohomology}
Assume that $\XX$ is a smooth projective variety over $k=\Fq$.
We write $\XXbar$ for $\XX\times_{\Fq}\Fqbar$.  Note that
$G_k=\gal(\Fqbar/\Fq)$ acts on $\XXbar$ via the factor $\Fqbar$.

Choose a prime $\ell\neq p$.  We have $\ell$-adic cohomology groups
$H^i(\XXbar,\Ql)$ which are finite-dimensional $\Ql$-vector spaces
and which vanish unless $0\le i\le 2\dim\XX$.

Functoriality in $\XXbar$ gives a continuous action of
$\gal(\Fqbar/\Fq)$.  Since the geometric Frobenius
($\Fr_q(a)=a^{q^{-1}}$) is a topological generator of
$\gal(\Fqbar/\Fq)$, the characteristic polynomial of $\Fr_q$ on
$H^i(\XXbar,\Ql)$ determines the eigenvalues of the action of
$\gal(\Fqbar/\Fq)$; in fancier language, it determines the action up
to semi-simplification.

An important result (inspired by \cite{Weil49} and proven in
great generality in \cite{SGA5})  says that the factors
$P_i$ of $Z(\XX,t)$ are characteristic polynomials of Frobenius:
\begin{equation}\label{eq:P-cohom}
P_i(T)=\det(1-T\Fr_q|H^i(\XXbar,\Ql)).
\end{equation}
From this point of view, the functional equation and Riemann
hypothesis for $Z(\XX,T)$ are statements about duality and purity.

To discuss the connections, we need more notation.  Let
$\Zl(1)=\varprojlim_n\mu_{\ell^n}(\Fqbar)$ and
$\Ql(1)=\Zl(1)\tensor_{\Zl}\Ql$, so that $\Ql(1)$ is a one-dimensional
$\Ql$-vector space on which $\gal(\Fqbar/\Fq)$ acts via the
$\ell$-adic cyclotomic character.  More generally, for $n>0$ set
$\Ql(n)=\Ql(1)^{\tensor n}$ ($n$-th tensor power) and
$\Ql(-n)=\Hom(\Ql(n),\Ql)$, so that for all $n$, $\Ql(n)$ is a
one-dimensional $\Ql$-vector space on which $\gal(\Fqbar/\Fq)$ acts
via the $n$th power of the $\ell$-adic cyclotomic character.

We have $H^0(\XXbar,\Ql)\cong\Ql$ (with trivial Galois action) and
$H^{2\dim\XX}(\XXbar,\Ql)\cong\Ql(\dim\XX)$.  The functional equation
follows from the fact that we have a canonical non-degenerate, Galois
equivariant pairing
$$H^i(\XXbar,\Ql)\times H^{2\dim\XX-i}(\XXbar,\Ql)\to 
H^{2\dim\XX}(\XXbar,\Ql)\cong\Ql(\dim\XX).$$ 
Indeed, the non-degeneracy of this pairing implies that if $\alpha$ is
an eigenvalue of $\Fr_q$ on $H^i(\XXbar,\Ql)$, then
$q^{\dim\XX}/\alpha$ is an eigenvalue of $\Fr_q$ on
$H^{2\dim\XX-i}(\XXbar,\Ql)$.

The Riemann hypothesis in this context is the statement that the
eigenvalues of $\Fr_q$ on $H^i(\XXbar,\Ql)$ are algebraic integers
with absolute value $q^{i/2}$ in every complex embedding.

See \cite{SGA4.5} or \cite{MilneEC} for an overview of \'etale
cohomology and its connections with the Weil conjectures.

\section{Jacobians}
\numberwithin{equation}{subsection}
\subsection{Picard and Albanese properties}\label{ss:pic-alb}
We briefly review two (dual) universal properties of the Jacobian of a
curve that we will need.  See \cite{Milne86jv} for more details.

We assume throughout that the curve $\CC$ has an $\Fq$-rational point
$x$, i.e., a closed point with residue field $\Fq$.  If $T$ is another
connected variety over $\Fq$ with an $\Fq$-rational point $t$, a {\it
  divisorial correspondence\/} between $(C,x)$ and $(T,t)$ is an
invertible sheaf $\LL$ on $C\times_\Fq T$ such that $\LL|_{C\times t}$
and $\LL|_{x\times T}$ are trivial.  Two divisorial correspondences
are equal when they are isomorphic as invertible sheaves.  Note that
the set of divisorial correspondences between $(\CC,x)$ and $(T,t)$
forms a group under tensor product and is thus a subgroup of
$\Pic(\CC\times T)$.  We write
$$\DivCorr((\CC,x),(T,t))\subset\Pic(\CC\times T)$$
for this subgroup.  One may think of a divisorial correspondence as
giving a family of invertible sheaves on $C$: $s\mapsto\LL|_{C\times
  s}$.

Let $J=J_\CC$ be the Jacobian of $\CC$ and write $0$ for its identity
element.  Then $J$ is a $g$-dimensional abelian variety over $\Fq$ and
it carries the ``universal divisorial correspondence with $C$.''  More
precisely, there is a divisorial correspondence $\MM$ between $(C,x)$
and $(J,0)$ such that if $S$ is another connected variety over $\Fq$
with $\Fq$-rational point $s$ and $\LL$ is a divisorial correspondence
between $(C,x)$ and $(S,s)$, then there is a unique morphism
$\phi:S\to J$ sending $s$ to $0$ such that $\LL=\phi^*\MM$.  (Of
course $\MM$ depends on the choice of base point $x$, but we omit this
from the notation.)

It follows that there is a canonical morphism, the Abel-Jacobi
morphism, $AJ:\CC\to J$ sending $x$ to $0$.  Intuitively, this
corresponds to the family of invertible sheaves parameterized by $\CC$
that sends $y\in\CC$ to $\OO_\CC(y-x)$.  More precisely, let
$\Delta\subset \CC\times C$ be the diagonal, let
$$D=\Delta-x\times \CC-\CC\times x,$$
and let $\LL=\OO_{\CC\times\CC}(D)$ which is a divisorial
correspondence between $(C,x)$ and itself.  The universal property
above then yields the morphism $AJ:\CC\to J$.  It is known that $AJ$
is a closed immersion and that its image generates $J$ as an algebraic
group.

The second universal property enjoyed by $J$ (or rather by $AJ$) is
the Albanese property: it is universal for maps to abelian varieties.
More precisely, if $A$ is an abelian variety and $\phi:\CC\to A$ is a
morphism sending $x$ to 0, then there is a unique homomorphism of
abelian varieties $\psi:J\to A$ such that $\phi=\psi\compose AJ$.

Combining the two universal properties gives a useful connection
between correspondences and homomorphisms: Suppose $\CC$ and $\DD$ are
curves over $\Fq$ with rational points $x\in\CC$ and $y\in\DD$.  Then
we have an isomorphism
\begin{equation}\label{eq:divcor-hom}
\DivCorr((\CC,x),(\DD,y))\cong\Hom(J_\CC,J_\DD).
\end{equation}
Intuitively, given a divisorial
correspondence on $\CC\times\DD$, we get a family of invertible
sheaves on $\DD$ parameterized by $\CC$ and thus a morphism $\CC\to
J_\DD$.  The Albanese property then gives a homomorphism $J_\CC\to
J_\DD$.  We leave the precise version as an exercise, or
see \cite{Milne86jv}*{6.3}.  We will use this isomorphism later to
understand the N\'eron-Severi group of a product of curves.

\subsection{The Tate module}
Let $A$ be an abelian variety of dimension $g$ over $\Fq$, for example
the Jacobian of a curve of genus $g$.  (See \cite{Milne86av} for a
brief introduction to abelian varieties and \cite{MumfordAV} for a
much more complete treatment.)  Choose a prime $\ell\neq p$.  Let
$A[\ell^n]$ be the set of $\Fqbar$ points of $A$ of order dividing
$\ell^n$.  It is a group isomorphic to $(\Z/\ell^n\Z)^{2g}$ with a
linear action of $\gal(\Fqbar/\Fq)$.  We form the inverse limit
$$T_\ell A=\varprojlim_n A[\ell^n]$$ 
where the transition maps are given by multiplication by $\ell$.  Let
$V_\ell A=T_\ell A\tensor_{\Zl}\Ql$, a $2g$-dimensional $\Ql$-vector
space with a linear action of $\gal(\Fqbar/\Fq)$.  It is often called
the {\it Tate module\/} of $A$.

According to Roquette, what we now call the Tate module seems to have
first been used in print by Deuring \cite{Deuring40} as a substitute
for homology in his work on correspondences on curves.  It appears
already in a letter of Hasse from 1935, see \cite{Roquette06}*{p.~36}.

The following proposition is the modern interpretation of the
connection between homology and torsion points.

\begin{prop} Let $A$ be an abelian variety over a field $k$ and let
  $\ell$ be a prime not equal to the characteristic of $k$.  Let
  $V_\ell A$ be the Tate module of $A$ and $(V_\ell A)^*$ its dual as
  a $G_k=\gal(k^{sep}/k)$-module.
\begin{itemize}
\item There is a canonical isomorphism of $G_k$-modules
$$(V_\ell A)^*\cong H^1(A\times\kbar,\Ql).$$
\item If $A$ is the Jacobian of a curve $\CC$ over $k$, then
$$H^1(A\times\kbar,\Ql)\cong
H^1(\CC\times\kbar,\Ql).$$
\end{itemize}
\end{prop}

For a proof of part 1, see \cite{Milne86av}*{15.1} and for part 2, see
\cite{Milne86jv}*{9.6}.

\begin{exers} 
  These exercises are meant to make the Proposition more plausible.
\begin{enumerate}
\item Show that if $A(\C)$ is a complex torus $\C^g/\Lambda$, then the
singular homology $H_1(A(\C),\Ql)$ is canonically isomorphic to
$V_\ell A(\C)$.  (Hint: Use the universal coefficient theorem to
show that $H_1(A(\C),\Z/\ell^n\Z)\cong\Lambda/\ell^n\Lambda$.)
\item (Advanced) Let $\CC$ be a smooth projective curve over an
  algebraically closed field $k$.  Let $\ell$ be a prime not equal to
  the characteristic of $k$.  Use geometric class field theory (as in
  \cite{SerreAGCF}) to show that unramified Galois covers $\CC'\to\CC$
  equipped with an isomorphism $\gal(\CC'/\CC)\cong\Z/\ell\Z$ are in
  bijection with elements of $\Hom(J_\CC[\ell],\Z/\ell\Z)$.  (Make a
  convention to deal with the trivial homomorphism.)  This suggests
  that $H^1(\CC,\Z/\ell\Z)$ ``should be''
  $\Hom(J_\CC[\ell],\Z/\ell\Z)$ and $H_1(C,\Z/\ell\Z)$ ``should be''
  $J_\CC[\ell]$.  The reason we only have ``should be'' rather than a
  theorem is that a non-trivial Galois cover $\CC'\to\CC$ is never
  locally constant in the Zariski topology.  This is a prime
  motivation for introducing the \'etale topology.
\end{enumerate}
\end{exers}

\numberwithin{equation}{section}

\section{Tate's theorem on homomorphisms 
of abelian varieties}\label{s:Tate-thm}

As usual, let $k$ be a finite field and let $A$ and $B$ be two abelian
varieties over $k$.  Choose a prime $\ell$ not equal to the
characteristic of $k$ and form the Tate modules $V_\ell A$ and $V_\ell
B$.  Any homomorphism of abelian varieties $\phi:A\to B$ induces a
homomorphism of Tate modules $\phi_*:V_\ell A\to V_\ell B$ and this
homomorphism commutes with the action of $G_k=\gal(\kbar/k)$ on the
Tate modules.  We get an induced homomorphism
$\Hom_k(A,B)\tensor\Ql\to\Hom_{G_k}\left(V_\ell A,V_\ell B\right)$.
Tate's famous result \cite{Tate66a} asserts that this is an
isomorphism:
\begin{thm}\label{thm:TateIsogThm} 
The map $\phi\mapsto\phi_*$ induces an isomorphism of
  $\Ql$-vector spaces: 
$$\Hom_k(A,B)\tensor\Ql\isoto\Hom_{G_k}\left(V_\ell  A,V_\ell B\right).$$
\end{thm}

We also mention
\cite{Zarhin08} which gives a different proof and a strengthening with
finite coefficients.

We will use Tate's theorem in Theorem~\ref{thm:products} of Lecture~2
to understand the divisors on a product of curves in terms of
homomorphisms between their Jacobians.


\lecture{Elliptic curves over function fields}

In this lecture we discuss the basic facts about elliptic curves over
function fields over finite fields.  We assume the reader has some
familiarity with elliptic curves over global fields such as $\Q$ or
number fields, as explained, e.g., in \cite{SilvermanAEC}, and we will
focus on aspects specific to characteristic~$p$.  The lecture ends
with statements of the main results known about the conjecture of
Birch and Swinnerton-Dyer in this context.

\section{Elliptic curves}\label{s:ell-curves}
\numberwithin{equation}{subsection}
\subsection{Definitions}
We write $k=\Fq$ for the finite field of cardinality $q$ and
characteristic $p$ and we let $K$ be the function field of a smooth,
projective, absolutely irreducible curve $\CC$ over $k$.

An {\it elliptic curve\/} over $K$ is a smooth, projective, absolutely
irreducible curve of genus 1 over $K$ equipped with a $K$-rational
point $O$ that will serve as the origin of the group law.

All the basic geometric facts, e.g., of \cite{SilvermanAEC}*{Ch.~III
  and App.~A}, continue to hold in the context of function fields.  We
review a few of them to establish notation, but will not enter into
full details.

Using the Riemann-Roch theorem, an elliptic curve $E$ over $K$ can
always be presented as a projective plane cubic curve defined by a
Weierstrass equation, i.e., by an equation of the form
\begin{equation}\label{eq:cubic}
Y^2Z+a_1XYZ+a_3YZ^2=X^3+a_2X^2Z+a_4XZ^2+a_6Z^3
\end{equation}
where $a_1,\dots,a_6\in K$.  The origin $O$ is the point at infinity
$[0:1:0]$.  We often give the equation in affine form:
\begin{equation}\label{eq:cubica}
y^2+a_1xy+a_3y=x^3+a_2x^2+a_4x+a_6
\end{equation}
where $x=X/Z$ and $y=Y/Z$.

The quantities $b_2,\dots,b_8,c_4,c_6,\Delta,j$ are defined by the
usual formulas (\cite{SilvermanAEC}*{III.1} or \cite{Deligne75}).  Since
$E$ is smooth, by the following exercise $\Delta\neq0$.

\begin{rem-exers}
  The word ``smooth'' in the definition of an elliptic curve means
  that the morphism $E\to\spec K$ is smooth.  Smoothness of a morphism
  can be tested via the Jacobian criterion (see, e.g.,
  \cite{HartshorneAG}*{III.10.4} or \cite{LiuAGAC}*{4.3.3}).  Show
  that the projective plane cubic~(\ref{eq:cubic}) is smooth if and
  only if $\Delta\neq0$.  Because the ground field $K$ is not perfect,
  smoothness is strictly stronger than the requirement that $E$ be
  regular, i.e., that its local rings be regular local rings
  (cf.~\cite{LiuAGAC}*{4.2.2}).  For example, show that the projective
  cubic defined by $Y^2Z=X^3-tZ^3$ over $K=\Fp(t)$ with $p=2$ or $3$
  is a regular scheme, but is not smooth over $K$.
\end{rem-exers}

\begin{defns} Let $E$ be an elliptic curve over $K$.
\begin{enumerate}
\item We say $E$ is {\it constant\/} if there is an elliptic curve $E_0$ defined
over $k$ such that $E\cong E_0\times_kK$.  Equivalently, $E$ is
constant if it can be defined by a Weierstrass cubic (\ref{eq:cubic})
where the $a_i\in k$.
\item We say $E$ is {\it isotrivial\/} if there exists a finite
  extension $K'$ of $K$ such that $E$ becomes constant over $K'$.
  Note that a constant curve is isotrivial.
\item We say $E$ is {\it non-isotrivial\/} if it is not isotrivial.
  We say $E$ is {\it non-constant\/} if it is not constant.
\end{enumerate}
\end{defns}

\begin{rem-exers}
  Show that $E$ is isotrivial if and only if $j(E)\in k$.  Suppose
  that $E$ is isotrivial, so that $E$ becomes constant over a
  finite extension $K'$ and let $k'$ be the field of constants of $K'$
  (the algebraic closure of $k$ in $K'$).  {\it A priori\/}, the
  definition of isotrivial says that there is an elliptic curve $E_0$
  over $k'$ such that $E\times_KK'\cong E_0\times_{k'}K'$.  Show that
  we may take $K'$ to have field of constants $k$ and $E_0$ to be
  defined over $k$.  Show also that we may take $K'$ to be separable
  and of degree dividing 24 over $K$.
\end{rem-exers}

\begin{exer}
  For any elliptic curve $E$ over $K$, the functor on $K$-algebras
  $L\mapsto\Aut_L(E\times L)$ is represented by a group scheme
  $\underline{\Aut}(E)$.  (Concretely, this means there is a group
  scheme $\underline{\Aut}(E)$ such that for any $K$-algebra $L$,
  $\Aut_L(E\times L)$ is $\underline{\Aut}(E)(L)$, the group of
  $L$-valued points of $\underline{\Aut}(E)$.)  Show that
  $\underline{\Aut}(E)$ is an \'etale group scheme.  Equivalently,
  show that any element of $\Aut_{\overline K}(E)$ is defined over a
  separable extension of $K$.  (This is closely related to the
  previous exercise.)
\end{exer}

\subsection{Examples}\label{ss:examples}
Let $K=\Fp(t)$ with $p>3$ and define elliptic curves
\begin{align*}
E_1:\quad&y^2=x^3+1\\
E_2:\quad&y^2=x^3+t^6\\
E_3:\quad&y^2=x^3+t\\
E_4:\quad&y^2=x^3+x+t.
\end{align*}
Then $E_1\cong E_2$ over $K$ and both are constant, $E_3$ is isotrivial and
non-constant, whereas $E_4$ is non-isotrivial.

For more examples, let $K=\Fp(t)$ (with $p$ restricted as indicated) and define
\begin{align*}
(p\neq3)\qquad&E_5:\quad y^2+ty=x^3\\
(p\neq2)\qquad&E_6:\quad y^2=x^3+tx\\
(p~\text{arbitrary})\qquad&E_7:\quad y^2+xy+ty=x^3\\
(p~\text{arbitrary})\qquad&E_8:\quad y^2+xy=x^3+tx\\
(p~\text{arbitrary})\qquad&E_9:\quad y^2+xy=x^3+t.
\end{align*}
Then $E_5$ and $E_6$ are isotrivial and non-constant whereas $E_7$,
$E_8$, and $E_9$ are non-isotrivial. 

\numberwithin{equation}{section}

\section{Frobenius}
If $X$ is a scheme of characteristic $p$, we define the {\it absolute
  Frobenius\/} morphism $\Fr_X:X\to X$ as usual:  It is the identity
on the underlying topological space and raises functions to the $p$-th
power.  When $X=\spec K$, $\Fr_X$ is just the map of schemes induced
by the ring homomorphism $K\to K$, $a\mapsto a^p$.

Suppose as usual that $K$ is a function field and let $E$ be an
elliptic curve over $K$.  Define a new elliptic curve $E^{(p)}$ over $K$ by the
fiber product diagram:
\begin{equation*}
\xymatrix{\mathllap{E^{(p)}=}\spec K\times_{\spec K}E\ar[r]\ar[d]&E\ar[d]\\
\spec K\ar[r]^{\Fr}&\spec K}
\end{equation*}
More concretely, if $E$ is presented as a Weierstrass cubic as in
equation~(\ref{eq:cubica}), then $E^{(p)}$ is given by the equation
with $a_i$ replaced by $a_i^p$.  The universal property of the fiber
product gives a canonical morphism $\Fr_{E/K}$, the {\it relative Frobenius\/}:
\begin{equation*}
\xymatrix{E\ar[r]^{\Fr_{E/K}}\ar[dr]&E^{(p)}\ar[r]\ar[d]&E\ar[d]\\
&\spec K\ar[r]^{\Fr}&\spec K}
\end{equation*}
By definition $\Fr_{E/K}$ is a morphism over $K$.  In terms of
Weierstrass equations for $E$ and $E^{(p)}$ as above, it is just the
map $(x,y)\mapsto(x^p,y^p)$.  

It is evident that $\Fr_{E/K}$ is an isogeny, i.e., a surjective
homomorphism of elliptic curves, and that its degree is $p$.  We
define $V=V_{E/K}$ to be the dual isogeny, so that
$V_{E/K}\circ\Fr_{E/K}=[p]$, multiplication by $p$ on $E$.

Note that $j(E^{(p)})=j(E)^p$ so that if $E$ is non-isotrivial, $E$
and $E^{(p)}$ are not isomorphic.  Thus, using Frobenius and its
iterates, we see that there are infinitely many non-isomorphic
elliptic curves isogenous to any non-isotrivial $E$.  This is in
marked contrast to the situation over number fields
(cf.~\cite{Faltings86}).

\begin{lemma}\label{l:(p)}
  Let $E$ be an elliptic curve over $K$.  Then $j(E)$ is a $p$-th
  power in $K$ if and only if there exists an elliptic curve $E'$ over
  $K$ such that $E\cong E^{\prime(p)}$.
\end{lemma}

\begin{proof}
  We sketch a fancy argument and pose as an exercise a more
  down-to-earth proof.  Obviously if there is an $E'$ with $E\cong
  E^{\prime(p)}$, then $j(E)=j(E^{\prime(p)})=j(E')^p\in K^p$.
  Conversely, suppose $j(E)\in K^p$ and choose an elliptic curve $E''$
  such that $j(E'')^p=j(E)$.  It follows that $E^{\prime\prime(p)}$ is
  isomorphic to $E$ over a finite separable extension of $K$.  In
  other words, $E$ is the twist of $E^{\prime\prime(p)}$ by a cocycle
  in $H^1(G_K,\Aut_{K^{sep}}(E^{\prime\prime(p)}))$.  But there is a
  canonical isomorphism
  $\Aut_{K^{sep}}(E^{\prime\prime(p)})\cong\Aut_{K^{sep}}(E^{\prime\prime})$
  and twisting $E''$ by the corresponding element of
  $$H^1(G_K,\Aut_{K^{sep}}(E^{\prime\prime}))\cong
  H^1(G_K,\Aut_{K^{sep}}(E^{\prime\prime(p)}))$$ 
  we obtain an elliptic curve $E'$ with $E^{\prime(p)}\cong E$.
\end{proof}

\begin{exer}
Use explicit equations, as in \cite{SilvermanAEC}*{Appendix~A}, to prove
the lemma.
\end{exer}

\section{The Hasse invariant}
Let $F$ be a field of characteristic $p$ and $E$ an elliptic curve
over $F$.  Let $\O_E$ be the sheaf of regular functions on $E$ and let
$\Omega^1_E$ be the sheaf of K\"ahler differentials on $E$.  The
coherent cohomology group $H^1(E,\O_E)$ is a one-dimensional
$F$-vector space and is Serre dual to the space of invariant
differentials $H^0(E,\Omega^1_E)$.  Choose a non-zero differential
$\omega\in H^0(E,\Omega^1_E)$ and let $\eta$ be the dual element of
$H^1(E,\O_E)$.  The absolute Frobenius $\Fr_E$ induces a ($p$-linear)
homomorphism:
$$\Fr_E^*:H^1(E,\O_E)\to H^1(E,\O_E).$$
We define an element $A=A(E,\omega)$ of $F$ by requiring that
$\Fr_E^*(\eta)=A(E,\omega)\eta$.  This is the {\it Hasse invariant\/}
of $E$.  It has weight $p-1$ in the sense that
$A(E,\lambda^{-1}\omega)=\lambda^{p-1}A(E,\omega)$ for all $\lambda\in
F^\times$.

Suppose $E$ is given by a Weierstrass equation (\ref{eq:cubica}) and
$\omega=dx/(2y+a_1x+a_3)$.  If $p=2$, then $A(E,\omega)=a_1$.  If
$p>2$, choose an equation with $a_1=a_3=0$.  Then $A(E,\omega)=$ the
coefficient of $x^{p-1}$ in $(x^3+a_2x^2+a_4x+a_6)^{(p-1)/2}$.  These
assertions follow from \cite{KatzMazurAM}*{12.4} where several other
calculations of $A$ are also presented.

Recall that $E/K$ is {\it ordinary\/} if the group of $p$-torsion
points $E(\overline K)[p]\neq 0$ and {\it supersingular\/} otherwise.
It is known that $E$ is supersingular if and only if $A(E,\omega)=0$
(e.g., \cite{KatzMazurAM}*{12.3.6 and 12.4}) and in this case
$j(E)\in\F_{p^2}$ (e.g., \cite{KatzMazurAM}*{proof of 2.9.4}).
(Alternatively, one may apply \cite{SilvermanAEC}*{V.3.1} to $E$ over
$\overline K$.)  In particular, if $E$ is supersingular, then it must
also be isotrivial.

\section{Endomorphisms}
The classification of endomorphism rings in
\cite{SilvermanAEC}*{III.9} goes over verbatim to the function field
case: $\en_{\overline K}(E)$ is either $\Z$, an order in an imaginary
quadratic number field, or an order in a quaternion algebra over $\Q$
ramified exactly at $\infty$ and $p$.  The quaternionic case occurs if
and only if $E$ is supersingular, and the imaginary quadratic case
occurs if and only if $j(E)$ is in $\Fpbar$ and $E$ is not
supersingular (\cite{SilvermanAEC}*{V.3.1 and Exer.~V.5.8}).

In particular, if $E$ is non-isotrivial, then $\en_{\overline
  K}(E)=\en_K(E)=\Z$.

\section{The Mordell-Weil-Lang-N\'eron theorem}
We write $E(K)$ for the group of $K$-rational points of $E$ and we
call $E(K)$ the Mordell-Weil group of $E$ over $K$.  Lang and N\'eron
(independently) generalized the classical Mordell-Weil theorem to the
function field context:

\begin{thm} 
  Assume that $K=\Fq(\C)$ is the function field of a curve over a
  finite field and let $E$ be an elliptic curve over $K$.  Then $E(K)$
  is a finitely generated abelian group.
\end{thm}

(The theorems of Lang and N\'eron apply much more generally to any
abelian variety $A$ over a field $K$ that is finitely generated over
its ``constant field'' $k$, but one has to take care of the ``constant
part'' of $A$.  See \cite{UlmerCRM} for details.)

We will not give a detailed proof of the MWLN theorem here, but will
mention two strategies.  One is to follow the method of proof of the
Mordell-Weil (MW) theorem over a number field.  Choose a prime number
$\ell\neq p$.  By an argument very similar to that in
\cite{SilvermanAEC}*{Ch.~VIII} one can show that $E(K)/\ell E(K)$ is
finite (the ``weak Mordell-Weil theorem'') by embedding it in a Selmer
group and showing that the Selmer group is finite by using the two
fundamental finiteness results of algebraic number theory (finiteness
of the class group and finite generation of the unit group) applied to
Dedekind domains in $K$.  One can then introduce a theory of heights
exactly as in \cite{SilvermanAEC} and show that the MW theorem follows
from the weak MW theorem and finiteness properties of heights.  See
the original paper of Lang and N\'eron \cite{LangNeron59} for the full
details.  A complete treatment in modern language has been given by
Conrad \cite{Conrad06}.

One interesting twist in the function field setting comes if one takes
$\ell=p$ above.  It is still true that the Selmer group for $p$ is
finite, but one needs to use the local restrictions at all places; the
maximal abelian extension of exponent $p$ unramified outside a finite
but non-empty set of places is not finite and so one needs some
control on ramification at every place.  See \cite{Ulmer91} for a
detailed account of $p$-descent in characteristic $p$.

A second strategy of proof, about which we will say more in Lecture~3,
involves relating the Mordell-Weil group of $E$ to the N\'eron-Severi
group of a closely related surface $\EE$.  In fact, finite generation
of the N\'eron-Severi group (known as the ``theorem of the base'') is
equivalent to the Lang-N\'eron theorem.  A direct proof of the theorem
of the base was given by Kleiman in \cite{SGA6}*{XIII}.  See also
\cite{MilneEC}*{V.3.25}.

\section{The constant case}\label{s:constant}
It is worth pausing in the general development to look at the case of
a constant curve $E$.  Recall that $K$ is the function field $k(\CC)$
of the curve $\CC$ over $k=\Fq$.  Suppose $E_0$ is an elliptic curve
over $k$ and let $E=E_0\times_kK$.

\begin{prop}
We have a canonical isomorphism 
$$E(K)\cong\mor_k(\CC,E_0)$$
where $\mor_k$ denotes morphisms of varieties over $k$
\textup{(}=morphisms of $k$-schemes\textup{)}.  Under this
isomorphism, $E(K)_{tor}$ corresponds to the subset of constant
morphisms.
\end{prop}

\begin{proof}
  By definition, $E(K)$ is the set of $K$-morphisms 
$$\spec K\to  E=E_0\times_kK.$$  
By the universal property of the fiber product, these are in bijection
with $k$-morphisms $\spec K\to E_0$.  Since $\CC$ is a smooth curve,
any $k$-morphism $\spec K\to E_0$ extends uniquely to a $k$-morphism
$\CC\to E_0$.  This establishes a map $E(K)\to\mor_k(\CC,E_0)$.  If
$\eta:\spec K\to\CC$ denotes the canonical inclusion, composition with
$\eta$ ($\phi\mapsto\phi\circ\eta$) induces a map $\mor_k(\CC,E_0)\to
E(K)$ inverse to the map above.  This establishes the desired
bijection and this bijection is obviously compatible with the group
structures.

Since $k$ is finite, it is clear that a constant morphism goes over to
a torsion point.  Conversely, if $P\in E(K)$ is torsion, say of order
$n$, then the image of the corresponding $\phi:\CC\to E_0$ must lie in
the set of $n$-torsion points of $E_0$, a discrete set, and this
implies that $\phi$ is constant.
\end{proof}

For example, if $K$ is rational (i.e., $\CC=\P^1$ so that $K=k(t)$),
then $E(K)=E_0(k)$.

\begin{cor}
Let $J_\CC$ be the Jacobian of $\CC$.  We have canonical isomorphisms
$$E(K)/E(K)_{tor}\cong\Hom_{k-av}(J_\CC,E_0)\cong\Hom_{k-av}(E_0,J_\CC).$$
\end{cor}

\begin{proof}
  The Albanese property of the Jacobian of $\CC$ (Subsection
  \ref{ss:pic-alb} of Lecture~0) gives a surjective homomorphism
$$\mor_{k}(\CC,E_0)\to\Hom_{k-av}(J_\CC, E_0).$$
This homomorphism sends non-constant (and therefore surjective)
morphisms to non-constant (surjective) homomorphisms, so its kernel
consists exactly of the constant morphisms.  The second isomorphism in
the statement of the corollary follows from the fact that Jacobians are
self-dual.
\end{proof}

By Poincar\'e complete reducibility \cite{Milne86av}*{12.1}, $J_\CC$ is
isogenous to a product of simple abelian varieties.  Suppose $J_\CC$
is isogenous to $E_0^m\times A$ and $A$ admits no non-zero morphisms
to $E_0$.  We say that ``$E_0$ appears in $J_\CC$ with multiplicity
$m$.''  Then it is clear from the corollary that
$E(K)/E(K)_{tor}\cong\en_k(E_0)^m$ and so the rank of $E(K)$ is $m$,
$2m$, or $4m$.

Tate and Shafarevich used these ideas to exhibit isotrivial elliptic
curves over $F=\F_p(t)$ of arbitrarily large rank.  Indeed, using
Tate's theorem on isogenies of abelian varieties over finite fields
(reviewed in Section~\ref{s:Tate-thm} of Lecture~0) and a calculation
of zeta functions in terms of Gauss sums, they were able to produce a
hyperelliptic curve $\CC$ over $\Fp$ whose Jacobian is isogenous to
$E_0^m\times A$ where $E_0$ is a supersingular elliptic curve and the
multiplicity $m$ is as large as desired.  If $K=\Fp(\CC)$, $E$ is the
constant curve $E=E_0\times F$, and $E'$ is the twist of $E$ by the
quadratic extension $K/F$, then $\rk E'(F)=\rk E(K)$ and so $E'(F)$
has large rank by the analysis above.  See the original article
\cite{TateShafarevich67} for more details and a series of articles by
Elkies (starting with \cite{Elkies94}) for a beautiful application to
the construction of lattices with high packing densities.

\section{Torsion}
An immediate corollary of the MWLN theorem is that $E(K)_{tor}$ is
finite.  In fact, $E(K)_{tor}$ is isomorphic to a group of the form
$$\Z/m\Z\times\Z/n\Z$$
where $m$ divides $n$ and $p$ does not divide $m$.  (See for example
\cite{SilvermanAEC}*{Ch.~3}.)  One can also see using the theory of
modular curves that every such group appears for a suitable $K$ and
$E$.

In another direction, one can give uniform bounds on torsion that
depend only on crude invariants of the field $K$.

Indeed, in the constant case, $E(K)_{tor}\cong E_0(\Fq)$ which has order
bounded by $(q^{1/2}+1)^2$.  In the isotrivial case, there is a finite
extension $K'$ with the same field of constants $k=\Fq$ over which $E$
becomes constant.  Thus $E(K)_{tor}\subset E(L)_{tor}$ again has
cardinality bounded by $(q^{1/2}+1)^2$.

We now turn to the non-isotrivial case.

\begin{prop}
  Assume that $E$ is non-isotrivial and let $g_\CC$ be the genus of
  $\CC$.  Then there is a finite \textup{(}and effectively
  calculable\textup{)} list of groups---depending only on $g_\CC$ and
  $p$---such that for any non-isotrivial elliptic curve $E$ over $K$,
  $E(K)_{tor}$ appears on the list.
\end{prop}

\begin{proof}
  (Sketch) First consider the prime-to-$p$ torsion subgroup of $E(K)$.
  It has the form $G=\Z/m\Z\times\Z/n\Z$ where $m|n$ and $p\nodiv m$.
  There is a modular curve $X(m,n)$, irreducible and defined over
  $\Fp(\mu_m)$, that is a coarse moduli space for elliptic curves
  with subgroups isomorphic to $G$.  We get a morphism $\CC\to X(m,n)$
  which is non-constant (because $E$ is non-isotrivial) and therefore
  surjective.  The Riemann-Hurwitz formula then implies that $g_\CC\ge
  g_{X(m,n)}$.  But the genus of $X(m,n)$ goes to infinity with $n$.
  Indeed, $g_{X(m,n)}\ge g_{X(1,n)}$ and standard genus formulae
  (\cite{MiyakeMF}*{4.2}) together with crude estimation show that the
  latter is bounded below by 
$$1+\frac{n^2}{24\zeta(2)}-\frac{n\log_2n}{4}.$$  
This shows that for a
  fixed value of $g_\CC$, only finitely many groups $G$ as above can
  appear as $E(K)_{tor}$.

  The argument for $p$-torsion is similar, except that ones uses the
  Igusa curves $Ig(p^n)$ (cf.~\cite{KatzMazurAM}*{Ch.~12}).  If $E(K)$
  has a point of order $p^n$, we get a non-constant morphism $\CC\to
  Ig(p^n)$ and the genus of $Ig(p^n)$ is asymptotic to $p^{2n}/48$
  \cite{Igusa68}*{p.~96}.
\end{proof}

This proposition seems to have been rediscovered repeatedly over the years.
The first reference I know of is \cite{Levin68}.  

Since the genus of a function field is an analogue of the discriminant
(more precisely $q^{2g-2}$ is an analogue of the absolute value of the
discriminant of a number field), the proposition is an analogue of
bounding $E(K)_{tor}$ in terms of the discriminant of a number field
$K$.  One could ask for a strengthening where torsion is bounded by
``gonality'', i.e., by the smallest degree of a non-constant map
$\CC\to\P^1$.  This would be an analogue of bounding $E(K)_{tor}$ in
terms of the degree of a number field $K$, as in the theorems of
Mazur, Kamienny, and Merel \cite{Merel96}.  This is indeed possible and
can be proven by mimicking the proof of the proposition, replacing
bounds on the genus of the modular curve with bounds on its gonality.
See \cite{Poonen07} for the best results currently
known on gonality of modular curves.

\begin{exer}
  Compute the optimal list mentioned in the proposition for $g=0$.
  (This is rather involved.)  
  Note that the optimal list in fact depends on $p$.  Indeed,
  $\Z/11\Z$ is on the list if and only if $p=11$.
\end{exer}

One can be very explicit about $p$-torsion:

\begin{prop}
Suppose that $E$ is a non-isotrivial elliptic curve over $K$.
Then $E(K)$ has a point of order $p$ if and only if $j(E)\in K^{p}$ and
$A(E,\omega)$ is a $(p-1)$st power in $K^\times$.
\end{prop}

Note that whether $A(E,\omega)$ is a $(p-1)$st power is independent of
the choice of the differential $\omega$. 

\begin{proof}
Let $E\xrightarrow{\Fr}E^{(p)}\xrightarrow{V}E$ be the standard
factorization of multiplication by $p$ into Frobenius and Verschiebung.
Recall (e.g., \cite{Ulmer91}*{2.1}) that $A(E,\omega)$ is a $(p-1)$st
power in $K$ if and only if $\ker \Fr\cong\mu_p$ if and only if $\ker
V\cong\Z/p\Z$ if and only if there is a non-trivial $p$-torsion point
in $E^{(p)}(K)$.

Now suppose that $P\in E(K)$ is a non-trivial $p$-torsion point.  Then
$\Fr(P)$ is a non-trivial $p$-torsion point in $E^{(p)}(K)$ and so
$A(E,\omega)$ is a $(p-1)$st power in $K$.  Let $E'$ be the quotient
of $E$ by the cyclic subgroup generated by $P$: $E'=E/\<P\>$.  Since
$\<P\>$ is in the kernel of multiplication by $p$, we have a
factorization of multiplication by $p$:
$$[p]:E\to E'\to E.$$
Since $E\to E'$ is \'etale of degree $p$ and $[p]$ is inseparable of
degree $p^2$, we have that $E'\to E$ is purely inseparable of degree
$p$.  But an elliptic curve in characteristic $p$ has a unique
inseparable isogeny of degree $p$ (namely the quotient by the unique
connected subgroup of order $p$, the kernel of Frobenius) so we have
an identification $E=E^{\prime(p)}$.  By \ref{l:(p)}, $j(E)\in K^p$.

Conversely, suppose $A(E,\omega)$ is a $(p-1)$st power and $j(E)\in
K^p$.  Let $E'$ be the elliptic curve such that $E^{\prime(p)}\cong
E$.  Given a differential $\omega$ on $E$, there is a differential
$\omega'$ on $E'$ such that $A(E,\omega)=A(E',\omega')^p$ (as can be
seen for example by using Weierstrass equations).  It follows that
$A(E',\omega')$ is also a $(p-1)$st power in $K$.  Thus we have a
non-trivial point of order $p$ in $E^{\prime(p)}(K)=E(K)$.
\end{proof}

Part of the proposition generalizes trivially by iteration: if $E(K)$
has a point of order $p^n$, then $j(E)\in K^{p^n}$.  A full
characterization of $p^n$ torsion seems harder---the condition that
$A(E,\omega)$ be a $(p-1)$st power is closely related to the equations
defining the Igusa curve $Ig(p)$ (\cite{KatzMazurAM}*{12.8}), but we
do not have such explicit equations for $Ig(p^n)$ when $n>1$.

\section{Local invariants}\label{s:local-invs}
Let $E$ be an elliptic curve over $K$ and let $v$ be a place of $K$.
A model (\ref{eq:cubica}) for $E$ with coefficients in the valuation
ring $\O_{(v)}$ is said to be {\it integral at $v$\/}.  The valuation
of the discriminant $\Delta$ of an integral model is a non-negative
integer and so there are models where this valuation takes its minimum
value.  Such models are {\it minimal integral models at $v$\/}.

Choose a model for $E$ that is minimal integral at $v$:
$$y^2+a_1xy+a_3y=x^3+a_2x^2+a_4x+a_6.$$
Let $\overline a_i\in\kappa(v)$ be the reductions of the coefficients
and let $E_v$ be the plane cubic
\begin{equation}\label{eq:reduction}
y^2+\overline a_1xy+\overline a_3y
=x^3+\overline a_2x^2+\overline a_4x+\overline a_6
\end{equation}
over the residue field $\kappa_v$.
It is not hard to check using Weierstrass equations that the
isomorphism type of the reduced cubic (\ref{eq:reduction}) is independent
of the choice of minimal model.

If the discriminant of a minimal integral model at $v$ has valuation
zero, i.e., is a unit at $v$, then the reduced equation defines an
elliptic curve over $\kappa_v$.  If the minimal valuation is positive,
then the reduced curve is singular.  We distinguish three cases
according to the geometry of the reduced curve.

\begin{defn}\leavevmode
\begin{enumerate}
\item If $E_v$ is a smooth cubic, we say $E$
  has {\it good reduction at $v$}.
\item If $E_v$ is a nodal cubic, we say $E$ has {\it multiplicative
    reduction at $v$}.  If the tangent lines at the node are rational
  over $\kappa(v)$ we say the reduction is {\it split multiplicative\/}
  and if they are rational only over a quadratic extension, we say the
  reduction is {\it non-split multiplicative\/}.
\item If $E_v$ is a cuspidal cubic, we say $E$ has {\it additive
    reduction\/}.
\end{enumerate}
\end{defn}

Define an integer $a_v$ as follows:
\begin{equation}\label{eq:a_v}
a_v=\begin{cases}
q_v+1-\#E_v(\kappa_v)&\text{if $E$ has good reduction at $v$}\\
1&\text{if $E$ has split multiplicative reduction at $v$}\\
-1&\text{if $E$ has non-split multiplicative reduction at $v$}\\
0&\text{if $E$ has additive reduction at $v$}
\end{cases}
\end{equation}

\begin{exer}
To make this definition less {\it ad hoc\/}, note that in the good
reduction case, the numerator of the $\zeta$-function of the reduced
curve is $1-a_vq_v^{-s}+q_v^{1-2s}$.  Show that in the bad reduction
cases, the $\zeta$-function of the reduced curve is
$$\frac{1-a_vq_v^{-s}}{(1-q_v^{-s})(1-q_v^{1-s})}.$$
\end{exer}

In the good reduction case, the results about zeta
functions and \'etale cohomology reviewed in Lecture~0,
Sections~\ref{s:zetas} and \ref{s:cohomology} imply the ``Hasse
bound'':  $|a_v|\le2\sqrt{q_v}$.

There are two more refined invariants in the bad reduction cases: the
N\'eron model and the conductor.  The local exponent of the conductor
at $v$, denoted $n_v$ is defined as
\begin{equation}
n_v=\begin{cases}
0&\text{if $E$ has good reduction at $v$}\\
1&\text{if $E$ has multiplicative reduction at $v$}\\
2+\delta_v&\text{if $E$ has additive reduction at $v$}
\end{cases}
\end{equation}
Here $\delta_v$ is a non-negative integer that is 0 when $p>3$ and is
$\ge0$ when $p=2$ or 3.  We refer to \cite{Tate75} for a definition and
an algorithm to compute $\delta_v$.

The (global) conductor of $E$ is defined to be the divisor $\n=\sum_v
n_v[v]$.  Its degree is $\deg\n=\sum_vn_v\deg v$.

The N\'eron model will be discussed in Lecture~3 below.

\begin{exer}
  Mimic \cite{SilvermanAEC}*{Ch.~VII} to define a filtration on the
  points of $E$ over a completion $K_v$ of $K$.  Show that the
  prime-to-$p$ part of $E(K)_{tor}$ maps injectively into
  $E(K_v)/E(K_v)_1$.  Relate $E(K_v)/E(K_v)_1$ to the special fiber of
  the N\'eron model of $E$ at $v$.  As in the classical case, this
  gives an excellent way to bound the prime-to-$p$ part of
  $E(K)_{tor}$.  
\end{exer}

\section{The $L$-function}
We define the $L$-function of $E/K$ as an Euler product:
\begin{equation}\label{eq:Ldef}
L(E,T)=\prod_{\text{good }v}\left(1-a_vT^{\deg v}+q_vT^{2\deg
    v}\right)^{-1}
\prod_{\text{bad }v}\left(1-a_vT^{\deg v}\right)^{-1}
\end{equation}
and
$$L(E,s)=L(E,q^{-s}).$$
(Here $T$ is a formal indeterminant and $s$ is a complex number.
Unfortunately, there is no standard reasonable parallel of the
notations $Z$ and $\zeta$ to distinguish the function of $T$ and the
function of $s$.)  Because of the Hasse bound on the size of $a_v$,
the product converges absolutely in the region $\RP s>3/2$, and as we
will see below, it has a meromorphic continuation to all $s$.

When $E$ is constant it is elementary to calculate $L(E,s)$ in terms
of the zeta-functions of $E_0$ and $\CC$.  

\begin{exer}\label{exer:L-const}
Suppose that $E=E_0\times_kK$.  Write the $\zeta$-functions of $E_0$
and $\CC$ as rational functions:
$$\zeta(E_0,s)=\frac{\prod_{i=1}^{2}(1-\alpha_iq^{-s})}{(1-q^{-s})(1-q^{1-s})}$$
and
$$\zeta(\CC,s)=\frac{\prod_{j=1}^{2g_\CC}(1-\beta_jq^{-s})}{(1-q^{-s})(1-q^{1-s})}.$$
Prove that
$$L(E,s)=\frac{\prod_{i,j}(1-\alpha_i\beta_jq^{-s})}
{\prod_{i=1}^{2}(1-\alpha_iq^{-s})\prod_{i=1}^{2}(1-\alpha_iq^{1-s})}.$$
\end{exer}

Thus $L(E,s)$ is a rational function in $q^{-s}$ of degree $4g_\CC-4$,
it extends to a meromorphic function of $s$, and it satisfies a
functional equation for $s\leftrightarrow 2-s$.  Its poles lie on the
lines $\RP s=1/2$ and $\RP s=3/2$ and its zeroes lie on the line $\RP
s=1$.

Although the proofs are much less elementary, these facts extend to
the non-constant case as well:

\begin{thm}
  Suppose the $E$ is a non-constant elliptic curve over $K$.  Let $\n$
  be the conductor of $E$.  Then $L(E,s)$ is a polynomial in $q^{-s}$
  of degree $N=4g_\CC-4+\deg\n$, it satisfies a functional equation
  for $s\leftrightarrow2-s$, and its zeroes lie on the line $\RP s=1$.
  More precisely,
$$L(E,s)=\prod_{i=1}^{N}(1-\alpha_iq^{-s})$$
where each $\alpha_i$ is an algebraic integer of absolute value $q$ in
every complex embedding.  The collection of $\alpha_i$ \textup{(}with
multiplicities\textup{)} is invariant under $\alpha_i\mapsto q^2/\alpha_i$.
\end{thm}

The theorem is a combination of results of Grothendieck, Deligne, and
others.  We will sketch a proof of it in Lecture~4.

Note that in all cases $L(E,s)$ is holomorphic at $s=1$.  In the
non-constant case, its order of vanishing at $s=1$ is bounded above by
$N$ and it equals $N$ if and only if $L(E,s)=(1-q^{1-s})^N$.

\section{The basic BSD conjecture}
This remarkable conjecture connects the analytic behavior of the
function $L(E,s)$, constructed from local data, to the Mordell-Weil
group, a global invariant.

\begin{conj}[Birch and Swinnerton-Dyer]
$$\rk E(K) = \ord_{s=1}L(E,s)$$
\end{conj}

The original conjecture was stated only for elliptic curves over $\Q$
\cite{BSD65} but it is easily seen to make sense for abelian varieties
over global fields.  There is very strong evidence in favor of it,
especially for elliptic curves over $\Q$ and abelian varieties over
function fields. See \cite{GrossPCMI}*{Lecture~3, \S4} for a summary
of the best theoretical evidence in the number field case.  We will
discuss what is known for elliptic curves in the function field case
later in this course.  See Section~\ref{s:results} for statements of
the main results and \cite{UlmerCRM} for a discussion of the case of
higher dimensional abelian varieties over function fields.

\section{The Tate-Shafarevich group}
We define the Tate-Shafarevich group of $E$ over $K$ as
$$\sha(E/K)=\ker\left(H^1(K,E)\to\prod_v H^1(K_v,E)\right).$$
Here the cohomology groups can be taken to be Galois cohomology
groups:
$$H^1(K,E)=H^1(G_K,E(K^{sep}))$$
and similarly for $H^1(K_v,E)$; or they can be taken as \'etale or flat
cohomology groups of $\spec K$ with coefficients in the sheaf
associated to $E$.  The flat cohomology definition is essential for
proving finer results on $p$-torsion in $\sha(E/K)$.

\begin{exer}
Show that the group $H^1(K,E)$ (and therefore $\sha(E/K)$) is
torsion.  Hint:  Show that given a class $c\in H^1(K,E)$, there is a
finite Galois extension $L/K$ such that $c$ vanishes in $H^1(L,E)$. 
\end{exer}

The refined BSD conjecture relates the leading coefficient of $L(E,s)$
at $s=1$ to invariants of $E$ including heights, Tamagawa numbers, and
the order of $\sha(E/K)$.  In particular, the conjecture that
$\sha(E/K)$ is finite is included in the refined BSD conjecture.  We
will not discuss that conjecture in these lectures, so we refer to
\cite{GrossPCMI} and \cite{UlmerCRM} for more details.

\section{Statements of the main results}\label{s:results}
Much is known about the BSD conjecture over function fields.  We start
with general results.

\begin{thm}\label{thm:BSD1}
Let $E$ be an elliptic curve over a function field $K$.  Then we have:
\begin{enumerate}
\item\label{item:inequality}
$\rk E(K)\le\ord_{s=1}L(E,s)$.
\item\label{item:sha}
The following are equivalent:
\begin{itemize}
\item $\rk E(K)=\ord_{s=1}L(E,s)$
\item $\sha(E/K)$ is finite
\item for any one prime number $\ell$ \textup{(}$\ell=p$ is
  allowed\textup{)}, the $\ell$-primary part $\sha(E/K)_{\ell^\infty}$
  is finite.
\end{itemize}
\item\label{item:descent}
If $K'/K$ is a finite extension and if the BSD conjecture holds
  for $E$ over $K'$, then it holds for $E$ over $K$.
\end{enumerate}
\end{thm}

The theorem was proven by Tate \cite{Tate66b} and Milne \cite{Milne75}
and we will sketch a proof in Lecture~3.  When the equivalent
conditions of Item~\ref{item:sha} hold, it turns out that the refined
BSD conjecture automatically follows.  (This is also due to Tate and
Milne and will be discussed in detail in \cite{UlmerCRM}.)

We now state several special cases where the conjecture is known to be
true.  As will be seen in the sequel, they all ultimately reduce
either to Tate's theorem on isogenies of abelian varieties over finite
fields (Theorem~\ref{thm:TateIsogThm} of Lecture~0) or to a theorem of
Artin and Swinnerton-Dyer on $K3$ surfaces
\cite{ArtinSwinnertonDyer73}.

\begin{thm}\label{thm:BSD2}
If $E$ is an isotrivial elliptic curve over
a function field $K$, then the BSD conjecture holds for $E$.  
\end{thm}

Recall that a constant curve is also isotrivial.

To state the next result, we make an {\it ad hoc\/} definition.  If
$E$ is an elliptic curve over $K=\Fq(t)$ we define the {\it height\/}
$h$ of $E$ to be the smallest non-negative integer such that $E$ can
be defined by a Weierstrass equation~(\ref{eq:cubic}) where the $a_i$
are all polynomials and $\deg(a_i)\le hi$.  For example, the curves
$E_1$ and $E_2$ in Subsection~\ref{ss:examples} have height
$h=0$ and the other curves $E_3,\dots E_9$ there all have height
$h=1$.  See Section~\ref{s:height} of Lecture~3 below for a more
general definition.

\begin{thm}\label{thm:BSD-low-height}
Suppose that $K=k(t)$ and that $E$ is an elliptic curve over $K$ of
height $h\le2$.  Then the BSD conjecture holds for $E$.
\end{thm}

Note that this case overlaps the preceding one since an elliptic curve
over $k(t)$ is constant if and only if its height is zero
(cf.~Proposition~\ref{prop:h=0} in Lecture~3).

The following case is essentially due to Shioda \cite{Shioda86}.  To
state it, consider a polynomial $f$ in three variables with coefficients
in $k$ which is the sum of exactly 4 non-zero monomials, say
$$f=\sum_{i=1}^4c_i\prod_{j=1}^3x_j^{e_{ij}}$$
where the $c_i\in k$ are non-zero.  Set $e_{i4}=1-\sum_{j=1}^3e_{ij}$
and let $A$ be the $4\times4$ integer matrix $A=(e_{ij})$.  If $\det
A\neq0\pmod p$, we say that $f$ {\it satisfies Shioda's condition\/}.
Note that the condition is independent of the order of the variables
$x_j$.

\begin{thm}\label{thm:4-monos}
  Suppose that $K=k(t)$ and that $E$ is an elliptic curve over $K$.
  Suppose that $E$ is birational to a plane curve
  $V(f)\subset\A^2_K$ where $f$ is a polynomial in $k[t,x,y]\subset
  K[x,y]$ which is the sum of exactly 4 non-zero monomials and which
  satisfies Shioda's condition.  Then the BSD conjecture holds for
  $E$.
\end{thm} 

For example, the theorem applies to the curves $E_4$, $E_7$, $E_8$,
and $E_9$ of Subsection~\ref{ss:examples} over $K=\Fq(t)$ for any
prime power $q$.  It applies more generally to these curves when $t$
is replaced by $t^d$ for any $d$ prime to $p$.  Note that when $d$ is
large, the height of the curve is also large, and so we get cases of
BSD not covered by Theorem~\ref{thm:BSD-low-height}.

Finally we state another more recent and ultimately much more flexible
special case due to Lisa Berger \cite{Berger08}.

\begin{thm}\label{thm:berger}
Suppose that $K=k(t)$ and that $E$ is an elliptic curve over $K$.
Suppose that $E$ is birational to a plane curve of the form
$$f(x)=t^dg(y)$$
where $f$ and $g$ are rational functions of one variable and $d$ is
prime to $p$.  Then the BSD conjecture holds for $E$.
\end{thm}

Here one should clear denominators to interpret the equation $f=t^dg$
(or work in a Zariski open subset of the plane).  For example, if
$f(x)=x(x-1)$ and $g(y)=y^2/(1-y)$ then we have the plane curve over
$K=k(t)$ defined by 
$$x(x-1)(1-y)=t^dy^2$$
which turns out to be birational to 
$$y^2+xy+t^dy=x^3+t^dx^2.$$

\section{The rest of the course}
The remainder of these lectures will be devoted to sketching the
proofs of most of the main results and applying them to construct
elliptic curves of large rank over function fields.

More precisely, in Lecture~2 we will review facts about surfaces and
the Tate conjecture on divisors.  This is a close relative of the BSD
conjecture.

In Lecture~3 we will explain the relationship between the BSD and Tate
conjectures and use it to prove the part of Theorem~\ref{thm:BSD1}
related to $\ell\neq p$ as well as most of the other theorems stated
in the previous section.

In Lecture~4 we will recall a general result on vanishing of
$L$-functions in towers and combine it with the results above to
obtain many elliptic curves of arbitrarily large rank.

In Lecture~5, we will give other applications of these ideas to ranks
of elliptic curves and explicit points.


\lecture{Surfaces and the Tate conjecture}
\section{Motivation}
Consider an elliptic curve $E/K$ and suppose that
$K=k(t)$ and that we choose an equation for
$E$ as in Lecture~1, equation~(\ref{eq:cubica}) where the $a_i$ are in
$k[t]$.  Then (\ref{eq:cubica}), viewed in $K[x,y]$, defines an affine
open subset of an elliptic curve $E$.  But if we view it as an
equation in $k[t,x,y]$, it defines an affine surface with a projection
to the affine $t$ line.  The generic fiber of this projection is the
affine curve just mentioned.

With a little more work (discussed in the next lecture), for any $E$
over $K=k(\CC)$ we can define a smooth projective surface $\EE$ over
$k$ with a morphism $\pi:\EE\to\CC$ whose generic fiber is $E$.
Obviously there will be close connection between the arithmetic of
$\EE$ and that of $E$.  Although $\EE$ has higher dimension than $E$,
it is defined over the finite field $k$ and as a result we have better
control over its arithmetic.  Pursuing this line of inquiry leads to
the main theorems stated at the end of the previous section.

In this lecture, we discuss the relevant facts and conjectures about
surfaces over finite fields.  In the next lecture we will look
carefully at the connections between $\EE$ and $E$ and deduce the main
classical theorems.  

There are many excellent references for the general theory of
surfaces, including \cite{BeauvilleCAS},
\cite{BarthHulekPetersVandeVenCCS}, and \cite{BadescuAS}.  We
generally refer to \cite{BadescuAS} below since it works throughout
over a field of arbitrary characteristic.

\section{Surfaces}
Let $k=\Fq$ be a finite field of characteristic $p$.  As always, by a
surface over $k$ we mean a purely 2-dimensional, separated, reduced
scheme of finite type over $k$.  Such a scheme is automatically
quasi-projective and is projective if and only if it is complete
\cite{BadescuAS}*{1.28}.  Since $k$ is perfect, a surface $\XX$ is a
regular scheme if and only if $\XX\to\spec k$ is a smooth morphism
(e.g., \cite{LiuAGAC}*{4.3.3, Exer. 3.24}).  We sloppily say that
``$\XX$ is smooth'' if these conditions hold.  Resolution of
singularities is known for surfaces: For any surface $\XX$, there is a
proper birational morphism $\tilde\XX\to\XX$ with $\tilde\XX$ smooth.
(We may even take this morphism to be a composition of normalizations
and blow ups at closed points \cite{Lipman78}.  See also
\cite{Artin86} for a nice exposition.)  Therefore, every surface is
birational to a smooth projective surface.  In the cases of interest
to us, this can be made very explicit in an elementary manner.

Throughout we assume that $\XX$ is a smooth, projective, absolutely
irreducible surface over $k$ and we assume that $\XX(k)$ is non-empty,
i.e., $\XX$ has a $k$-rational point.

\section{Divisors and the N\'eron-Severi group}
We give a lightning review of divisors and equivalence relations on
divisors.  See, for example, \cite{HartshorneAG}*{V.1} for more
details.

\numberwithin{equation}{subsection}
\subsection{Divisor classes}
A (Weil)  {\it divisor\/} is a finite formal $\Z$-linear combination of
reduced, closed, codimension 1 subvarieties of $\XX$:
$$D=\sum a_ZZ.$$
In other words, the set of divisors is the free abelian group on the 
reduced, closed, codimension 1 subvarieties on $\XX$.

If $Z$ is a reduced, closed subvariety of $\XX$ of codimension 1,
there is an associated valuation
$$\ord_Z:k(\XX)^\times\to\ZZ$$
that sends a rational function to its order of zero or pole along
$Z$.

A rational function $f$ on $\XX$ has a divisor:
$$\divr(f)=\sum_Z\ord_Z(f)Z.$$

A divisor $D$ is said to be {\it linearly equivalent to zero\/} if
there is a rational function $f$ such that $\divr(f)=D$.  Two divisors
$D$ and $D'$ are linearly equivalent if their difference $D-D'$ is
linearly equivalent to zero.

The group of divisors modulo those linear equivalent to zero is the
{\it divisor class group\/} $DivCl(\XX)$.  It is a fundamental
invariant of $\XX$.

\subsection{The Picard group}
Let $\Pic(\XX)$ be the Picard group of $\XX$, i.e., the group of
isomorphism classes of invertible sheaves on $\XX$ with group law
given by the tensor product.  There is a cohomological calculation of
$\Pic(\XX)$:
$$\Pic(\XX)\cong H^1(\XX,\O_\XX^\times).$$

The map sending a divisor $D$ to the invertible sheaf $\O_\XX(D)$
induces an isomorphism $DivCl(\XX)\isoto\Pic(\XX)$.

\subsection{The N\'eron-Severi group}
As usual, we write $\XXbar$ for $\XX\times_k\kbar$.  We first
introduce the notion of algebraic equivalence for divisors on
$\XXbar$.  Intuitively, two divisors $D$ and $D'$ are algebraically
equivalent if they lie in a family parameterized by a connected
variety (which we may take to be a smooth curve).  More precisely, if
$T$ is a smooth curve over $\kbar$ and $\DD\subset \XX\times_\kbar T$
is a divisor that is flat over $T$, then we get a family of divisors
on $\XX$ parameterized by $T$: $t\in T$ corresponds to
$\XX\times\{t\}\cap\DD$.  Two divisors $D_1$ and $D_2$ on $\XX$ are
algebraically equivalent if they lie in such a family, i.e., if there
is a curve $T$ and a divisor $\DD$ as above and two points $t_1$ and
$t_2\in T(\kbar)$ such that $D_i=\XX\times_\kbar\{t_i\}\cap\DD$.
({\it A priori\/}, to ensure transitivity of this relation we should
use chains of equivalences (see \cite{HartshorneAG}*{Exer.~V.1.7}) but
see \cite{FultonIT}*{10.3.2} for an argument that shows the definition
works as is.)  Note that linear equivalence is algebraic equivalence
where $T$ is restricted to be $\P^1$
(\cite{HartshorneAG}*{Exer.~V.1.7}) and so algebraic equivalence is
weaker than linear equivalence.

The group of divisors on $\XXbar$ modulo those algebraically
equivalent to zero is the {\it N\'eron-Severi\/} group $\NS(\XXbar)$.
A classical (and difficult) theorem, the ``theorem of the base,'' says
that $\NS(\XXbar)$ is finitely generated.  See \cite{LangNeron59} and
\cite{SGA6}*{XIII.5.1} for proofs and Lecture~3 below for more
discussion. See also \cite{Conrad06} for a modern discussion of the
results in \cite{LangNeron59}.

Since linear equivalence is weaker than algebraic equivalence,
$\NS(\XXbar)$ is a quotient of $\Pic(\XXbar)$.

We define $\NS(\XX)$ to be the image of $\divr(\XX)$ in $\NS(\XXbar)$
or equivalently 
the image of $\Pic(\XX)$ in $\NS(\XXbar)$.  Thus
$\NS(\XX)$ is again a finitely generated abelian group.  As we will
see, it is of arithmetical nature.

\begin{exer} Let $G_k=\gal(\kbar/k)$.  Show that $\NS(\XX)$ is the group
  of invariants $\NS(\XXbar)^{G_k}$.  You will need to use that $k$ is a
  finite field.
\end{exer}

\numberwithin{equation}{section}

\section{The Picard scheme}
We define $\Pic^0(\XX)$ as the kernel of the surjection
$\Pic(\XX)\to\NS(\XX)$.  In order to understand this group better, we
will introduce more structure on the Picard group.  The main fact we
need to know is that the group $\Pic^0(\XX\times\kbar)$ is the set of
points on an abelian variety and is therefore a divisible group.
(I.e., for every class $c\in\Pic^0(\XX\times\kbar)$ and every positive
integer $n$, there is a class $c'$ such that $n\,c'=c$.)  Readers
willing to accept this assertion can skip the rest of this section.

The Picard group $\Pic(\XX)$ is the set of $k$-points of a group
scheme.  More precisely, under our hypotheses on $\XX$ there is a
group scheme called the {\it Picard scheme\/} and denoted
$\ps_{\XX/k}$ which is locally of finite type over $k$ and represents
the relative Picard functor.  This means that if $T\to S=\spec k$ is a
morphism of schemes and $\pi_T:\XX_T:=\XX\times_{\spec k}T\to T$ is
the base change then
$$\ps_{\XX/k}(T)=\frac{\Pic(\XX_T)}{\pi_T^*\Pic(T)}.$$
Here the left hand side is the group of $T$-valued points of
$\ps_{\XX/k}$.  See \cite{Kleiman05} for a thorough and detailed
overview of the Picard scheme, and in particular
\cite{Kleiman05}*{9.4.8} for the proof that there is a scheme
representing the relative Picard functor as above.

We write $\ps^0_{\XX/k}$ for the connected component of $\ps_{\XX/k}$
containing the identity.  Under our hypotheses, $\ps^0_{\XX/k}$ is a
geometrically irreducible projective group scheme over $k$
\cite{Kleiman05}*{9.5.3, 9.5.4}.  It may be non-reduced.  (See
examples in \cite{Igusa55} and \cite{Serre58} and a full analysis of
this phenomenon in \cite{MumfordLCAS}.)  We let
$\Picvar_{\XX/k}=\left(\ps^0_{\XX/k}\right)_{red}$, the {\it Picard
  variety\/} of $\XX$ over $k$, which is an abelian variety over $k$.

If $k'$ is a field extension of $k$, we have
$$\Pic^0(\XX_{k'})=\ps^0_{\XX/k}(k')=\Picvar_{\XX/k}(k')$$
so that $\Pic^0(\XX_{k'})$ is the set of points of an abelian
variety.  

By \cite{Kleiman05}*{9.5.10}, $\ps^0_{\XX/k}(k)=\Pic^0(\XX)$, in other
words, the class of a divisor in $\ps(\XX)$ lies in $\ps^0(\XX)$ if
and only if the divisor is algebraically equivalent to 0.

\section{Intersection numbers and numerical equivalence}

There is an intersection pairing on the N\'eron-Severi group:
$$\NS(\XX)\times\NS(\XX)\to\Z$$
which is bilinear and symmetric.  If $D$ and $D'$ are divisors, we
write $D.D'$ for their intersection pairing.

There are two approaches to defining the pairing.  In the first
approach, one shows that given two divisors, there are divisors in the
same classes in $\NS(\XX)$ (or even the same classes in $\Pic(\XX)$)
that meet transversally.  Then the intersection number is literally
the number of points of intersection.  The work in this approach is to
prove a moving lemma and then show that the resulting pairing is well
defined.  See \cite{HartshorneAG}*{V.1} for the details.  

In the second approach, one uses coherent cohomology.  If $\LL$ is an
invertible sheaf on $\XX$, let
$$\chi(\LL)=\sum_{i=0}^2(-1)^i \dim_k H^i(\XX,\LL)$$
be the coherent Euler characteristic of $\LL$.  Then define
$$D.D'=\chi(\OO_\XX)-\chi(\OO_\XX(-D))-\chi(\OO_\XX(-D'))
+\chi(\OO_\XX(-D-D')).$$ One checks that if $C$ is a smooth
irreducible curve on $\XX$, then $C.D=\deg\OO_\XX(D)|_{C}$ and that if
$C$ and $C'$ are two distinct irreducible curves on $\XX$ meeting
transversally, then $C.C'$ is the sum of local intersection
multiplicities.  See \cite{BeauvilleCAS}*{I.1-7} for details.
(Nowhere is it used in this part of \cite{BeauvilleCAS} that the
ground field is $\C$.)

Two divisors $D$ and $D'$ are said to be {\it numerically
  equivalent\/} if $D.D''=D'.D''$ for all divisors $D''$.  If
$\Num(\XX)$ denotes the group of divisors in $\XX$ up to numerical
equivalence, then we have surjections
$$\Pic(\XX)\onto\NS(\XX)\onto\Num(\XX)$$
and so $\Num(\XX)$ is a finitely generated group.  It is clear from
the definition that $\Num(\XX)$ is torsion-free and so we can insert
$\NS(\XX)/{tor}$ (N\'eron-Severi modulo torsion) into this chain:
$$\Pic(\XX)\onto\NS(\XX)\onto\NS(\XX)/{tor}\onto\Num(\XX).$$

\section{Cycle classes and homological equivalence}\label{s:cycle}
There is a general theory of cycle classes in $\ell$-adic cohomology,
see for example \cite{SGA4.5}*{[Cycle]}.  In the case of divisors,
things are much simpler and we can construct a cycle class map from
the Kummer sequence.

Indeed, consider the short exact sequence of sheaves on $\XXbar$ for
the \'etale topology:
$$0\to\mu_{\ell^n}\to\G_m\stackrel{\ell^n}{\longrightarrow}\G_m\to0.$$
(The sheaves $\mu_{\ell^n}$ and $\G_m$ are perfectly reasonable
sheaves in the Zariski topology on $\XX$, but the arrow in the right
is not surjective in that context.  We need to use the \'etale
topology or a finer one.)  Taking cohomology, we get a homomorphism
$$Pic(\XXbar)/\ell^n=H^1(\XXbar,\G_m)/\ell^n\to
H^2(\XXbar,\mu_{\ell^n}).$$
Since $\Pic^0(\XXbar)$ is a divisible group, we have 
$\NS(\XXbar)/\ell^n=\Pic(\XXbar)/\ell^n$
and so taking an inverse limit gives an injection 
\begin{equation*}
\NS(\XXbar)\tensor\Zl\to H^2(\XXbar,\Zl(1)).
\end{equation*}

Composing with the natural homomorphism $\NS(\XX)\to\NS(\XXbar)$ 
gives our cycle class map
\begin{equation}\label{eq:cycle}
\NS(\XX)\to\NS(\XX)\tensor\Zl\to H^2(\XXbar,\Zl(1)).
\end{equation}

We declare two divisors to be ($\ell$-){\it homologically
  equivalent\/} if their classes in $H^2(\XXbar,\Zl(1))$ are equal.
(We will see below that this notion is independent of $\ell$.)  The
group of divisors modulo homological equivalence will (temporarily) be
denoted $\Ho(\XX)$.  It will turn out to be a finitely generated
free abelian group.

The intersection pairing on $NS(\XX)$ corresponds
under the cycle class map to the cup product on cohomology.  This
means that a divisor that is homologically equivalent to zero is also
numerically equivalent to zero.  Thus we have a chain of surjections:
$$\Pic(\XX)\onto\NS(\XX)\onto\NS(\XX)/{tor}
\onto\Ho(\XX)\onto\Num(\XX).$$

\section{Comparison of equivalence relations on divisors}

A theorem of Matsusaka \cite{Matsusaka57} asserts that the surjection
$$\NS(\XX)/{tor}\onto\Num(\XX)$$
is in fact an isomorphism.  Thus
$$\NS(\XX)/{tor}\cong\Ho(\XX)\cong\Num(\XX)$$
and these groups are finitely generated, free abelian groups.  Since
$\NS(\XX)$ is finitely generated, $\NS(\XX)_{tor}$ is finite.

In all of the examples we will consider, $\NS(\XX)$ is torsion free.
(In fact, for an elliptic surface with a section, the surjection
$\NS(\XX)\to\Num(\XX)$ is always an isomorphism, see
\cite{SchuttShiodaES}*{Theorem~6.5}.)  So to understand $\Pic(\XX)$ we
have only to consider the finitely generated free abelian group
$\NS(\XX)$ and the group $\Pic^0(\XX)$, which is (the set of points
of) an abelian variety.

\begin{exer}
  In the case of a surface $\XX$ over the complex numbers, use the
  cohomology of the exponential sequence
$$0\to\Z\to\OO_\XX\stackrel{\exp}{\longrightarrow}\OO^\times_\XX\to0$$
to analyze the structure of $\Pic(\XX)$.
\end{exer}

\section{Examples}
\numberwithin{equation}{subsection}
\subsection{$\P^2$} 
It is well known (e.g., \cite{HartshorneAG}*{II.6.4}) that two curves on
$\P^2$ are linearly equivalent if and only if they have the same
degree.  It follows that $\Pic(\P^2)=\NS(\P^2)\cong\Z$.

\subsection{$\P^1\times\P^1$}
By \cite{HartshorneAG}*{II.6.6.1}, two curves on
$\P^1\times\P^1$ are linearly equivalent if and only if they have the
same bi-degree.  It follows that
$\Pic(\P^1\times\P^1)=\NS(\P^1\times\P^1)\cong\Z^2$.

\subsection{Abelian varieties}
If $\XX$ is an abelian variety (of any dimension $g$), then $\Pic^0(\XX)$
is the dual abelian variety and $\NS(\XX)$ is a finitely
generated free abelian group of rank between 1 and $4g^2$.  See
\cite{MumfordAV} for details.

\subsection{Products of curves}\label{ss:products}
Suppose that $\CC$ and $\DD$ are smooth projective curves over $k$
with $k$-rational points $x\in\CC$ and $y\in\DD$.  By definition (see
Subsection~\ref{ss:pic-alb} of Lecture~0), the group of divisorial
correspondences between $(\CC,x)$ and $(\DD,y)$ is a subgroup of
$\Pic(\CC\times\DD)$ and it is clear that
\begin{align*}
\Pic(\CC\times\DD)&\cong\Pic(\CC)\times\Pic(\DD)\times
     \DivCorr\left((\CC,x),(\DD,y)\right)\\
&\cong\Pic^0(\CC)\times\Pic^0(\DD)\times\Z^2\times
     \DivCorr\left((\CC,x),(\DD,y)\right).
\end{align*}
Moreover, as we saw in Lecture~0,
$$\DivCorr\left((\CC,x),(\DD,y)\right)\cong\Hom(J_\CC,J_\DD)$$
is a discrete group.  It follows that
\begin{equation}\label{eq:Pic0(CxD)}
\Pic^0(\CC\times\DD)\cong\Pic^0(\CC)\times\Pic^0(\DD)
\end{equation}
and
\begin{equation}\label{eq:NS(CxD)}
\NS(\CC\times\DD)\cong\Z^2\times\Hom(J_\CC,J_\DD).
\end{equation}
This last isomorphism will be important for a new approach to elliptic
curves of high rank over function fields discussed in Lecture~5.

\subsection{Blow ups}\label{ss:blow-ups}
Let $\XX$ be a smooth projective surface over $k$ and let
$\pi:\YY\to\XX$ be the blow up of $\XX$ at a closed point $x\in\XX$ so
that $E=\pi^{-1}(x)$ is a rational curve on $\YY$.  Then we have
canonical isomorphisms
$$\Pic(\YY)\cong\Pic(\XX)\oplus\Z\quad\text{and}\quad
\NS(\YY)\cong\NS(\XX)\oplus\Z$$
where in both groups the factor $\Z$ is generated by the class of $E$.
See \cite{HartshorneAG}*{V.3.2}.

\subsection{Fibrations}\label{ss:fibrations}
Let $\XX$ be a smooth projective surface over $k$, $\CC$ a smooth
projective curve over $k$, and $\pi:\XX\to\CC$ a non-constant
morphism.  Assume that the induced extension of function fields
$k(\CC)\into k(\XX)$ is separable and $k(\CC)$ is algebraically closed
in $k(\XX)$.  Then for every closed point $y\in\CC$, the fiber
$\pi^{-1}(y)$ is connected, and it is irreducible for almost all $y$.
Write $F$ for the class in $\NS(\XX)$ of the fiber over a $k$-rational
point $y$ of $\CC$.  (This exists because we assumed that $\XX$ has a
$k$-rational point.)  We write $\<F\>$ for the subgroup of $\NS(\XX)$
generated by $F$.

It is clear from the definition of $\NS(\XX)$ that if $y'$ is another
closed point of $\CC$, then the class in $\NS(\XX)$ of $\pi^{-1}(y')$
is equal to $(\deg y')F$.

Now suppose that $z\in\CC$ is a closed point such that $\pi^{-1}(z)$
is reducible, say
$$\pi^{-1}(z)=\sum_{i=1}^{f_z}n_iZ_i$$
where the $Z_i$ are the irreducible components of $\pi^{-1}(z)$ and
the $n_i$ are their multiplicities in the fiber.  Then a consideration
of intersection multiplicities 
(see for example \cite{SilvermanAT}*{III.8})
shows that for any integers $m_i$,
$$\sum_im_iZ_i\in\<F\>\subset\NS(\XX)$$
if and only if there is a rational number $\alpha$ such that
$m_i=\alpha n_i$ for all $i$.  More precisely, the intersection
pairing restricted to the part of $\NS(\XX)$ generated by the classes
of the $Z_i$ is negative semi-definite, with a one-dimensional kernel
spanned by integral divisors that are rational multiples of the whole
fiber.  It follows that the subgroup of $\NS(\XX)/\<F\>$ generated by
the classes of the $Z_i$ has rank $f_z-1$.  It is free of this rank if
the gcd of the multiplicities $n_i$ is 1.

It also follows that if $D$ is a divisor supported on a fiber of $\pi$ and
$D'$ is another divisor supported on other fibers, then $D=D'$ in
$\NS(\XX)/\<F\>$ if and only if $D=D'=0$ in $\NS(\XX)/\<F\>$.

Define $L^2\NS(\XX)$ to be the subgroup of $\NS(\XX)$ generated by
all components of all fibers of $\pi$ over closed points of $\CC$.  By
the above, it is the direct sum of the $\<F\>$ and the subgroups of
$\NS(\XX)/\<F\>$ generated by the components of the various fibers.
Thus we obtain the following computation of the rank of $L^2\NS(\XX)$.

\begin{prop}
For a closed point $y$ of $\CC$, let $f_y$ denote
the number of irreducible components in the fiber $\pi^{-1}(y)$.
Then the rank of
$L^2\NS(\XX)$ is 
$$1+\sum_y(f_y-1).$$
If for all $y$ the greatest common
divisor of the multiplicities of
the components in the fiber of $\pi$ over $y$ is 1, then
$L^2\NS(\XX)$ is torsion-free.
\end{prop}

\numberwithin{equation}{section}

\section{Tate's conjectures $T_1$ and $T_2$}
Tate's conjecture $T_1$ for $\XX$ (which we denote $T_1(\XX)$)
characterizes the image of the cycle class map:

\begin{conj}[$T_1(\XX)$] For any prime $\ell\neq p$, the cycle class
  map induces an isomorphism
$$\NS(\XX)\tensor\Ql\to H^2(\XXbar,\Ql(1))^{G_k}$$
\end{conj}

We will see below that $T_1(\XX)$ is equivalent to the apparently
stronger integral statement that the cycle class induces an
isomorphism
$$\NS(\XX)\tensor\Zl\to H^2(\XXbar,\Zl(1))^{G_k}$$

We will also see that $T_1(\XX)$ is independent of $\ell$ which is why
we have omitted $\ell$ from the notation.

Since $G_k$ is generated topologically by $Fr_q$, we have
$$H^2(\XXbar,\Ql(1))^{G_k}=H^2(\XXbar,\Ql(1))^{Fr_q=1}
=H^2(\XXbar,\Ql)^{Fr_q=q}.$$
The injectivity of the cycle class map implies that
$$\rk \NS(\XX)\le \dim_{\Ql}H^2(\XXbar,\Ql)^{Fr_q=q}$$
and $T_1(\XX)$ is the statement that these two dimensions are equal.

The second Tate conjecture relates the zeta-function to divisors.
Recall that $\zeta(\XX,s)$ denotes the zeta function of $\XX$, defined
in Lecture~0, Section~\ref{s:zetas}.

\begin{conj}[$T_2(\XX)$] We have
$$\rk \NS(\XX) = -\ord_{s=1}\zeta(\XX,s)$$
\end{conj}

Note that by the Riemann hypothesis, the poles of $\zeta(\XX,s)$ at
$s=1$ come from $P_2(\XX,q^{-s})$.  More precisely, using the
cohomological formula~(\ref{eq:P-cohom}) of Lecture~0 for $P_2$, we
have that the order of pole of $\zeta(\XX,s)$ at $s=1$ is equal to the
multiplicity of $q$ as an eigenvalue of $Fr_q$ on $H^2(\XXbar,\Ql)$.

Thus we have a string of inequalities
\begin{equation}\label{prop:T-ineqs}
\rk \NS(\XX)\le \dim_{\Ql}H^2(\XXbar,\Ql)^{Fr_q=q}
\le -\ord_{s=1}\zeta(\XX,s).
\end{equation}
Conjecture $T_1(\XX)$ is that the first inequality is an equality and
conjecture $T_2(\XX)$ is that the leftmost and rightmost integers
are equal.  It follows trivially that $T_2(\XX)$ implies
$T_1(\XX)$.  Tate proved the reverse implication.

\begin{prop}\label{prop:T1->T2}
The conjectures $T_1(\XX)$ and $T_2(\XX)$ are equivalent.  In
particular, $T_1(\XX)$ is independent of $\ell$.
\end{prop}

\begin{proof}
First note that the intersection pairing on $\NS(\XX)$ is
non-degenerate, so we get an isomorphism
$$\NS(\XX)\tensor\Ql\cong\Hom(\NS(\XX),\Ql).$$
On the other hand, the cup product on $H^2(\XXbar,\Ql(1))$ is also
non-degenerate (by Poincar\'e duality), so we have
$$H^2(\XXbar,\Ql(1))\cong\Hom(H^2(\XXbar,\Ql(1)),\Ql).$$
If we use a superscript $G_k$ to denote invariants and a subscript
$G_k$ to denote coinvariants, then we have a natural homomorphism
$$H^2(\XXbar,\Ql(1))^{G_k}\to H^2(\XXbar,\Ql(1))_{G_k}$$
which is an isomorphism if and only if the subspace of
$H^2(\XXbar,\Ql(1))$ where $\Fr_q$ acts by 1 is equal to the whole of
the generalized eigenspace for the eigenvalue 1.  As we have seen
above, this holds if and only if we have
$$\dim_{\Ql}H^2(\XXbar,\Ql)^{Fr_q=q}= -\ord_{s=1}\zeta(\XX,s).$$

Now consider the diagram
\begin{equation*}
\xymatrix{
\NS(\XX)\tensor\Ql\ar[d]^h\ar@{=}[rr]&&\Hom(\NS(\XX),\Ql)\\
H^2(\XXbar,\Ql(1))^{G_k}\ar[r]^f&H^2(\XXbar,\Ql(1))_{G_k}\ar@{=}[r]&
\Hom(H^2(\XXbar,\Ql(1))^{G_k},\Ql)\ar[u]_{h^*}.}
\end{equation*}
The lower right arrow is an isomorphism by elementary linear algebra.
The maps $h$ and $h^*$ are the cycle map and its transpose and they
are isomorphisms if and only if $T_1(\XX)$ holds.  One checks that the
diagram commutes (\cite{Tate66b}*{p.~24} or
\cite{Milne75}*{Lemma~5.3}) and so $T_1(\XX)$ implies that $f$ is an
isomorphism.  Thus $T_1(\XX)$ implies $T_2(\XX)$.
\end{proof}

We remark that the equality of $\dim_{\Ql}H^2(\XXbar,\Ql)^{Fr_q=q}$
and $-\ord_{s=1}\zeta(\XX,s)$ would follow from the semi-simplicity of
$Fr_q$ acting on $H^2(\XXbar,\Ql)$ (or even from its semisimplicity on
the $Fr_q=q$ generalized eigenspace).  This is a separate ``standard''
conjecture (see for example \cite{Tate94}); it does not seem to imply
$T_1(\XX)$.

\section{$T_1$ and the Brauer group}
We define the (cohomological) Brauer group $\Br(\XX)$ by
$$\Br(\XX)=H^2(\XX,\G_m)=H^2(\XX,\OO_\XX^\times)$$
(with respect to the \'etale or finer topologies).  Because $\XX$ is a
smooth proper surface over a finite field, the cohomological Brauer
group is isomorphic to the usual Brauer group (defined in terms of
Azumaya algebras) and it is known to be a torsion group.  (See
\cite{MilneEC}*{IV.2} and also three fascinating articles by
Grothendieck collected in \cite{Grothendieck68}.)  Artin and Tate
conjectured in \cite{Tate66b} that $\Br(\XX)$ is finite.

Similarly, define
$$\Br(\XXbar)=H^2(\XXbar,\G_m)=H^2(\XXbar,\OO_\XX^\times).$$
This group is torsion but need not be finite.

Taking the cohomology of the exact sequence
$$0\to\mu_{\ell^n}\to\G_m\stackrel{\ell^n}{\longrightarrow}\G_m\to0$$
as in Section~\ref{s:cycle}, we have an exact sequence
\begin{equation}\label{eq:Kummer-ell^n}
0\to\NS(\XXbar)/\ell^n\to
H^2(\XXbar,\mu_{\ell^n})\to\Br(\XXbar)_{\ell^n}\to0.
\end{equation}
Taking $G_k$-invariants and then the inverse limit over powers of
$\ell$, we obtain an exact sequence
$$0\to\NS(\XX)\tensor\Zl\to H^2(\XXbar,\Zl(1))^{G_k}\to
T_\ell\Br(\XX)\to0.$$
Since $\Br(\XX)_\ell$ is finite, $T_\ell\Br(\XX)$ is zero if and only
if the $\ell$-primary part of $\Br(\XX)$ is finite.  It follows that
the $\ell$ part of the Brauer group is finite if and only if
$T_1(\XX)$ for $\ell$ holds if and only if the integral version of
$T_1(\XX)$ for $\ell$ holds.  In particular, since $T_1(\XX)$ is
independent of $\ell$, if $\Br(\XX)[\ell^\infty]$ is finite for one
$\ell$, then $\Br(\XX)[\ell^\infty]$ is finite for all $\ell\neq p$.
It is even true, although more difficult to prove, that $T_1(\XX)$ is
equivalent to the finiteness of $\Br(\XX)$.

\begin{thm}\label{thm:T1-Br}
  $T_1(\XX)$ holds if and only if $\Br(\XX)$ is finite if and only if
  there is an $\ell$ \textup{(}$\ell=p$ allowed\textup{)} such that
  the $\ell$-primary part of $\Br(\XX)$ is finite.
\end{thm}

\begin{proof}
  We sketch the proof of the prime-to-$p$ part of this assertion
  following \cite{Tate66b} and refer to \cite{Milne75} for the full
  proof.  We already noted that the $\ell$-primary part of $\Br(\XX)$
  is finite for one $\ell\neq p$ if and only if $T_1(\XX)$ holds.  To
  see that almost all $\ell$-primary parts vanish, we consider the
  following diagram, which is an integral version of the diagram in
  the proof of Proposition~\ref{prop:T1->T2}:
\begin{equation*}
\xymatrix@C-15pt{
\NS(\XX)\tensor\Zl\ar[d]_h\ar[r]^>>>>{e}&
\Hom(\NS(\XX)\tensor\Zl,\Zl)\ar@{=}[r]&
\Hom(\NS(\XX)\tensor\Ql/\Zl,\Ql/\Zl)\\
H^2(\XXbar,\Zl(1))^{G_k}\ar[r]^f&H^2(\XXbar,\Zl(1))_{G_k}\ar@{=}[r]&
\Hom(H^2(\XXbar,(\Ql/\Zl)(1))^{G_k},\Ql/\Zl)\ar[u]_{g^*}}
\end{equation*}

Here $e$ is induced by the intersection form, $h$ is the cycle class
map, $f$ is induced by the identity map of $H^1(\XXbar,\Zl(1))$ and
$g^*$ is the transpose of a map
$$g:\NS(\XX)\tensor\Ql/\Zl\to H^2(\XXbar,(\Ql/\Zl)(1))$$
obtained by taking the direct limit over powers of $\ell$ of the first
map in equation~(\ref{eq:Kummer-ell^n}).

We say that a homomorphism $\phi:A\to B$ of $\Zl$-modules is a {\it
  quasi-isomorphism\/} if it has a finite kernel and cokernel.  In this
case, we define 
$$z(\phi)=\frac{\#\ker(\phi)}{\#\coker(\phi)}.$$  
It is easy to check that if $\phi_3=\phi_2\phi_1$ (composition) and if
two of the maps $\phi_1$, $\phi_2$, $\phi_3$ are quasi-isomorphisms,
then so is the third and we have $z(\phi_3)=z(\phi_2)z(\phi_1)$.

In the diagram above, if we assume $T_1(\XX)$, then $h$ is an
isomorphism.  The map $e$ is induced from the intersection pairing and
is a quasi-isomorphism and $z(e)$ is (the $\ell$ part of) the order of
the torsion subgroup of $\NS(\XX)$ divided by (the $\ell$ part of)
discriminant of the intersection form.  We saw above that under the
assumption of $T_1(\XX)$, the map $f$ is a quasi-isomorphism and it
turns out that $z(f)$ is essentially (the $\ell$ part of) the leading
term of the zeta function $\zeta(\XX,s)$ at $s=1$.  In particular,
under $T_1(\XX)$, $e$, $f$, and $h$ are isomorphisms for almost all
$\ell$.  The same must therefore be true of $g^*$.  By taking
$G_k$-invariants and a direct limit over powers of $\ell$ in
equation~(\ref{eq:Kummer-ell^n}), one finds that $z(g^*)$ is equal to
the order of $\Br(\XX)[\ell^\infty]$ and so this group is trivial for
almost all $\ell$.  This completes our sketch of the proof of the
theorem.
\end{proof}

The sketch above has all the main ideas needed to prove that the
prime-to-$p$ part of the Artin-Tate conjecture on the leading
coefficient of the zeta function at $s=1$ follows from the Tate
conjecture $T_1(\XX)$.  The $p$-part is formally similar although more
delicate.  To handle it, Milne replaces the group in the lower right
of the diagram with the larger group
$\Hom(H^2(\XX,(\Qp/\Zp)(1)),\Qp/\Zp)$.  The $z$ invariants of the maps
to and from this group turn out to have more $p$-adic content that is
related to the term $q^\alpha(\XX)$ in the Artin-Tate leading
coefficient conjecture.  We refer to \cite{Milne75} for the full
details and to \cite{UlmerCRM} for a discussion of several related
points, including the case $p=2$ (excluded in Milne's article, but now
provable due to improved $p$-adic cohomology) and higher dimensional
abelian varieties.

\section{The descent property of $T_1$}

If $\tilde\XX\to\XX$ is the blow up of $\XX$ at a closed point, then
$T_1(\tilde\XX)$ is equivalent to $T_1(\XX)$.  Indeed, under blowing
up both the rank of $\NS(\cdot)$ and the dimension of
$H^2(\cdot,\Ql(1))^{G_k}$ increase by one.  (See
Example~\ref{ss:blow-ups} above.)  In fact:

\begin{prop}\label{prop:T-descent}
  $T_1(\XX)$ is invariant under birational isomorphism.  More
  generally, if $\XX\to\YY$ is a dominant rational map, then
  $T_1(\XX)$ implies $T_1(\YY)$.
\end{prop}

\begin{proof}
  We give simple proof of the case where $\XX$ and $\YY$ are surfaces.
  See \cite{Tate94} for the general case.

  First, we may assume $\XX\ratto\YY$ is a morphism.  Indeed, let
  $\tilde\XX\to\XX$ be a blow up resolving the indeterminacy of
  $\XX\ratto\YY$, i.e., so that the composition
  $\tilde\XX\to\XX\ratto\YY$ is a morphism.  As we have seen above
  $\T_1(\XX)$ implies $T_1(\tilde\XX)$ so we may replace $\XX$ with
  $\tilde\XX$ and show that $T_1(\YY)$ holds.

  So now suppose that $\pi:\XX\to\YY$ is a dominant morphism.  Since
  the dimensions of $\XX$ and $\YY$ are equal, $\pi$ must be
  generically finite, say of degree $d$.  But then the push forward
  and pull-back maps on cycles present $\NS(\YY)\tensor\Ql$ as a
  direct factor of $NS(\XX)\tensor\Ql$; they also present
  $H^2(\YYbar,\Ql(1))$ as a direct factor of $H^2(\XXbar,\Ql(1))$.
  The cycle class maps and Galois actions are compatible with these
  decompositions and since by assumption $NS(\XX)\tensor\Ql\isoto
  H^2(\XXbar,\Ql(1))^{G_k}$, we must also have
  $NS(\YY)\tensor\Ql\isoto H^2(\YYbar,\Ql(1))^{G_k}$, i.e.,
  $T_1(\YY)$.
\end{proof}

Note that the dominant rational map $\XX\ratto\YY$ could be a ground
field extension, or even a purely inseparable morphism.

\section{Tate's theorem on products}
In this section we sketch how $T_1$ for products of curves follows from
Tate's theorem on endomorphisms of abelian varieties over finite
fields.

\begin{thm}[Tate]\label{thm:products}
Let $\CC$ and $\DD$ be curves over $k$ and set $\XX=\CC\times_k\DD$.
Then $T_1(\XX)$ holds.
\end{thm}

\begin{proof}
  Extending $k$ if necessary, we may assume that $\CC$ and $\DD$ both
  have rational points.  Fix rational base points $x$ and $y$ (which we
  will mostly omit from the notation below).  Recall from
  Subsection~\ref{ss:products} that
$$\NS(\CC\times\DD)\cong\Z^2\times\DivCorr(\CC,\DD)
\cong\Z^2\times\Hom(J_\CC,J_\DD).$$

By the K\"unneth formula,
\begin{align*}
H^2(\XXbar,\Ql)&\cong 
\left(H^2(\CCbar,\Ql)\tensor H^0(\DDbar,\Ql)\right)\oplus
\left(H^0(\CCbar,\Ql)\tensor H^2(\DDbar,\Ql)\right)\\
&\qquad\oplus \left(H^1(\CCbar,\Ql)
\tensor H^1(\DDbar,\Ql)\right)\\
&\cong\Ql(-1)\oplus\Ql(-1)\oplus 
\left(H^1(\CCbar,\Ql)\tensor H^1(\DDbar,\Ql)\right)
\end{align*}
Twisting by $\Ql(1)$ and taking invariants, we have
$$H^2(\XXbar,\Ql(1))^{G_k}=
\Ql^2\oplus\left(H^1(\CCbar,\Ql)\tensor
  H^1(\DDbar,\Ql)(1)\right)^{G_k}.$$ 
Under the cycle class map, the factor $\Z^2$ of $\NS(\XX)$
(corresponding to $\CC\times\{y\}$ and $\{x\}\times\DD$) spans the
factor $\Ql^2$ of $H^2(\XXbar,\Ql(1))^{G_k}$ (corresponding to
$H^2\tensor H^0$ and $H^0\tensor H^2$ in the Kunneth decomposition).
Thus what we have to show is that the cycle class map induces an
isomorphism
$$\Hom(J_\CC,J_\DD)\tensor\Ql\isoto 
\left(H^1(\CCbar,\Ql)\tensor H^1(\DDbar,\Ql)(1)\right)^{G_k}$$

But $H^1(\DDbar,\Ql)(1)\cong H^1(\DDbar,\Ql)^*\cong V_\ell(J_\DD)$ and
$H^1(\CCbar,\Ql)\cong V_\ell(J_\CC)^*$ ($*=\Ql$-linear dual).  Thus
$$\left(H^1(\CCbar,\Ql)\tensor H^1(\DDbar,\Ql)(1)\right)^{G_k}\cong
\left( V_\ell(J_\CC)^*\tensor V_\ell(J_\DD)\right)^{G_k}
\cong\Hom_{G_k}(V_\ell(J_\CC),V_\ell(J_\DD)).$$

Thus the needed isomorphism is
$$\Hom(J_\CC,J_\DD)\tensor\Ql\isoto
\Hom_{G_k}(V_\ell(J_\CC),V_\ell(J_\DD))$$ 
and this is exactly the statement of Tate's theorem (Lecture~0,
Theorem~\ref{thm:TateIsogThm}).  This completes the proof of the
theorem.
\end{proof}

\begin{rems}\leavevmode
\begin{enumerate}
\item A variation of the argument above, using Picard and Albanese
  varieties, shows that $T_1$ for a product $\XX\times\YY$ of
  varieties of any dimension follows from $T_1$ for the factors.
\item It is worth noting that Tate's conjecture $T_1$ (and the proof
  of it for products of curves) only characterizes the image of in
  $\ell$-adic cohomology of $\NS(\XX)\tensor\Zl$, not the image of
  $\NS(\XX)$ itself.  This should be contrasted with the Lefschetz
  $(1,1)$ theorem, which characterizes the image of $\NS(\XX)$ in
  deRham cohomology when the ground field is $\C$.
\end{enumerate}
\end{rems}

\section{Products of curves and DPC}
Assembling the various parts of this lecture gives the main result:

\begin{prop}\label{prop:T-DPC}
Let $\XX$ be a smooth, projective surface over $k$.  If there is a
dominant rational map 
$$\CC\times_k\DD\ratto\XX$$
from a product of curves to $\XX$, then the Tate conjectures
$T_1(\XX)$ and $T_2(\XX)$ hold.
\end{prop}

Indeed, by Theorem~\ref{thm:products}, we have $T_1(\CC\times\DD)$ and
then by Proposition~\ref{prop:T-descent} we deduce $T_1(\XX)$.  By
Proposition~\ref{prop:T1->T2}, $T_2(\XX)$ follows as well.

We say that ``$\XX$ is dominated by a product of curves (DPC).''  The
question of which varieties are dominated by products of curves has
been studied by Schoen \cite{Schoen96}.  In particular, over any field
there are surfaces that are not dominated by products of curves.
Nevertheless, as we will see below, the collection of DPC surfaces is
sufficiently rich to give some striking results on the Birch and
Swinnerton-Dyer conjecture.


\lecture{Elliptic curves and elliptic surfaces}\label{l:BSD-T}
We keep our standard notations throughout this lecture:  
$p$ is a prime, $k=\Fq$ is the finite field of characteristic $p$ with
$q$ elements, $\CC$ is a smooth, projective, absolutely irreducible
curve over $k$, $K=k(\CC)$ is the function field of $\CC$, and $E$ is
an elliptic curve over $K$.

\section{Curves and surfaces}\label{s:curves-surfaces}
In this section we will construct an elliptic surface $\EE\to\CC$
canonically associated to an elliptic curve $E/K$.  More precisely, we
give a constructive proof of the following result:

\begin{prop}\label{prop:model}
  Given an elliptic curve $E/K$, there exists a surface $\EE$ over $k$
  and a morphism $\pi:\EE\to\CC$ with the following properties: $\EE$
  is smooth, absolutely irreducible, and projective over $k$, $\pi$ is
  surjective and relatively minimal, and the generic fiber of $\pi$ is
  isomorphic to $E$.  The surface $\EE$ and the morphism $\pi$ are
  uniquely determined up to isomorphism by these requirements.
\end{prop}

Here ``the generic fiber of $\pi$'' means $\EE_K$,  the fiber product:
\begin{equation*}
\xymatrix{\EE_K:=\eta\times_{\CC}\EE\ar[r]\ar[d]&\EE\ar[d]^{\pi}\\
\eta=\spec K\ar[r]&\CC}
\end{equation*}
``Relatively minimal'' means that if $\EE'$ is another smooth,
absolutely irreducible, projective surface over $k$ with a surjective
morphism $\pi':\EE'\to\CC$, then any birational morphism $\EE\to\EE'$
commuting with $\pi$ and $\pi'$ is an isomorphism.  Relative
minimality is equivalent to the condition that there are
no rational curves of self-intersection $-1$ in the fibers of $\pi$
(i.e., to the non-existence of curves in fibers that can be blown
down).

\begin{rems}
  The requirements on $\EE$ and $\pi$ imply that $\pi$ is flat and
  projective and that all geometric fibers of $\pi$ are connected.
  These properties of $\pi$ will be evident from the explicit
  construction below.  It follows that $\pi_*\OO_\EE\cong\OO_\CC$ and
  more generally that $\pi$ is ``cohomologically flat in dimension
  zero,'' meaning that for every morphism $T\to C$ the base change 
$$\pi_T:\EE_T=\EE\times_\CC T\to T$$ 
satisfies $\pi_{T*}\OO_{\EE_T}=\OO_T$.
\end{rems}

Uniqueness in Proposition~\ref{prop:model} follows from general
results on minimal models, in particular
\cite{Lichtenbaum68}*{Thm.~4.4}.  See \cite{Chinburg86} and
\cite{LiuAGAC}*{9.3} for other expositions.

We first give a detailed construction of a (possibly singular)
``Weierstrass surface'' $\WW\to\CC$ and then resolve singularities to
obtain $\EE\to\CC$.

More precisely, the proposition follows from the following two results.

\begin{prop}\label{prop:W}
  Given an elliptic curve $E/K$, there exists a surface $\WW$ over $k$
  and a morphism $\pi_0:\WW\to\CC$ with the following properties:
  $\WW$ is normal, absolutely irreducible, and projective over $k$,
  $\pi_0$ is surjective, each of its fibers is isomorphic to an
  irreducible plane cubic, and its generic fiber is isomorphic to $E$.
\end{prop}

This proposition is elementary, but does not seem to be explained in
detail in the literature, so we give a proof below.

\begin{prop}\label{prop:Tate-algo}
  There is an explicit sequence of blow ups \textup{(}along closed
  points and curves in $\WW$\textup{)} yielding a proper birational
  morphism $\sigma:\EE\to\WW$ where the surface $\EE$ and the composed
  morphism $\pi=\pi_0\compose\sigma:\EE\to\WW\to\CC$ have the
  properties mentioned in Proposition~\ref{prop:model}.
\end{prop}

\begin{proof}[Proof of Proposition~\ref{prop:W}]
  Choose a Weierstrass equation for $E$:
\begin{equation}\label{eq:model}
y^2+a_1xy+a_3y=x^3+a_2x^2+a_4x+a_6
\end{equation}
where the $a_i$ are in $K=k(\CC)$.  Recall that we have defined the
notion of a minimal integral model at a place $v$ of $K$: the $a_i$
should be integral at $v$ and the valuation at $v$ of $\Delta$ should
be minimal subject to the integrality of the $a_i$.  Clearly, there is
a non-empty Zariski open subset $U\subset\CC$ such that for every
closed point $v\in U$, the model~(\ref{eq:model}) is a minimal
integral model.

Let $\WW_1$ be the closed subset of $\P^2_U:=\P^2_k\times_k U$ defined
by the vanishing of
\begin{equation}\label{eq:EE}
Y^2Z+a_1XYZ+a_3YZ^2-\left(X^3+a_2X^2Z+a_4XZ^2+a_6Z^3\right)
\end{equation}
where $X,Y,Z$ are the standard homogeneous coordinates on $\P^2_k$.
Then $\WW_1$ is geometrically irreducible and there is an obvious
projection $\pi_1:\WW_1\to U$ (the restriction to $\WW_1$ of the
projection $\P^2_U\to U$).  The fiber of $\pi_1$ over a closed point
$v$ of $U$ is the plane cubic
$$Y^2Z+a_1(v)XYZ+a_3(v)YZ^2=X^3+a_2(v)X^2Z+a_4(v)XZ^2+a_6(v)Z^3$$
over the residue field $\kappa_v$ at $v$.  The generic point $\eta$ of
$\CC$ lies in $U$ and the fiber  of $\pi_1$ at $\eta$ is $E/K$.  

There are finitely many points in $\CC\setminus U$ and we must extend
the model $\WW_1\to U$ over each of these points.  Choose one of them,
call it $w$, and choose a model of $E$ that is integral and minimal at
$w$.  In other words, choose a model of $E$
\begin{equation}\label{eq:model'}
y^{\prime2}+a'_1x'y'+a'_3y'=x^{\prime3}+a'_2x'^2+a'_4x'+a'_6
\end{equation}
where the $a'_i\in K$ are integral at $w$ and the valuation at $w$ of the
discriminant $\Delta$ is minimal.  The new coordinates are related to
the old by a transformation
\begin{equation}\label{eq:change-of-coords}
(x,y)=(u^2x'+r,u^3y'+su^2x'+t)
\end{equation}
with $u\in K^\times$ and $r,s,t\in K$.  Let $U'$ be a Zariski open
subset of $\CC$ containing $w$ on which all of the $a_i'$ are integral
and the model~(\ref{eq:model'}) is minimal.  Let $\WW'$ be the
geometrically irreducible closed subset of $\P^1_{U'}$ defined by the
vanishing of
$$Y^{\prime2}Z'+a'_1X'Y'Z'+a'_3Y'Z^{\prime2}-
\left(X^{\prime3}+a'_2X^{\prime2}Z'+a'_4X'Z^{\prime2}+a'_6Z^{\prime}3\right)$$
with its obvious projection $\pi':\WW'\to U'$.
On the open set $V=U\cap U'$, $u$ is a unit and
the change of coordinates~(\ref{eq:change-of-coords}), or rather
its projective version
$$(X,Y,Z)=(u^2X'+rZ,u^3Y'+su^2X'+tZ',Z')$$
defines an isomorphism between $\pi_1^{-1}(V)$ and $\pi^{\prime-1}(V)$
compatible with the projections.  Glueing $\WW_1$ and $\WW'$ along
this isomorphism yields a new surface $\WW_2$ equipped with a
projection $\pi_2:\WW_2\to U_2$ where $U_2=U\cup U'$.  Note that $U_2$
is strictly larger than $U$.  Moreover $\pi_2$ is surjective, its
geometric fibers are irreducible projective plane cubics, and its
generic fiber is $E$.

We now iterate this construction finitely many times to extend the
original model over all of $\CC$.  We arrive at a surface $\WW$
equipped with a proper, surjective morphism $\pi:\WW\to\CC$ whose
geometric fibers are irreducible plane cubics and whose generic fiber
is $E$.  Since $\CC$ is projective over $k$, so is $\WW$.  Since $\WW$
is obtained by glueing reduced, geometrically irreducible surfaces
along open subsets, it is also reduced and geometrically irreducible.
Since it has only isolated singular points, by Serre's criterion it is
normal.

This completes the proof of Proposition~\ref{prop:W}.
\end{proof}

Note that the closure in $\WW$ of the identity element of $E$ is a
divisor on $\WW$ which maps isomorphically to the base curve $\CC$.
We write $s_0:\CC\to\WW$ for the inverse morphism.  This is the {\it
  zero section\/} of $\pi_0$.  In terms of the coordinates on $\WW_1$
used in the proof above, it is just the map $t\mapsto([0,1,0],t)$.

\begin{proof}[Discussion of Proposition~\ref{prop:Tate-algo}]
The algorithm mentioned in the Proposition is the
subject of Tate's famous paper \cite{Tate75}.  His article does not
mention blowing up, but the steps of the algorithm nevertheless give
the recipe for the blow ups needed.  The actual process of blowing up
is explained in detail in \cite{SilvermanAT}*{IV.9} so we will
not give the details here.  Rather, we explain why there is a simple
algorithm, following \cite{Conrad05}.

First note that the surface $\WW$ is reduced and irreducible and so
has no embedded components.  Also, it has isolated singularities.
(They are contained in the set of singular points of fibers of
$\pi_0$.)  By Serre's criterion, $\WW$ is thus normal.  Moreover, and
this is the key point, its singularities are {\it rational double
  points\/}.  (See \cite{Artin86} for the definition and basic
properties of rational singularities and \cite{BadescuAS}*{Chapters 3
  and 4} for many more details.  See \cite{Conrad05}*{Section~8} for
the fact that the singularities of a minimal Weierstrass model are
rational.)  This implies that the blow up of $\WW$ at one of its
singular points is again normal (so has isolated singularities) and
again has at worst rational double points.  An algorithm to
desingularize is then simply to blow up at a singular point and
iterate until the resulting surface is smooth.  Given the explicit
nature of the equations defining $\WW$, finding the singular points
and carrying out the blow ups is straightforward.

In fact, Tate's algorithm also calls for blowing up along certain
curves.  (This happens at steps 6 and 7.)  This has the effect of
dealing with several singular points at the same time, so is more
efficient, but it is not essential to the success of the algorithm.

This completes our discussion of Proposition~\ref{prop:Tate-algo}.
See below for a detailed example covering a case not treated
explicitly in \cite{SilvermanAT}.
\end{proof}

Conrad's article \cite{Conrad05} also gives a coordinate-free
treatment of integral minimal models of elliptic curves.

It is worth remarking that Tate's algorithm and the possible
structures of the bad fibers are essentially the same in
characteristic $p$ as in mixed characteristic.  On the other hand, for
non-perfect residue fields $k$ of characteristic $p\le3$, there are
more possibilities for the bad fibers, in both equal and mixed
characteristics---see \cite{Szydlo04}.

The zero section of $\WW$ lifts uniquely to a section which we
again denote by $s_0:\CC\to\EE$.

\section{The bundle $\omega$ and the height of $\EE$}
We construct an invertible sheaf on $\CC$ as follows, using the
notation of the proof of Proposition~\ref{prop:model}.  Take the
trivial invertible sheaf $\OO_U$ on $U$ with its generating section
$1_U$.  At each stage of the construction, extend this sheaf by
glueing $\OO_U$ and $\OO_{U'}$ over $U\cap U'$ by identifying $1_U$
and $u^{-1}1_{U'}$ where $u$ is the function appearing in the change
of coordinates~(\ref{eq:change-of-coords}).

The resulting invertible sheaf $\omega$ has several other
descriptions.  For example, the sheaf of relative differentials
$\Omega^1_{\EE/\CC}$ is invertible on the locus of $\EE$ where
$\pi:\EE\to\CC$ is smooth (in particular in a neighborhood of the zero
section) and, more or less directly from the definition, $\omega$ can
be identified with $s_0^*(\Omega^1_{\EE/\CC})$.  Using relative
duality theory, $\omega$ can also be identified with the inverse of
$R^1\pi_*\OO_\EE$.  Finally, since $\WW$ as only rational
singularities, $\omega$ is also isomorphic to $R^1\pi_{0*}\OO_\WW$

One may identify the coefficients $a_i$ of the Weierstrass equation
locally defining $\WW$ with sections of $\omega^i$.  Using this point
of view, $\WW$ can be identified with a closed subvariety of a certain
$\P^2$-bundle over $\CC$.  Namely, let $V$ be the locally free
$\OO_\CC$ module of rank three
\begin{equation}\label{eq:bundle}
V=\omega^2\oplus\omega^3\oplus\OO_\CC
\end{equation}
(where the exponents denote tensor powers).  If $\P V$ denotes the
projectivization of $V$ over $\CC$, a $\P^2$ bundle over $\CC$, then
$\WW$ is naturally the closed subset of $\P E$ defined locally by the
vanishing of Weierstrass equations as in (\ref{eq:EE}).

\begin{exers}
  Verify the identifications and assertions in this section.  In the
  case where $\CC=\P^1$, so $K=k(t)$, check that
  $\omega=\OO_{\P^1}(h)$ where $h$ is the smallest positive integer
  such that $E$ has a model (\ref{eq:model}) where the $a_i$ are in
  $k[t]$ and $\deg a_i\le hi$.
\end{exers}

\begin{exers}
  Check that $c_4$, $c_6$, and $\Delta$ define {\it canonical\/}
  sections of $\omega^4$, $\omega^6$, and $\omega^{12}$ respectively,
  independent of the choice of equation for $E$.  If $p=2$ or $3$,
  check that $b_2$ defines a canonical section of $\omega^2$ and that
  $c_4=b_2^2$ and $c_6=-b_2^3$.  If $p=2$, check that $a_1$ defines a
  canonical section of $\omega$ and that $b_2=a_1^2$.  Note that since
  positive powers of $\omega$ have non-zero sections, the degree of
  $\omega$ is non-negative.
\end{exers}

\begin{defn}
  The {\it height\/} of $\EE$, denoted $h$, is defined by
  $h=\deg(\omega)$, the degree of $\omega$ as an invertible sheaf on
  $\CC$.
\end{defn}

Note that if $E/K$ is constant (in the sense of Lecture~1) then
the height of the corresponding $\EE$ is 0.

\section{Examples}
The case when $\CC=\P^1$ is particularly simple.  First of all, one
may choose a model~(\ref{eq:model}) that is integral and minimal
simultaneously at every finite $v$, i.e., for every $v\in\A^1_k$.
Indeed, start with any model and change coordinates so that the $a_i$
are in $k[t]$.  If $w$ is a finite place where this model is not
minimal, it is possible (because $k[t]$ is a PID) to choose a change
of coordinates
$$(x,y)=(u^2x'+r,u^3y'+su^2x'+t)$$
where $r,s,t,u\in k[t][1/w]$ and $u$ a unit yielding a model that is
minimal at $w$.  Such a change of coordinates does not change the
minimality at any other finite place.  Thus after finitely many steps,
we have a model integral and minimal at all finite places.  (This
argument would apply for any $K$ and any Dedekind domain $R\subset K$
which is a PID, yielding a model with the $a_i\in R$ that is minimal
at all $v\in\spec R$.)

Focusing attention at $t=\infty$, there is a change of
coordinates~(\ref{eq:change-of-coords}) with $u=t^{-h}$
yielding a model integral and minimal at $\infty$.  (Here $h$ is
minimal so that $\deg(a_i)\le hi$.)  So the bundle
$\omega=\O(h)=\O(h\infty)$.  

As a very concrete example, consider the curve
$$y^2=x(x+1)(x+t^d)$$
over $\Fp(t)$ where $p>2$ and $d$ is not divisible by $p$.  Since
$\Delta=16t^{2d}(t^d-1)^2$, this model is integral and minimal at all
non-zero finite places.  It is also minimal at zero as one may see by
noting that $c_4$ and $c_6$ are units at 0.  At infinity, the change
of coordinates
$$(x,y)=(t^{2h}x',t^{3h}y')$$
with $h=\lceil d/2\rceil$ yields a minimal integral model.  Thus
$\omega=\O(h)$.

Working with Tate's algorithm shows that $E$ has $I_2$
reduction at the $d$-th roots of unity, $I_{2d}$ reduction at $t=0$,
and either $I_{2d}^*$ or $I_{2d}$ reduction at infinity depending on
whether $d$ is odd or even.  

Since the case of $I_n$ reduction is not treated explicitly in
\cite{SilvermanAT}, we give more details on the blow ups needed to
resolve the singularity over $t=0$.  In terms of the coordinates on
$\WW_1$ used in the proof of Proposition~\ref{prop:Tate-algo} we can
consider the affine surface defined by
$$x^3+(t^d+1)x^2+t^dx-y^2=0$$
which is an open neighborhood of the singularity at $x=y=t=0$.  
If $d=1$, then the tangent cone is the irreducible plane conic defined
by $x^2+tx-y^2=0$.  The singular point thus blows up into a smooth
rational curve and it is easy to check that the resulting surface is
smooth in a neighborhood of the fiber $t=0$.  Now assume that $d>1$.
Then the tangent cone is the reducible conic $x^2-y^2=0$ and so the
singular point blows up into two rational curves meeting at one point.
More precisely, the blow up is covered by three affine patches.  In
one of them, the surface upstairs is 
$$tx_1^3+(t^d+1)x_1^2+t^{d-1}x_1-y_1^2=0$$
and the morphism is $x=tx_1$, $y=ty_1$. The exceptional divisor is the
reducible curve $t=x_1^2-y_1^2=0$ and the point of intersection of the
components $t=x_1=y_1=0$ is again a double point.  Considering the
other charts shows that there are no other singular points in a
neighborhood of $t=0$ and that the exceptional divisor meets the
original fiber over $t=0$ in two points.  We now iterate this process
$d-1$ times, introducing two new components at each stage.  After
$d-1$ blow ups, the interesting part of our surface is given by
$$t^{d-1}x_{d-1}^3+(t^d+1)x_{d-1}^2+tx_{d-1}-y_{d-1}^2=0.$$
At this last stage, blowing up introduces one more component meeting
the two components introduced in the preceding step at one point each.
The (interesting part of the) surface is now
$$t^{d}x_{d}^3+(t^d+1)x_{d}^2+x_{d}-y_{d}^2=0$$
which is regular in a neighborhood of $t=0$.  Thus we see that the
fiber over $t=0$ in $\EE$ is a chain of $2d$ rational curves, i.e., a
fiber of type $I_{2d}$.

The resolution of the singularities over points with $t^d=1$ is
similar but simpler because only one blow up is required.  At
$t=\infty$, if $d$ is even then the situation is very similar to that
over $t=0$ and the reduction is again of type $I_{2d}$.  If $d$ is
odd, the reduction is of type $I_{2d}^*$.  We omit the details in this
case since it is treated fully in \cite{SilvermanAT}.

\begin{exer}
In the table in Tate's algorithm paper \cite{Tate75} (and the slightly
more precise version in \cite{SilvermanAEC}*{p.~448}), the last three
rows have restrictions on $p$.  Give examples showing that these
restrictions are all necessary for the discriminant and conductor
statements, and for the statement about $j$ in the $I_n^*$, $p=2$
case.  Show that the other assertions about the $j$-invariant are
correct for all $p$. 
\end{exer}

\section{$\EE$ and the classification of surfaces}\label{s:height}
It is sometimes useful to know how $\EE$ fits into the
Enriques-Kodaira classification of surfaces.  In this section only, we
replace $k$ with $\kbar$ and write $\EE$ for what elsewhere is denoted
$\EEbar$.

Recall that the height of $\EE$ is defined as $h=\deg\omega$.

\begin{prop}\label{prop:h=0}
$\omega\cong\O_\CC$ if and only if $E$ is constant.  If
$h=\deg(\omega)=0$, then $E$ is isotrivial.
\end{prop}

\begin{proof}
  It is obvious that if $E$ is constant, then $\omega\cong\O_\CC$.
  Conversely, suppose $\omega\cong\O_\CC$.  Then the construction of
  $\pi_0:\WW\to\CC$ in Proposition~\ref{prop:W} yields an irreducible
  closed subset of $\P^2_\CC$ (because the $\P^2$-bundle $\P V$ in
  (\ref{eq:bundle}) is trivial):
$$\WW\subset\P^2_\CC=\P^2_k\times_k\CC.$$
Let $\sigma:\WW\to\P^2_k$ be the restriction of the projection
$\P^2_C\to\P^2_k$.  Then $\sigma$ is not surjective (since most points
in the line at infinity $Z=0$ are not in the image) and so its image
has dimension $<2$.  Considering the restriction of $\sigma$ to a
fiber of $\pi_0$ shows that the image of $\sigma$ is in fact an
elliptic curve $E_0$ and then it is obvious from dimension
considerations that
$$\WW=\pi_0^{-1}(\pi_0(\WW))=E_0\times\CC.$$  
It follows that $E$, the generic fiber of $\pi_0$,
is isomorphic to $E_0\times\spec K$, i.e., that $E$ is constant.

Now assume that $h=0$.  Then $\Delta$ is a non-zero global section of
the invertible sheaf $\omega^{12}$ on $\CC$ of degree 0.  Thus
$\omega^{12}$ is trivial.  It follows that there is a finite
unramified cover of $\CC$ over which $\omega$ becomes trivial and so
by the first part, $E$ becomes constant over a finite extension, i.e.,
$E$ is isotrivial.
\end{proof}

Note that $E$ being isotrivial does not imply that $h=0$.

\begin{exer}
Give an example of a non-constant $E$ of height zero.  Hint: Consider
the quotient of a product of elliptic curves by a suitable free action
of a group of order two.
\end{exer}

\begin{prop}
The canonical bundle of $\EE$ is
$\Omega^2_\EE\cong\pi^*\left(\Omega^1_C\tensor\omega\right)$.
\end{prop}

Here we are using that $\EE\to\CC$ has a section and therefore no
multiple fibers.  The proof, which we omit, proceeds by considering
$R^1\pi_*\O_\EE$ and using relative duality.  See for example
\cite{BadescuAS}*{7.15}. 

We now consider several cases:

If $2g_\CC-2+h>0$, then it follows from the Proposition that the
dimension of $H^0(\EE,(\Omega^2)^{\tensor^n})$ grows linearly with
$n$, so $\EE$ has Kodaira dimension 1.

If $2g_\C-2+h=0$, then the Kodaira dimension of $\EE$ is zero and
there are two possibilities: (1) $g_\CC=1$ and $h=0$; or (2) $g_\CC=0$
and $h=2$.  In the first case, there is an unramified cover of $\CC$
over which $\EE$ becomes constant and so $\EE$ is the quotient of a
product of two elliptic curves.  These surfaces are sometimes called
``bi-elliptic.''  In the second case, $\Omega^2_\EE=\O_\EE$ and
$H^1(\EE,\O_\EE)=H^0(\CC,\omega^{-1})=0$ and so $\EE$ is a K3 surface.

If $2g_\CC-2+h<0$, then the Kodaira dimension of $\EE$ is $-\infty$ and
there are again two possibilities: (1) $g_\CC=0$ and $h=1$, in which
case $\EE$ is a rational surface by Castelnuovo's criterion; or (2)
$g_\CC=0$ and $h=0$, in which case $E$ is constant and $\EE$ is a
ruled surface $E_0\times\CC=E_0\times\P^1$.

\section{Points and divisors, Shioda-Tate}\label{s:Shioda-Tate}
If $D$ is an irreducible curve on $\EE$, then its generic fiber
$$D.E:=D\times_{\CC}E$$ 
is either empty or is a closed point of $E$.  The former occurs if and
only if $D$ is supported in a fiber of $\pi$.  In the latter case, the
residue degree of $D.E$ is equal to the generic degree of $D\to\CC$.
Extending by linearity, we get homomorphism
$$\divr(\EE)\to\divr(E)$$
whose kernel consists of divisors supported in the fibers of $\pi$.  

There is a set-theoretic splitting of this homomorphism, induced by the
map sending a closed point of $E$ to its scheme-theoretic closure in
$\EE$.  However, this is not in general a group homomorphism.  

Let $L^1\divr(\EE)$ be the subgroup of divisors $D$ such that the
degree of $D.E$ is zero and let $L^2\divr(\EE)$ be subgroup such that
$D.E=0$.  We write $L^i\Pic(\EE)$ and $L^i\NS(\EE)$ ($i=1,2$) for the
images of $L^i(\EE)$ in $\Pic(\EE)$ and $\NS(\EE)$ respectively.   

The Shioda-Tate theorem relates the N\'eron-Severi group of $\EE$ to the
Mordell-Weil group of $E$:

\begin{thm}\label{thm:Shioda-Tate}
If $\EE\to\CC$ is non-constant, $D\mapsto D.E$ induces an isomorphism
$$\frac{L^1\NS(\EE)}{L^2\NS(\EE)}\cong E(K)$$
If $\EE\to\CC$ is constant, we have
$$\frac{L^1\NS(\EE)}{L^2\NS(\EE)}\cong E(K)/E(k)$$
\end{thm}

This theorem seems to have been known to the ancients (Lang, N\'eron,
Weil, ...) and was stated explicitly in \cite{Tate66b} and in papers
of Shioda.  A detailed proof in a more general context is given in
\cite{Shioda99}.  Note however that in \cite{Shioda99} the ground
field is assumed to be algebraically closed.  See \cite{UlmerCRM} for
the small modifications needed to treat finite $k$.

It is obvious that $NS(\EE)/L^1NS(\EE)$ is infinite cyclic.  We saw in
Example~\ref{ss:fibrations} of Lecture~2 that $L^2NS(\EE)$ is free
abelian of rank $1+\sum_v(f_v-1)$.  So as a corollary of the theorem,
we have the following rank formula, known as the Shioda-Tate formula:
\begin{equation}\label{eq:STformula}
\rk E(K)=\rk NS(\EE)-2-\sum_v(f_v-1)
\end{equation}

For more on the geometry of elliptic surfaces and elliptic curves over
function fields, with an emphasis on rational and K3 surfaces, I
recommend \cite{SchuttShiodaES}.

\section{$L$-functions and Zeta-functions}
We are going to relate the $L$-function of $E$ and the zeta function
of $\EE$.  We note that from the definition, $Z(\EE,T)$ depends only
on the underlying set of closed points of $\EE$ and we may partition
this set using the map $\pi$.

We have
\begin{align*}
Z(\EE,T)&=\prod_{\text{closed }x\in\EE}\left(1-T^{\deg(x)}\right)^{-1}\\
&=\prod_{\text{closed  }y\in\CC}
      \prod_{x\in\pi^{-1}(y)}\left(1-T^{\deg(x)}\right)^{-1}\\
&=\prod_{\text{closed  }y\in\CC}Z(\pi^{-1}(y),T^{\deg(y)})
\end{align*}

For $y$ such that $\pi^{-1}(y)$ is a smooth elliptic curve, we know
that 
$$Z(\pi^{-1}(y),T)=\frac{(1-a_yT+q_yT^2)}{(1-T)(1-q_yT)}$$
and the numerator here is the factor that enters into the definition of
$L(E,T)$.

To complete the calculation, we need an analysis of the contribution
of the bad fibers.  We consider the fiber $\pi^{-1}(y)$ as a scheme of
finite type over the residue field $\kappa_y$, the field of $q_y$
elements.  As such, it has irreducible components.  Its ``geometric
components'' are the components of the base change to
$\overline{\kappa}_y$; these are defined over some finite extension of
$\kappa_y$.

For certain reduction types ($I_n$, $I_n^*$ ($n\ge0$), $IV$ and
$IV^*$) it may happen that all the geometric components are defined
over $\kappa_y$, in which case we say the reduction is ``split'', or
it may happen that some geometric components are only defined over a
quadratic extension of $\kappa_y$, in which case we say the reduction
is ``non-split.''  This agrees with the standard usage in the case of
$I_n$ reduction and may be non-standard in the other cases.

\begin{prop}
The zeta function of the a singular fiber of $\pi$ has the form
\begin{align*}
Z(\pi^{-1}(y),T)&=
\frac{(1-T)^{a}(1+T)^b}{(1-q_yT)^{f}(1+q_yT)^g}\\
&=\frac{1}{(1-T)(1-q_yT)}\frac{(1-T)^{a+1}(1+T)^b}{(1-q_yT)^{f-1}(1+q_yT)^g}
\end{align*}
where the integers $a$, $b$, $f$, and $g$ are
determined by the reduction type at $y$ and are given in the following
table:
$$
\vbox{\offinterlineskip\hrule
\halign{&\vrule#&\strut\quad\hfil#\hfil\quad\cr
&\hfill &&$a$&&$b$&&$f$&&$g$&\cr
\noalign{\hrule}
&\hfill split $I_n$&&$0$&&$0$&&$n$&&$0$&\cr
\noalign{\hrule}
&\hfill non-split $I_n$, $n$ odd&&$-1$&&$1$&&$(n+1)/2$&&$(n-1)/2$&\cr
\noalign{\hrule}
&\hfill non-split $I_n$, $n$ even&&$-1$&&$1$&&$n/2+1$&&$(n-2)/2$&\cr
\noalign{\hrule}
&\hfill split $I_n^*$&&$-1$&&$0$&&$5+n$&&$0$&\cr
\noalign{\hrule}
&\hfill non-split $I_n^*$&&$-1$&&$0$&&$4+n$&&$1$&\cr
\noalign{\hrule}
&\hfill $II$&&$-1$&&$0$&&$1$&&$0$&\cr
\noalign{\hrule}
&\hfill $II^*$&&$-1$&&$0$&&$9$&&$0$&\cr
\noalign{\hrule}
&\hfill $III$&&$-1$&&$0$&&$2$&&$0$&\cr
\noalign{\hrule}
&\hfill $III^*$&&$-1$&&$0$&&$8$&&$0$&\cr
\noalign{\hrule}
&\hfill split $IV$&&$-1$&&$0$&&$3$&&$0$&\cr
\noalign{\hrule}
&\hfill non-split $IV$&&$-1$&&$0$&&$2$&&$1$&\cr
\noalign{\hrule}
&\hfill split $IV^*$&&$-1$&&$0$&&$7$&&$0$&\cr
\noalign{\hrule}
&\hfill non-split $IV^*$&&$-1$&&$0$&&$3$&&$4$&\cr
\noalign{\hrule}
}}
$$
\end{prop}

\begin{exer}
  Use an elementary point-counting argument to verify the proposition.
  In particular, check that the number of components of $\pi^{-1}(y)$
  that are rational over $\kappa_y$ is $f$ and that the order of pole
  at $T=q_y^{-1}$ of
$${Z(\pi^{-1}(y),T)}{(1-T)(1-q_yT)}$$
is $f-1$.
\end{exer}

Using the Proposition and the definition of the $L$-function
(in Lecture~1, equation~(\ref{eq:Ldef})) we find that
\begin{equation}\label{eq:Z-L}
L(E,T)=\frac{Z(\CC,T)Z(\CC,qT)}{Z(\EE,T)}
   \prod_{\text{bad }v}
   \frac{(1-T)^{a_v+1}(1+T)^{b_v}}{(1-q_vT^{\deg(v)})^{f_v-1}(1+q_vT^{\deg(v)})^{g_v}}
\end{equation}
where $a_v$, $b_v$, $f_v$ and $g_v$ are the invariants defined in the
Proposition at the place $v$.  Using the Weil conjectures
(see~Section~\ref{s:zetas} of Lecture~0), we see that the orders of
$L(E,s)$ and $\zeta(\EE,s)$ at $s=1$ are related as follows:
\begin{equation}\label{eq:Z-L-ords}
\ord_{s=1}L(E,s)=-\ord_{s=1}\zeta(\EE,s)
-2-\sum_v(f_v-1).
\end{equation}

\begin{rem}
  This simple approach to evaluating the order of zero of the
  $L$-function does not yield the important fact that $L(E,T)$ is a
  polynomial in $T$ when $E$ is non-constant, nor does it yield the
  Riemann hypothesis for $L(E,T)$.
\end{rem}

For a slightly more sophisticated (and less explicit) comparison of
$\zeta$-functions and $L$-functions in a more general context, see
\cite{Gordon79}.

\section{The Tate-Shafarevich and Brauer groups}
The last relationship between $E$ and $\EE$ we need concerns the
Tate-Shafarevich and Brauer groups.

\begin{thm}\label{thm:Sha-Br}
Suppose that $E$ is an elliptic curve over $K=k(\CC)$ and $\EE\to\CC$
is the associated elliptic surface as in
Proposition~\ref{prop:model}.  Then there is a canonical isomorphism
$$\Br(\EE)\cong\sha(E/K).$$
\end{thm}

The proof of this result, which is somewhat involved, is given in
\cite{Grothendieck68}*{Section~4}.  The main idea is simple enough:
one computes $\Br(\EE)=H^2(\EE,\G_m)$ using the morphism
$\pi:\EE\to\CC$ and a spectral sequence.  Using that the Brauer group
of a smooth, complete curve over a finite field vanishes, one finds
that the main term is $H^1(\CC,R^1\pi_*\G_m)$.  Since $R^1\pi_*\G_m$
is the sheaf associated to the relative Picard group, it is closely
related to the sheaf on $\CC$ represented by the N\'eron model of $E$.
This provides a connection with the Tate-Shafarevich group which leads
to the theorem.

See \cite{UlmerCRM} for more details about this and the closely
related connection between $H^2(\EEbar,\Zl(1))^{G_k}$ and the
$\ell$-Selmer group of $E$.

\section{The main classical results}
We are now in a position to prove the theorems of
Section~\ref{s:results} of Lecture~1.  For convenience, we restate
Theorem~\ref{thm:BSD1} and a related result.

\begin{thm}\label{thm:BSD-Tate}
Suppose that $E$ is an elliptic curve over $K=k(\CC)$ and $\EE\to\CC$
is the associated elliptic surface as in
Proposition~\ref{prop:model}.  
\begin{enumerate}
\item BSD holds for $E$ if and only if $T_2$ holds for $\EE$.
\item
$\rk E(K)\le\ord_{s=1}L(E,s)$.
\item
The following are equivalent:
\begin{itemize}
\item $\rk E(K)=\ord_{s=1}L(E,s)$
\item $\sha(E/K)$ is finite
\item for any one prime number $\ell$ \textup{(}$\ell=p$ is
  allowed\textup{)}, the $\ell$-primary part $\sha(E/K)_{\ell^\infty}$
  is finite.
\end{itemize}
\item
If $K'/K$ is a finite extension and if the BSD conjecture holds
  for $E$ over $K'$, then it holds for $E$ over $K$.
\end{enumerate}
\end{thm}

\begin{proof}
Comparing (\ref{eq:STformula}) and (\ref{eq:Z-L-ords}), we have that
$$\rk E(K)-\ord_{s=1}L(E,s)=\rk NS(\EE)+\ord_{s=1}\zeta(\EE,s).$$
Since BSD is the assertion that the left hand side is zero and $T_2$
is the assertion that the right hand side is zero, these conjectures
are equivalent.

By Theorem~\ref{prop:T-ineqs} of Lecture~2, the right hand side is
$\le0$ and therefore so is the left.  This gives the inequality $\rk
E(K)\le\ord_{s=1}L(E,s)$.

The statements about $\sha(E/K)$ follow from Theorem~\ref{thm:Sha-Br}
($\sha(E/K)\cong\Br(\EE)$), the equivalence of BSD and $T_2(\EE)$, and
Theorem~\ref{thm:T1-Br} of Lecture~2.

The last point follows from the equivalence of BSD and $T_2(\EE)$ and
Proposition~\ref{prop:T-descent} of Lecture~2.
\end{proof}

\begin{proof}[Proofs of Theorems~\ref{thm:BSD2} and
  \ref{thm:BSD-low-height} of Lecture~1]
  Theorem~\ref{thm:BSD2} of Lecture~1 concerns isotrivial elliptic
  curves.  By the last point of Theorem~\ref{thm:BSD-Tate} above, it
  suffices to show that BSD holds for constant curves.  But if $E$ is
  constant, then $\EE$ is a product of curves, so the Tate conjecture
  for $\EE$ follows from Theorem~\ref{thm:products} of Lecture~2.  The
  first point of Theorem~\ref{thm:BSD-Tate} above then gives BSD for
  $E$.

  Theorem~\ref{thm:BSD-low-height} of Lecture~1 concerns elliptic
  curves over $k(t)$ of low height.  By the discussion in
  Section~\ref{s:height}, if $E/k(t)$ has height $\le2$ then $\EE$ is
  a rational or K3 surface.  (Strictly speaking, this is true only
  over a finite extension of $k$, but the last point of
  Theorem~\ref{thm:BSD-Tate} allows us to make this extension without
  loss of generality.)  But $T_2(\XX)$ for a rational surface follows
  from Proposition~\ref{prop:T-DPC} of Lecture~2.  For $E$ such that
  $\EE$ is a K3 surfaces, Artin and Swinnerton-Dyer proved the
  finiteness of $\sha(E/K)$ (and therefore BSD) in
  \cite{ArtinSwinnertonDyer73}.
\end{proof}

\section{Domination by a product of curves}
Combining part 1 of Theorem~\ref{thm:BSD-Tate} with
Proposition~\ref{prop:T-DPC} of Lecture~2,
we have the following.
\begin{thm}\label{thm:DPC-BSD}
Let $E$ be an elliptic curve over $K$ with associated surface
$\EE$.  If $\EE$ is dominated by a product of curves, then BSD holds
for $E$.
\end{thm}

Theorem~\ref{thm:4-monos} (``four monomials'') and Berger's theorem
\ref{thm:Berger} are both corollaries of Theorem~\ref{thm:DPC-BSD}, as
we will explain in the remainder of this lecture.

\section{Four monomials}
We recall Shioda's conditions.  Suppose that $f\in R=k[x_1,x_2,x_3]$
is the sum of exactly four non-zero monomials:
$$f=\sum_{i=1}^4c_i\prod_{j=1}^3 x_j^{e_{ij}}$$
where $c_i\in k$ and the $e_{ij}$ are non-negative integers.  Let
$e_{i4}=1-\sum_{j=1}^3e_{ij}$ and form the $4\times4$ matrix
$A=(e_{ij})$.  Assuming that $\det(A)\neq0$ (in $\Z$), let $\delta$ be
the smallest positive integer such that there is a $4\times4$ integer
matrix $B$ with $AB=\delta I_{4\times4}$.  We say that $f$ {\it
  satisfies Shioda's 4-monomial condition\/} if $\delta\neq0$ in $k$,
i.e., if $p\nodiv\delta$.  The following exercise shows that this is
equivalent to the definition in Lecture~1.

\begin{exers}
  Show that a prime $\ell$ divides $\delta$ if and only if it divides
  $\det(A)$.  Show that if we change the definition of $e_{i4}$ to
  $e_{i4}=d-\sum_{j=1}^3e_{ij}$ for some other non-zero integer $d$
  and define $\delta_d$ using the new $A=(e_{ij})$, then $\delta_1$
  divides $\delta_d$ for all $d$.  I.e., $d=1$ is the optimal choice
  to minimize $\delta$.
\end{exers}
  
\begin{exer}\label{exer:change-of-coords}
  With $c_i$ and $e_{ij}$ as above, show that the system of equations
$$\prod_{j=1}^4d_{j}^{e_{ij}}=c_i^{-1}\qquad i=1,\dots,4$$
has a solution with $d_{j}\in\Fqbar$, $j=1,\dots,4$.
\end{exer}

\begin{proof}[Proof of Theorem~\ref{thm:4-monos} of Lecture~1]
  Briefly, the hypotheses imply that the associated elliptic surface
  $\EE\to\P^1$ is dominated by a Fermat surface (of degree $\delta$)
  and thus by a product of Fermat curves (of degree $\delta$).  Thus
  Theorem~\ref{thm:DPC-BSD} implies that BSD holds for $E$.

  In more detail, note that $\EE$ is birational to the affine surface
  $V(f)\subset\A^3_k$.  So it will suffice to show that $V(f)$ is
  dominated by a product of curves.  To that end, it will be convenient
  to identify $k[t,x,y]$ and $R=k[x_1,x_2,x_3]$ by sending $t\mapsto
  x_1$, $x\mapsto x_2$ and $y\mapsto x_3$, so that $f$ becomes
  $$f=\sum_{i=1}^4c_i\prod_{j=1}^3 x_j^{e_{ij}}.$$

  Exercise~\ref{exer:change-of-coords} implies that, after extending
  $k$ if necessary, we may change coordinates ($x_j\mapsto d_jx_j$) so
  that the coefficients $c_i$ are all $1$.  Then the matrix $A$
  defines rational a map $\phi$ from $V(f)$ to the Fermat surface of
  degree $1$ 
$$F^2_1 =\{y_1+y_2+y_3+y_4=0\}\subset\P^3_k,$$ 
namely $\phi^*(y_i)=\prod_{j=1}^4 x_j^{e_{ij}}$.  Similarly, the
matrix $B$ defines a rational map $\psi$ from the Fermat surface of
degree $\delta$
$$F^2_\delta=\{z_1^\delta+z_2^\delta+z_3^\delta+z_4^\delta=0\}
\subset\P^3_k$$
to $V(f)$, namely $\psi^*(x_i)=\prod_{j=1}^4 z_j^{B_{ij}}$.  The
composition of these maps is the standard projection from $F^2_\delta$
to $F^2_1$, namely $y_i\mapsto z_i^\delta$ and so both maps are
dominant.

  Finally, Shioda and Katsura \cite{ShiodaKatsura79} showed that
  $F^2_\delta$ is dominated by the product of Fermat curves
  $F^1_\delta\times F^1_\delta$.  Thus, after extending $k$, $\EE$ is
  dominated by a product of curves and Theorem~\ref{thm:DPC-BSD}
  finishes the proof.
\end{proof}

As we will explain below, this Theorem can be combined with results on
analytic ranks to give examples of elliptic curves over $\Fp(t)$ with
arbitrarily large Mordell-Weil rank.  (In fact, similar ideas can be used to
produce Jacobians of every dimension with large rank.  For this, see
\cite{Ulmer07b} and also \cite{UlmerCRM}.)  

Unfortunately, Theorem~\ref{thm:4-monos} is very rigid---as one sees
in the proof, varying the coefficients in the 4-nomial $f$ does not
vary the isomorphism class of $\EE$ over $\Fqbar$ and so we get only
finitely many non-isomorphic elliptic curves over $\Fpbar(t)$.
Berger's construction, explained in the next subsection, was motivated
by a desire to overcome this rigidity and give {\it families\/} of
examples of curves where one knows the BSD conjecture.

\section{Berger's construction}\label{s:Berger}
Berger gave a much more flexible construction of surfaces that are
dominated by a product of curves in a tower.  More precisely, we note
that if $\EE\to\P^1$ is an elliptic surface and $\phi:\P^1\to\P^1$ is
the morphism with $\phi^*(t)=u^d$ (corresponding to the field
extension $k(u)/k(t)$ with $u^d=t$), then it is not in general the
case that the base changed surface
$$\xymatrix{\EE'=\EE\times\P^1_k\ar[r]\ar[d]&\P^1_k\ar[d]\\
  \P^1_k\ar^\phi[r]&\P^1_k}$$ 
is dominated by a product of curves.  Berger's construction gives a
rich class of curves for which DPC {\it does\/} hold in every layer of
a tower of coverings.  We restate Theorem~\ref{thm:berger} from
Lecture~1 in a slightly different (but visibly equivalent) form.

\begin{thm}\label{thm:Berger}
Let $E$ be an elliptic curve over $K=k(t)$ and assume that there are
rational functions $f(x)$ and $g(y)$ on $\P^1_k$ such that $E$ is
birational to the curve $V(f(x)-tg(y))\subset\P^1_K\times\P^1_K$.
Then the BSD conjecture holds for $E$ over the field $k(u)=k(t^{1/d})$
for all $d$ prime to $p$.
\end{thm}

\begin{proof}
  Clearing denominators we may interpret $f(x)-tg(y)$ as defining a
  hypersurface $\XX$ in the affine space $\A^3$ with coordinates $x$,
  $y$, and $t$ and it is clear that the elliptic surface $\EE\to\P^1$
  associated to $E$ is birationally isomorphic to $\XX$.  On the other
  hand, $\XX$ is visibly birational to $\P^1\times\P^1$ since we may
  eliminate $t$.  Thus $\XX$ and $\EE$ are dominated by a product of
  curves.  This checks the case $d=1$.

  For larger $d$, note that the elliptic surface $\EE_d\to\P^1$
  associated to $E/k(u)$ is birational to the hypersurface $\XX_d$ in
  $\A^3_k$ defined by $f(x)-u^dg(y)$.  Berger showed by a fundamental
  group argument, generalizing \cite{Schoen96}, that $\XX_d$ is
  dominated by a product of curves, more precisely, by a product of
  covers of $\P^1$.  (For her argument to be correct, $\pi_1$ should
  be replaced by the prime-to-$p$ fundamental group $\pi_1^{p'}$
  throughout.)  This was later made more explicit in \cite{UlmerDPCT},
  where it was observed that $\XX_d$ is dominated by a product of two
  explicit covers of $\P^1$.

  More precisely, let $\CC_d$ and $\DD_d$ be the covers of $\P^1_k$
  defined by $z^d=f(x)$ and $w^d=g(y)$.  Then there is a rational map
  from $\CC_d\times\DD_d$ to the hypersurface $\XX_d$, namely
$$(x,z,y,w)\mapsto (x,y,u=z/w).$$
This is clearly dominant and so $\XX_d$ and $\EE$ are dominated by
products of curves.

Applying Theorem~\ref{thm:DPC-BSD} finishes the proof.
\end{proof}

Note that there is a great deal of flexibility in the choice of data
for Berger's construction.  As an example, take $f(x)=x(x-a)/(x-1)$
and $g(y)=y(y-1)$ where $a\in\Fq$ is a parameter.  Then if $a\neq1$,
the curve $f(x)=tg(y)$ in $\P^1\times\P^1$ has genus 1 and a rational
point.  A simple calculation shows that it is birational to the
Weierstrass cubic
$$y^2+txy-ty=x^3-tax^2+t^2ax.$$
Theorem~\ref{thm:Berger} implies that this curve satisfies the BSD
conjecture over $\F_{q^n}(t^{1/d})$ for all $n$ and all $d$ prime to
$p$.  Varying $q$ and $a$ we get infinitely many curves for which BSD
holds at every layer of a tower.

We will give more examples and discuss further applications of the idea
behind Berger's construction in Lectures~4 and 5.


\lecture{Unbounded ranks in towers}

In order to prove results on analytic ranks in towers, we need a more
sophisticated approach to $L$-functions.  In this lecture we explain
Grothendieck's approach to $L$-functions over function fields and then
use it and a new linear algebra lemma to find elliptic curves with
unbounded analytic and algebraic ranks in towers of function fields.

\section{Grothendieck's analysis of $L$-functions}
\numberwithin{equation}{subsection}
\subsection{Galois representations}\label{ss:gal-reps}
As usual, we let $K=k(\CC)$ be the function field of a curve over a
finite field $k$ and $G_K=\gal(K^{sep}/K)$ its Galois group.  As in
Lecture~0, Section~\ref{s:ffs}, we write $D_v$, $I_v$, and $\Fr_v$ for
the decomposition group, inertia group, and (geometric) Frobenius at a
place $v$ of $K$.

We fix a prime $\ell\neq p$ and consider a representation
\begin{equation}\label{eq:rho}
\rho:G_K\to\GL(V)\cong\GL_n(\Qlbar)
\end{equation}
on a finite-dimensional $\Qlbar$ vector space.
We make several standing assumptions about $\rho$.

First, we always assume $\rho$ is continuous and unramified away from
a finite set of places of $K$.  By a compactness argument (see
\cite{KatzSarnakRMFEM}*{9.0.7}) , it is
possible to define $\rho$ over a finite extension $L$ of $\Ql$, i.e.,
there is a representation
$$\rho':G_K\to\GL_n(L)$$
isomorphic to $\rho$ over $\Qlbar$.  Nothing we say will depend on the
field of definition of $\rho$ and we will generally not distinguish
between $\rho$ and isomorphic representations defined over subfields
of $\Qlbar$.

We also always assume that $\rho$ is pure of integral weight $w$,
i.e., for all $v$ where $\rho$ is unramified, the eigenvalues of
$\rho(\Fr_v)$ are Weil numbers of size $q_v^{w/2}$.

Finally, we sometimes assume that $\rho$ is ``symplectically self-dual
of weight $w$.''  This means that on the space $V$ where $\rho$ acts,
there is an $G_K$-equivariant, alternating pairing with values in
$\Qlbar(-w)$.

\subsection{Conductors}
The Artin conductor of $\rho$ is a divisor on $\CC$ (a formal sum of
places of $K$) and is a measure of its ramification.  We write
$\cond(\rho)=\n=\sum_vn_v[v]$.  To define the local coefficients, fix
a place $v$ of $K$ and let $G_i\subset I_v$ be the higher ramification groups
at $v$ (in the lower numbering).  Then define
$$n_v=\sum_{i=0}^\infty\frac1{[G_0:G_i]}\dim V/V^{G_i}.$$
Here $V^{G_i}$ denotes the subspace of $V$ invariant under $G_i$.  It
is clear that $n_v=0$ if and only if $\rho$ is unramified at $v$.  If
$\rho$ is tamely ramified at $v$ (i.e., $G_1$ acts trivially), then
$n_v=\dim V/V^{G_0}=\dim V/V^{I_v}$.  In general, the first term of
the sum above is the {\it tame conductor\/} and the rest of the sum is
the {\it Swan conductor\/}.  We refer to \cite{MilneEC}*{V.2} and also
\cite{SerreLRFG}*{\S19} for an alternative definition and more
discussion about the conductor, including the fact that the local
coefficients $n_v$ are integers.

\subsection{$L$-functions}
Let us fix an isomorphism $\Qlbar\cong\C$ so that we may regard
eigenvalues of Frobenius on $\ell$-adic representations as complex
numbers.  Having done this, a representation~(\ref{eq:rho}) gives rise
to an $L$-function, defined as an Euler product:
\begin{equation}\label{eq:L-def}
L(\rho,T)=\prod_v\det\left(1-T\Fr_v|V^{I_v}\right)
\end{equation}
and $L(\rho,s)=L(\rho,q^{-s})$.  The product is over the places of
$K$, the exponent $I_v$ denotes the subspace of elements invariant
under the inertia group $I_v$, and $\Fr_v$ is a Frobenius element at
$v$.

Because of our assumption that $\rho$ is pure of weight $w$, the
product defining $L(\rho,s)$ converges absolutely and defines a
holomorphic function in the region $\RP s>w/2+1$.

It is clear from the definition that if $\rho$ and $\sigma$ are Galois
representations then $L(\rho\oplus\sigma,s)=L(\rho,s)L(\sigma,s)$ and
$L(\rho(n),s)=L(\rho,s-n)$.

It is also clear that $L(\rho_{triv},s)=\zeta(\CC,s)$. and so
$L(\rho_{triv}(n),s)=\zeta(\CC,s-n)$.   

\begin{exer}
  Prove that if $\rho$ factors through $G_K\to G_k$, so that $\Fr_v$
  goes to $\alpha^{\deg v}$, then
$$L(\rho,T)=Z(\CC,\alpha T)$$
is a twisted version of the zeta function of $\CC$.  Compare with
Exercise~\ref{exer:L-const} of Lecture~1.  Note that a representation
factors through $G_K\to G_k$ if and only if it is trivial on $G_{\kbar
  K}$, so this exercise fills in the missing cases in the following theorem.
\end{exer}

\begin{thm}\label{thm:GA}
  Suppose that $\rho$ is a representation of $G_K$ \textup{(}satisfying the
  standing hypotheses of Subsection~\ref{ss:gal-reps}\textup{)} that
  contains no copies of the trivial representation when restricted to
  $G_{\kbar K}$.  Then there is a canonically defined $\Qlbar$-vector
  space $H(\rho)$ with continuous $G_k$ action such that
$$L(\rho,s)=\det\left(1-q^{-s}\Fr_q|H(\rho)\right).$$
The dimension of $H(\rho)$ is $\deg(\rho)(2g_\CC-2)+\deg\n$ where $\n$
is the conductor of $\rho$.
\end{thm}

\begin{proof}
  (Sketch) The representation $\rho:G_K\to\GL(V)$ gives rise to a
  constructible sheaf $\FF_\rho$ on $\CC$.  In outline: $\rho$ is
  essentially the same thing as a lisse sheaf $\FF_U$ on the open
  subset $j:U\into\CC$ over which $\rho$ is unramified.  We defined
  $\FF_\rho$ as the push-forward $j_*\FF_U$.  For each closed point
  $v$ of $\CC$, the stalk of $\rho$ at $v$ is $V^{I_v}$.

  Let $H^i(\CCbar,\FF)$ be the \'etale cohomology groups of $\FF$.
  They are finite dimensional $\Qlbar$ vector spaces and give
  continuous representations of $G_k$.

  The Grothendieck-Lefschetz fixed point formula says that for each
  finite extension $\Fqn$ of $k\cong\Fq$, we have
$$\sum_{x\in\CC(\Fqn)} \tr(Fr_x|\FF_x)=
\sum_{i=0}^2(-1)^i\tr\left(Fr_{q^n}|H^i(\CCbar,\FF)\right).$$ 
On the left hand side, the sum is over points of $\CC$ with
values in $\Fqn$ and the summand is the trace of the action of the
Frobenius at $x$ on the stalk of $\FF$ at a geometric point over $x$.

Multiplying both sides by $T^n/n$, summing over $n\ge1$, and exponentiating,
one finds that
$$L(\rho,T)=
\prod_{i=0}^2\det\left(1-T\Fr_q|H^i(\CCbar,\FF)\right)^{(-1)^{i+1}}.$$ 

Now $H^0(\CCbar,\FF)$ and $H^2(\CCbar,\FF)$ are isomorphic
respectively to the invariants and coinvariants of $V$ under $G_{\kbar
  K}$ and so under our hypotheses on $\rho$, $H^i(\CCbar,\FF)$
vanishes for $i=0,2$.  Thus we have
$$L(\rho,s)=\det\left(1-q^{-s}\Fr_q|H(\rho)\right)$$
where $H(\rho)=H^1(\CCbar,\FF)$.

The dimension formula comes from an Euler characteristic formula
proven by Raynaud and sometimes called the
Grothendieck-Ogg-Shafarevich formula.  It says
$$\sum_{i=0}^2(-1)^i\dim H^i(\CCbar,\FF)
=\deg(\rho)(2-2g_\CC)-\deg(\cond(\rho)).$$ 
Since $H^0$ and $H^2$ vanish, this gives the desired dimension formula.
\end{proof}

Obviously we have omitted many details.  I recommend
\cite{MilneEC}*{V.1 and V.2} as a compact and readable source for
several of the key points, including passing from $\ell$-torsion
sheaves to $\ell$-adic sheaves, the conductor, and the
Grothendieck-Ogg-Shafarevich formula.  See \cite{MilneEC}*{VI.13} for
the Grothendieck-Lefschetz trace formula.  

\begin{rem-exer}
If we are willing to use a virtual representation of $G_k$ in place of
a usual representation, then the Theorem has a more elegant
restatement which avoids singling out representations that are
trivial when restricted to $G_{\kbar K}$.  State and prove this
generalization.
\end{rem-exer}

\begin{exer}\label{exer:Artin}
  Check that we have the Artin formalism formula: if $F/K$ is a finite
  separable extension and $\rho$ is a representation of $G_F$, then
$$L(\rho,s)=L(\ind^{G_K}_{G_F}\rho,s).$$
Note that the left hand side is an Euler product on $F$ with almost
all factors of some degree, say $N$, whereas the right hand side is an
Euler product on $K$, with almost all factors of degree $N[F:K]$.  The
equality can be taken to be an equality of Euler products, where that
on the left is grouped according to the places of $K$.
\end{exer}

\subsection{Functional equation and Riemann hypothesis}
Theorem~\ref{thm:GA} shows that the $L$-function of $\rho$ has an
analytic continuation to the entire $s$ plane (meromorphic if we allow
$\rho$ to have trivial factors over $\kbar K$).  In this section we
deduce other good analytic properties of $L(\rho,s)$.

\begin{thm} Suppose \textup{(}in addition to the standing
  hypotheses\textup{)} that $\rho$ is symplectically self-dual of
  weight $w$.  Then $L(\rho,s)$ satisfies a functional equation
$$L(\rho,w+1-s)=\pm q^{N(s-(w+1)/2)}L(\rho,s)$$
where $N=(2g_\CC-2)\deg(\rho)+\deg(\cond(\rho))$.  
The zeroes of $\rho$ lie on the line $\RP s=(w+1)/2$.
\end{thm}

\begin{proof}
  (Sketch) We use the notation of the proof of Theorem~\ref{thm:GA}.
  The functional equation comes from a symmetric pairing
$$H(\rho)\times H(\rho)\to H^2(\CCbar,\Qlbar(-w))\cong\Qlbar(-w-1).$$
(Symmetric because $\rho$ is skew-symmetric and $H=H^1$.)  That there
is such a pairing is not as straightforward as it looks, because we
defined the sheaf $\FF$ as a push forward $j_*\FF_U$ where
$j:U\into\CC$ is a non-empty open set over which $\rho$ is unramified
and $\FF_U$ is the lisse sheaf on $U$ corresponding to $\rho$.  It is
well-known that $j^*$ identifies $H^1(\CCbar,\FF)$ with the image of
the ``forget supports'' map
$$H^1_c(\overline{U},\FF_U)\to H^1(\overline{U},\FF_U)$$
from compactly supported cohomology to usual cohomology.  (This is
often stated, but the only proof I know of in the literature is
\cite{Ulmer05}*{7.1.6}.)  The cup product
$$H^1_c(\overline{U},\FF_U)\times H^1(\overline{U},\FF^*_U)\to
H^2_c(\overline{U},\Qlbar)\cong\Qlbar(-1)$$
then induces a pairing on $H^1(\CCbar,\FF)$ via the above
identification.  Poincar\'e duality shows that the pairing is
non-degenerate and so $H(\rho)$ is orthogonally self-dual of
weight $w+1$.  

The location of the zeroes is related to the eigenvalues of Frobenius
on $H(\rho)=H^1(\CCbar,\FF)$ and these are Weil numbers of size
$q^{w+1}$ by Deligne's purity theorem \cite{Deligne80}.  I recommend
the Arizona Winter School 2000 lectures of Katz (published as
\cite{Katz01}) for a streamlined proof of Deligne's theorem in the
generality needed here.
\end{proof}

\section{The case of an elliptic curve}
Next, we apply the results of the previous section to elliptic
curves.  Throughout, $E$ will be an elliptic curve over a function
field $K=k(\CC)$ over a finite field $k$ of characteristic $p$.

\subsection{The Tate module}
We consider the Tate module of $E$.  More precisely, fix a prime
$\ell\neq p$ and let
$$T_\ell E=\varprojlim_n E(\Kbar)[\ell^n]
\quad\text{and}\quad V_\ell E=T_\ell E\tensor_\Zl \Ql.$$ 
Let $\rho_E$ be the representation of $G_K$ on the dual vector space
$V_\ell^*=\Hom(V_\ell E,\Ql)\cong H^1(\overline{E},\Ql)$.  Then
$\rho_E$ is two-dimensional and continuous and (by the criterion of
Ogg-N\'eron-Shafarevich, see \cite{SerreTate68}*{Thm.~1}) it is
unramified outside the (finite) set of places where $E$ has bad
reduction.

At every place $v$ of $K$ where $E$ has good reduction, we have
$$\det(1-\rho(\Fr_v)T)=1-a_vT+q_vT^2$$
where $a_v$ is defined as in (\ref{eq:a_v}) by
$\#E_v(\kappa_v)=1-a_v+q_v$.  This follows from the smooth base change
theorem \cite{MilneEC}*{VI.4} and the cohomological description of the
zeta function of the reduction, as in Section~\ref{s:cohomology} of
Lecture~0.  Thus $\rho$ is pure of weight $w=1$.

The Weil pairing induces an alternating, $G_k$-equivariant pairing
$V_\ell E\times V_\ell E\to\Ql(-1)$ and so $\rho$ is symplectically
self-dual of weight 1.

If $E$ is constant, then $\rho_E$ factors through $G_K\to G_k$ and
since $G_k$ is abelian, $\rho_E$ is the direct sum of two characters.
More precisely, if $E\cong E_0\times_kK$ and
$1-aT+qT^2=(1-\alpha_1T)(1-\alpha_2T)$ is the numerator of the
$Z$-function of $E_0$, then $\rho_E$ is the sum of the two characters
that send $\Fr_v$ to $\alpha_i^{\deg v}$.

If $E$ is non-isotrivial, then $\rho_E$ restricted to $G_{\kbar K}$
has no trivial subrepresentations.  One way to see this is to use a
slight generalization of the MWLN theorem, according to which $E(\kbar
K)$ is finitely generated (when $E$ is non-isotrivial).  Thus its
$\ell$-power torsion is finite and this certainly precludes a trivial
subrepresentation in $\rho|_{G_{\kbar K}}$.  In fact, by a theorem of
Igusa \cite{Igusa59}, $\rho|_{G_{\kbar
    K}}$ is contains an open subgroup of $\SL_2(\Z_\ell)$ so is
certainly irreducible, even absolutely irreducible.

\begin{exer}
  Show that if $E$ is isotrivial but not constant, then $\rho_E$
  restricted to $G_{\kbar K}$ has no trivial subrepresentation.  Hint:
  $E$ is a twist of a contant curve $E'=E_0\times_k K$.  Relate the
  action of $G_K$ on the Tate module of $E$ to its action on that of
  $E'$ and show that there exists an element $\sigma\in G_{\kbar K}$
  that acts on $V_\ell E$ via a non-trivial automorphism of $E$.  But
  a non-trivial automorphism has only finitely many fixed points.
\end{exer}

We can summarize this discussion as follows.

\begin{prop}
Let $\rho$ be the action of $G_K$ on the Tate module $V_\ell E$ of
$E$.  Then $\rho$ is continuous, unramified outside a finite set of
places of $K$, and is pure and symplectically self-dual of weight $1$.
If $E$ is non-constant, then $\rho|_{G_{\kbar K}}$ has no trivial
subrepresentations. 
\end{prop}

The conductor of $\rho_E$ as defined in the previous section is equal
to the conductor of $E$ as mentioned in Section~\ref{s:local-invs} of
Lecture~1.  This was proven by Ogg in \cite{Ogg67}.

\subsection{The $L$-function}
Applying the results of the previous section, we get a very
satisfactory analysis of the $L$-function of $E$.  Since we know
everything about the constant case by an elementary analysis
(cf.~exercise~\ref{exer:L-const} of Lecture~1), we restrict to the
non-constant case.

\begin{thm}
  Let $E$ be a non-constant elliptic curve over $K=k(\CC)$ and let $q$
  be the cardinality of $k$.  Let $\n$ be the conductor of $E$.  Then
  $L(E,s)$ is a polynomial in $q^{-s}$ of degree
  $N=4g_\CC-4+\deg(\n)$. Its inverse roots are Weil numbers of size
  $q$ and it satisfies a functional equation
$$L(E,2-s)=\pm q^{N(s-1)} L(E,s).$$
\end{thm}

Combining the Theorem with Theorem~\ref{thm:BSD1}, we obtain the
following.

\begin{cor}\label{cor:rankbound}
The rank of $E(K)$ is bounded above by $N=4g_\CC-4+\deg(\n)$.  If
equality holds, then $L(E,s)=(1-q^{1-s})^N$.
\end{cor}

The sign in the functional equation can be computed as a product of
local factors.  This can be seen via the connection with automorphic
forms (a connection which is outside the scope of these lectures)
or, because we are in the function field situation, directly via
cohomological techniques.  See \cite{Laumon84} for the latter.

\section{Large analytic ranks in towers}

\subsection{Statement of the theorem}\label{ss:analytic-ranks}
We give a general context in which one obtains large analytic
ranks by passing to layers of a suitable tower of function fields.

As usual, let $p$ be a prime and $q$ a power of $p$.  Let $K=\Fq(t)$,
for each $d$ not divisible by $p$, set $F_d=\Fq(t^{1/d})\cong\Fq(u)$,
and $K_d=\Fq(\mu_d)(t^{1/d})\cong\Fq(\mu_d)(u)$.  

Suppose that $E$ is an elliptic curve over $K$.  Let $\n$ be the
conductor of $E$ and let
$$\n'=\n-\dim (V_\ell E/V_\ell E^{I_0})[0]-\dim (V_\ell E/V_\ell E^{I_\infty})[\infty].$$
This is the conductor of $E$ except that we have removed the tame part
at $t=0$ and $t=\infty$.

\begin{thm}\label{thm:analytic-ranks}
  Let $E$ be an elliptic curve over $K$ and define $\n'$ as above.
  Suppose that $\deg\n'$ is odd.  Then the analytic rank of $E$ over
  $F_d$ \textup{(}and $K_d$\textup{)} is unbounded as $d$ varies.
  More precisely, there exists a constant $c$ depending only on $E$
  such that if $d$ has the form $d=q^n+1$, then
$$\ord_{s=1}L(E/F_d,s)\ge \frac{d}{2n}-c=\frac{q^n+1}{2n}-c.$$
and
$$\ord_{s=1}L(E/K_d,s)\ge d-c=q^n+1-c$$
\end{thm}

This theorem is proven in detail in \cite{Ulmer07b}*{\S2-4}.  We will
sketch the main lines of the argument below.

\subsection{A linear algebra lemma}
Our analytic rank results ultimately come from the following
odd-looking result of linear algebra.

\begin{prop}\label{prop:la}
  Let $V$ be a finite-dimensional vector space with subspaces $W_i$
  indexed by $i\in\Z/a\Z$ such that $V=\oplus_{i\in\Z/a\Z}W_i$.  Let
  $\phi:V\to V$ be an invertible linear transformation such that
  $\phi(W_i)= W_{i+1}$ for all $i\in\Z/a\Z$.  Suppose that $V$ admits
  a non-degenerate, $\phi$-invariant symmetric bilinear form $\< ,\>$.
  Suppose that $a$ is even and $\<,\>$ induces an isomorphism
  $W_{a/2}\cong W_0^*$ \textup{(}the dual vector space of
  $W_0$\textup{)}.  Suppose also that $N=\dim W_0$ is odd.  Then the
  polynomial $1- T^{a}$ divides $\det(1-\phi T|V)$.
\end{prop}

We omit the proof of this proposition, since it is not hard and it
appears in two forms in the literature already.  Namely, embedded in
\cite{Ulmer05}*{7.1.11ff} is a matrix-language proof of the
proposition, and a coordinate-free proof is given in
\cite{Ulmer07b}*{\S2}.

\subsection{Sketch of the proof of Theorem~\ref{thm:analytic-ranks}}
For simplicity, we assume that $E$ is non-isotrivial.  (If $p>3$ and
$E$ is isotrivial, then the theorem is vacuous because all of the
local conductor exponents $n_v$ are even.)  Let $\rho$ be the
representation of $G_K$ on $V=H^1(\overline{E},\Ql)=(V_\ell E)^*$ and
let $\rho_d$ be the restriction of $\rho$ to $G_{F_d}$.  Then by
Grothendieck's analysis, we have
$$L(E/F_d,s)=\det\left(1-\Fr_qq^{-s}|H(\rho_d)\right).$$
Here $H(\rho_d)$ is an $H^1$ on the rational curve whose function field
is $F_d=\Fqbar(u)=\Fqbar(t^{1/d})$. 

The projection formula in cohomology (a parallel of the Artin
formalism \ref{exer:Artin}) implies that
$$H(\rho_d)\cong H(\ind_{G_{F_d}}^{G_K}\rho)\cong 
H(\rho\tensor\ind_{G_{F_d}}^{G_K}{\bf 1})$$
where $\bf 1$ denotes the trivial representation.  Since the
cohomology $H$ is computed on $\overline{\P}^1_u$ (the $\P^1$ with
coordinate $u$, with scalars extended to $\Fqbar$) and
$\overline{\P}^1_u\to\overline{\P}^1_t$ is Galois with group $\mu_d$, we
have
$$H(\rho_d)\cong\bigoplus_{j=0}^{d-1}H(\rho\tensor\chi^j)$$
where $\chi$ is a character of $\gal(\Fqbar(u)/\Fqbar(t))$ of order
exactly $d$.

Now the decomposition displayed above is not preserved by Frobenius.
Indeed $\Fr_q$ sends $H(\rho\tensor\chi^j)$ to
$H(\rho\tensor\chi^{qj})$.  Thus we let $o\subset\Z/d\Z$ denote an
orbit for multiplication by $q$ and we regroup:
$$H(\rho_d)\cong\bigoplus_{o\subset\Z/d\Z}
\left(\bigoplus_{j\in o}H(\rho\tensor\chi^j)\right).$$

We write $V_o$ for the summand indexed by an orbit $o\subset\Z/d\Z$ in
the last display and $a_o$ for the cardinality of $o$.  As we will see
presently, the hypotheses of the theorem imply that
Proposition~\ref{prop:la} applies to most of the $V_o$ and for each
one where it does, we get a zero of the $L$-function.  Before we do
that, there is one small technical point to take care of: The linear
algebra proposition requires that $V$ be literally self-dual (not
self-dual with a weight) and it implies that $1$ is an eigenvalue of
$\phi$ on $V$.  To get the eigenvalue $q$ that we need, we should
twist $\rho$ by $-1/2$ (which is legitimate once we have fixed
choice of square root of $q$) so that
it has weight 0, apply the lemma, and twist back to get the desired
zero. We leave the details of these points to the reader.

Assuming we have made the twist just mentioned, we need to check which
$V_o$ are self-dual.  Since $\rho$ is self-dual, Poincar\'e duality
gives a non-degenerate pairing on $H(\rho_d)$ which puts
$H(\rho\tensor\chi^j)$ in duality with $H(\rho\tensor\chi^{-j})$.
Thus if $d=q^n+1$ for some $n>0$, then all of the orbits $o$ will
yield a self-dual $V_o$.  Possibly two of these orbits have odd order
(those through $0$ and $d/2$, which have order $1$) and all of the
other have $a_o$ even.  Moreover, for the orbits of even order,
setting $W_{o,i}=H(\rho\tensor\chi^{q^ij_o})$ for some fixed $j_o\in
o$, we have
$$V_o\cong\bigoplus_{i=0}^{a_o-1}W_{o,i}$$
with $W_{o,i}$ and $W_{o,i+a_0/2}$ in duality.

The last point that we need is that $W_{o,i}$ should be
odd-dimensional.  The hypothesis on $\n'$ implies that for all
characters $\chi^j$ of sufficiently high order (depending only on
$E$), the conductor of $\rho\tensor\chi^j$ is odd.  The
Grothendieck-Ogg-Shafarevich dimension formula (mentioned at the end
of the proof of Theorem~\ref{thm:GA}) then implies that for all orbits
$o$ consisting of characters of high order, $H(\rho\tensor\chi^{j_o})$
has odd dimension.

The linear algebra proposition~\ref{prop:la} now implies that for
$d=q^n+1$ and for most orbits $o\subset\Z/d\Z$, $1$ is an eigenvalue
of $\Fr_q$ on $V_o$ (and $q$ is an eigenvalue of $\Fr_q$ on the
corresponding factor of $H(\rho_d)$).  Since each of these orbits has
size $\le 2n$, there is a constant $c$ such that the number of
``good'' orbits is $\ge d/2n$.  Thus
$$\ord_{s=1}L(E/F_d,s)\ge\frac{d}{2n}-c$$
for a constant $c$ depending only on $E$.

To get the assertions over $K_d$, note that in passing from $F_d$ to
$K_d$, each factor $(1-q^{a_o}T^{a_o})$ of $L(E/F_d,T)$ becomes
$(1-qT)^{a_o}$ and so
$$\ord_{s=1}L(E/K_d,s)\ge d-c$$ 
for another $c$ independent of $E$.

This completes our discussion of Theorem~\ref{thm:analytic-ranks}.  We
refer to \cite{Ulmer07b}*{\S2-4} for more details.
\qed

\subsection{Examples}
It is easy to see that the hypotheses in
Theorem~\ref{thm:analytic-ranks} are not very restrictive and that
high analytic ranks are in a sense ubiquitous.  The following
rephrasing of the condition in the theorem should make this clear.

\begin{exer}
  Prove that if $p>3$ and $E$ is an elliptic curve over $K$, then
  Theorem~\ref{thm:analytic-ranks} guarantees that $E$ has unbounded
  analytic rank in the tower $F_d$ if the number of geometric points
  of $\P^1_{\Fq}$ over which $E$ has multiplicative reduction is odd.
\end{exer}

\begin{cor}\label{cor:a-r-examples}
Let $p$ be any prime number, $K=\Fp(t)$, and let $E$ be one of the
curves $E_7$, $E_8$, or $E_9$ defined in Subsection~\ref{ss:examples}
of Lecture~1.  Then 
$$\ord_{s=1}L(E/\Fp(t^{1/d}),s)$$ 
is unbounded as $d$ varies through integers prime to $p$
\end{cor}

\begin{proof}
  If $p>3$, then one sees immediately by considering the discriminant
  and $j$-invariant that $E$ has one finite, non-zero place of
  multiplicative reduction and is tame at 0 and $\infty$, thus it
  satisfies the hypotheses of Theorem~\ref{thm:analytic-ranks}.  If
$p=2$ or 3, one checks using Tate's algorithm that $E$ has good
reduction at all finite non-zero places and is tame at zero, but the
wild part of the conductor at $\infty$ is odd and so the theorem again
applies.
\end{proof}

For another example, take the Legendre curve 
$$y^2=x(x-1)(x-t)$$
over $\Fp(t)$, $p>2$.  It is tame at 0 and $\infty$ and has exactly
one finite, non-zero place of multiplicative reduction.

\section{Large algebraic ranks}
\subsection{Examples via the four-monomial theorem}
Noting that the curves $E_7$, $E_8$, and $E_9$ are defined by
equations involving exactly four monomials, we get a very nice
result on algebraic ranks.

\begin{thm}
  Let $p$ be any prime number, $K=\Fp(t)$, and let $E$ be one of the
  curves $E_7$, $E_8$, or $E_9$ defined in
  Subsection~\ref{ss:examples} of Lecture~1.  Then for all $d$ prime
  to $p$ and all powers $q$ of $p$, the Birch and Swinnerton-Dyer
  conjecture holds for $E$ over $K_d=\Fq(t^{1/d})$.  Moreover, the rank of
  $E(\Fp(t^{1/d}))$ is unbounded as $d$ varies.
\end{thm}

\begin{proof}
  This follows immediately from Corollary~\ref{cor:a-r-examples} and
  Theorem~\ref{thm:4-monos} of Lecture~1 as soon as we note that
  $E/K_d$ is defined by an equation satisfying Shioda's conditions.
\end{proof}

Similar ideas can be used to show that for every prime $p$ and every
genus $g>0$, there is an explicit hyperelliptic curve $C$ over
$\Fp(t)$ such that the Jacobian of $C$ satisfies BSD over
$\Fq(t^{1/d})$ for all $q$ and $d$ and has unbounded rank in the
tower $\Fp(t^{1/d})$.  This is the main theorem of \cite{Ulmer07b}.

\subsection{Examples via Berger's construction}
As we pointed out in Lecture~3, the Shioda 4-monomial construction is
rigid---varying the coefficients does not lead to families that
vary geometrically.  Berger's thesis developed a new construction with
parameters that leads to families of  curves for which the BSD
conjecture holds in a tower of fields.  This together with the analytic
ranks result \ref{thm:analytic-ranks} gives examples of families of
elliptic curves with unbounded ranks.  

To make this concrete, we quote the first example with
parameters from \cite{Berger08} that, together with the analytic rank
construction \ref{thm:analytic-ranks}, gives rise to unbounded
analytic and algebraic ranks.

\begin{thm}[Berger]
Let $k=\Fq$ be a finite field of characteristic $p$ and let $a\in\Fq$ with
$a\neq0,1,2$.  Let $E$ be the elliptic curve over $K=\Fq(t)$ defined by
$$y^2+a(t-1)xy+a(t^2-t)y=x^3+(2a+1)tx^2+a(a+2)t^2x+a^2t^3.$$
Then for all $d$ prime to $p$ the BSD conjecture holds for $E$ over
$\Fq(t^{1/d})$.  Moreover, for every $q$ and $a$ as above, the rank
of $E(\Fq(t^{1/d}))$ is unbounded as $d$ varies.
\end{thm}

\begin{proof}
  This is an instance of Berger's construction
  (Theorem~\ref{thm:Berger} of Lecture~3).  Indeed, let
  $f(x)=x(x-a)/(x-1)$ and $g(y)=y(y-a)/(y-1)$.  Then
  $V(f-tg)\subset\P^1_K\times\P^1_K$ is birational to $E$, which is a
  smooth elliptic curve for all $a\neq0,1$. 
  Berger's Theorem~\ref{thm:Berger} of Lecture~3 shows that $E$
  satisfies BSD over the fields $\Fq(t^{1/d})$.

  The discriminant of $E$ is
$$\Delta=a^2(a-1)^4t^4(t-1)^2\left(a^2t^2-(2a^2-16a+16)t+a^2\right).$$
Assume first that $p>3$.  One checks that $\Delta$ is relatively prime
to $c_4$ so that the zeroes of $\Delta$ are places of multiplicative
reduction.  Since the discriminant (in $t$) of the quadratic factor
$a^2t^2-(2a^2-16a+16)t+a^2$ is $-64(a-1)(a-2)^2$ we see that there are
three finite, non-zero geometric points of multiplicative reduction.
Since $p>3$, the reduction at 0 and $\infty$ is tame and so $\n'$
(defined as in Subsection~\ref{ss:analytic-ranks} of Lecture~4) has
degree 3.  Thus by Theorem~\ref{thm:analytic-ranks} of Lecture~4, $E$
has unbounded analytic ranks in the tower $\Fq(t^{1/d})$ and thus also
unbounded algebraic ranks by the previous paragraph on BSD.

If $p=2$ or 3, one needs to use Tate's algorithm to compute $\n'$,
which again turns out to have degree 3.  We leave the details of this
computation as a pleasant exercise for the reader.
\end{proof}

\numberwithin{equation}{section}


\lecture{More applications of products of curves}

In the last part of Lecture~4, we chose special curves $E$ and used a
domination $\CC\times\DD\ratto\EE$ of the associated surface to deduce
the Tate conjecture for $\EE$ and thus the BSD conjecture for $E$.
This yields an {\it a priori\/} equality of analytic and algebraic
ranks.  We then used other, cohomological, methods (namely the
analytic ranks theorem) to compute the analytic rank.

It turns out to be possible to use domination by a product of curves
and geometry to prove directly results about algebraic ranks and
explicit points.  We sketch some of these applications in this
lecture.

\section{More on Berger's construction}\label{s:moreBerger}
Let $k$ be a field (not necessarily finite), $K=k(t)$, and
$K_d=k(t^{1/d})=k(u)$.  Recall that in Berger's construction we start
with rational curves $\CC=\P^1_k$ and $\DD=\P^1_k$ and rational
functions $f(x)$ on $\CC$ and $g(y)$ on $\DD$.  We get a curve in
$\P^1_K\times\P^1_K$ defined by $f(x)-tg(y)=0$ and we let $E$ be the
smooth proper model over $K$ of this curve.  (Some hypotheses are
required for this to exist, but they are weaker than our standing
hypotheses below.)  The genus of $E$ was computed by Berger in
\cite{Berger08}*{Theorem~3.1}.  All the examples we consider
will be of genus 1 and will have a $K$-rational point.

We establish more notation to state a precise result.  Let us assume
for simplicity all the zeroes and poles of $f$ and $g$ are
$k$-rational.  Write
\begin{equation}\label{eq:divs}
\dvsr(f)=\sum_{i=1}^k a_iP_i-\sum_{i'=1}^{k'} a'_{i'}P'_{i'}
\quad\text{and}\quad \dvsr(g)=\sum_{j=1}^\ell
b_jQ_j-\sum_{j'=1}^{\ell'} b'_{j'}Q'_{j'}
\end{equation}
with $a_i,a'_{i'},b_i,b'_{j'}$ positive integers and $P_i$, $P'_{i'}$,
$Q_j$, and $Q'_{j'}$ distinct $k$-rational points.  Let
$$m=\sum_{i=1}^ka_i=\sum_{i'=1}^{k'} a'_{i'}
\quad\text{and}\quad
n=\sum_{j=1}^\ell b_j=\sum_{j'=1}^{\ell'} b'_{j'}.$$

As standing hypotheses, we assume that: (i) all the multiplicities
$a_i$, $a'_{i'}$, $b_j$, and $b'_{j'}$ are prime to the characteristic
of $k$; and (ii)
$\gcd(a_1\dots,a_k,a'_1,\dots,a'_{k'}
)=\gcd(b_1\dots,b_\ell,b'_1,\dots,b'_{\ell'})=1$.

Under these hypotheses, Berger computes that the genus of $E$ is 
\begin{equation}\label{eq:genus}
g_E=(m -1)(n-1)-\sum_{i,j}\delta(a_i,b_j)-\sum_{i',j'}\delta(a'_{i'},b'_{j'})
\end{equation}
where $\delta(a,b)=(ab-a-b+\gcd(a,b))/2$.

From now on we assume that we have chosen the data $f$ and $g$ so that
$E$ has genus 1.  Two typical cases are where $f$ and $g$ are
quadratic rational functions with simple zeroes and poles, or where
$f$ and $g$ are cubic polynomials.  There is always a $K$-rational
point on $E$; for example, we may take a point where $x$ and $y$ are
zeroes of $f$ and $g$.

Let $\EE_d\to\P^1$ be the elliptic surface over $k$ attached to
$E/K_d$.  It is clear that $\EE_d$ is birational to the closed subset
of $\P^1_k\times\P^1_k\times\P^1_k$ (with coordinates $x,y,u$) defined
by the vanishing of $f(x)-u^dg(y)$.  We saw in Section~\ref{s:Berger}
of Lecture~3 that $\EE$ is dominated by a product of curves and we
would now like to make this more precise.

Recall that we defined covers $\CC_d\to\CC=\P^1$ and
$\DD_d\to\DD=\P^1$ by the equations $z^d=f(x)$ and $w^d=g(y)$.  Note
that there is an action of $\mu_d$, the $d$-th roots of unity, on
$\CC_d$ and on $\DD_d$.

\begin{prop}\label{prop:quotient}
  The surface $\EE_d$ is birationally isomorphic to the quotient
  surface $(\CC_d\times\DD_d)/\mu_d$ where $\mu_d$ acts diagonally.
\end{prop}

\begin{proof}
  We have already noted that $\EE_d$ is birational to the zero set
  $\XX$ of $f(x)-u^dg(y)$ in $\P^1_k\times\P^1_k\times\P^1_k$.  Define
  a rational map from $\CC_d\times\DD_d$ to $\XX$ by sending
  $(x,z,y,w)$ to $(x,y,u=z/w)$.  It is clear that this map factors
  through the quotient $(\CC_d\times\DD_d)/\mu_d$.  Since the map is
  generically of degree $d$, it induces a birational isomorphism
  between $(\CC_d\times\DD_d)/\mu_d$ and $\XX$.  Thus
  $(\CC_d\times\DD_d)/\mu_d$ is birationally isomorphic to $\EE_d$.
\end{proof}

In the next section we will explain how this birational isomorphism can
be used to compute the N\'eron-Severi group of $\EE_d$ and the
Mordell-Weil group $E(K_d)$.

\section{A rank formula}\label{s:rankformula}
We keep the notation and hypotheses of the preceding subsection.
Consider the base $\P^1_k$, the one corresponding to $K$, with
coordinate $t$.  For each geometric point $x$ of this $\P^1_k$, let
$f_x$ be the number of components in the fiber of $\EE\to\P^1$ over
$x$.  For almost all $x$, $f_x=1$ and its value at any point can be
computed using Tate's algorithm.

Define two constants $c_1$ and $c_2$ by the formulae
$$c_1=\sum_{x\neq0,\infty}(f_x-1)$$
and
$$c_2=(k-1)(\ell-1)+(k'-1)(\ell'-1).$$
Here the sum is over geometric points of $\P^1_k$ except $t=0$ and
$t=\infty$ and $k$, $k'$, $\ell$, and $\ell'$ are the numbers of
distinct zeroes and poles of $f$ and $g$
(cf.~equation~(\ref{eq:divs})).  Note that $c_1$ and $c_2$ depend only
on the data defining $E/K$, not on $d$.

\begin{thm}\label{thm:rank-formula}
  Suppose that $k$ is algebraically closed and that $d$ is relatively
  prime to all of the multiplicities $a_i$, $a'_{i'}$, $b_j$, and
  $b'_{j'}$ and to the characteristic of $k$.  Then we have
$$\rk E(K_d)=\rk\Hom(J_{\CC_d},J_{\DD_d})^{\mu_d}-c_1d+c_2.$$
Here $\Hom(\cdots)^{\mu_d}$ signifies the homomorphisms
commuting with the actions of $\mu_d$ on the two Jacobians induced by
its action on the curves.
\end{thm}

\begin{proof}[Sketch of Proof]
In brief, we use the birational isomorphism
$$(\CC_d\times\DD_d)/\mu_d\ratto\EE_d$$
to compute the rank of the
N\'eron-Severi group of $\EE_d$ and then use the Shioda-Tate formula to
compute the rank of $E(K_d)$.

More precisely, we saw in Lecture~2, Subsection~\ref{ss:products} that
the N\'eron-Severi group of the product $\CC_d\times\DD_d$ is
isomorphic to $\Z^2\times\Hom(J_{\CC_d},J_{\DD_d})$.  It follows
easily that the N\'eron-Severi group of the quotient
$(\CC_d\times\DD_d)/\mu_d$ is isomorphic to
$\Z^2\times\Hom(J_{\CC_d},J_{\DD_d})^{\mu_d}$.  

One then keeps careful track of the blow-ups needed to pass from
$(\CC_d\times\DD_d)/\mu_d$ to $\EE_d$.  The effect of blow-ups on
N\'eron-Severi is quite simple and was noted in
Subsection~\ref{ss:blow-ups} of Lecture~2.  This is the main source of
the term $c_2$ in the formula.

Finally, one computes the rank of $E(K_d)$ using the Shioda-Tate
formula, as in Section~\ref{s:Shioda-Tate} of Lecture~3.  This step is
the main source of the term $c_1d$.

The hypothesis that $k$ is algebraically closed is not essential for any
of the above, but it avoids rationality questions that would greatly
complicate the formula.

For full details on the proof of this theorem (in a more general
context) see \cite{UlmerDPCT}*{Section~6}.
\end{proof}

\section{First examples}\label{s:firstexample}
One of the first examples is already quite interesting.  We give a
brief sketch and refer to \cite{UlmerDPCT} for more details.

With notation as in Section~\ref{s:moreBerger}, we take $f(x)=x(x-1)$
and $g(y)=y^2/(1-y)$.  The genus formula (\ref{eq:genus}) shows that $E$
has genus 1.  In fact, the change of coordinates $x=-y/(x+t)$,
$y=-x/t$ brings it into the Weierstrass form
$$y^2+xy+ty=x^3+tx^2.$$

We remark in passing that if the characteristic of $k$ is not $2$, $E$
has multiplicative reduction at $t=1/16$ and good reduction elsewhere
away from $0$ and $\infty$.  Thus by the analytic rank result of
Lecture~2, when $k$ is finite, say $k=\Fp$ and $p>3$, we expect $E$ to
have unbounded analytic rank in the tower $\Fp(t^{1/d})$.  (In fact a
more careful analysis gives the same conclusion for every $p$.)

Now assume that $k$ is algebraically closed.  To compute the constant
$c_1$, one checks that (for $k$ of any characteristic) $E$ has exactly
one irreducible component over each geometric point of $\P^1_k$.  Thus
$c_1=0$.  It is immediate from the definition that $c_2=0$.  Thus our
rank formula yields
$$\rk E(K_d)=\rk\Hom(J_{\CC_d},J_{\DD_d})^{\mu_d}.$$

Next we note that there is an isomorphism $\phi:\CC_d\to\DD_d$ sending
$(x,z)$ to $(y=1/x,w=1/z)$.  This isomorphism {\it anti-commutes\/}
with the $\mu_d$ action: Let $\zeta_d$ be a primitive $d$-th root of
unity and write $[\zeta_d]$ for its action on curves or Jacobians.  Then 
$\phi\compose[\zeta_d]=[\zeta_d^{-1}]\compose\phi$.  Using $\phi$ to
identify $\CC_d$ and $\DD_d$, our rank formula becomes
$$\rk E(K_d)=\rk\en(J_{\CC_d})^{anti-\mu_d}$$
where ``$\en(\cdots)^{anti-\mu_d}$'' denotes those endomorphisms
anti-commuting with $\mu_d$ in the sense above.

Suppose that $k$ has characteristic zero.  Then a consideration of the
(faithful) action of $\en(J_{\CC_d})$ on the differentials
$H^0(J_{\CC_d},\Omega^1)$ shows that $\en(J_{\CC_d})^{anti-\mu_d}=0$
for all $d$ (see \cite{UlmerDPCT}*{7.6}).  We conclude that for $k$ of
characteristic zero, the rank of $E(K_d)$ is zero for all $d$.

Now assume that $k$ has characteristic $p$ (and is algebraically
closed).  If we take $d$ of the form $p^f+1$ then we get many elements
of $\en(J_{\CC_d})^{anti-\mu_d}$. Namely, we consider the Frobenius
$\Fr_{p^f}$ and compute that
$$\Fr_{p^f}\compose[\zeta_d]=[\zeta_d^{p^f}]\compose\Fr_{p^f}=
[\zeta_d^{-1}]\compose\Fr_{p^f}.$$
The same computation shows that $\Fr_{p^f}\compose[\zeta_d^i]$
anticommutes with $\mu_d$ for all $i$.  It turns out that there are
two relations among these endomorphism in $\en(J_{\CC_d})$ if $p>2$
and just one relation if $p=2$ (see \cite{UlmerDPCT}*{7.8-7.10}).  Thus
we find that, for $d$ of the special form $d=p^f+1$,
$$\rk E(\Fpbar(t^{1/d}))=\begin{cases}
d-2&\text{if $p>2$}\\
d-1&\text{if $p=2$.}
\end{cases}$$
The reader may enjoy checking that this is in exact agreement with
what the analytic rank result (Theorem~\ref{thm:analytic-ranks} of
Lecture~4) predicts.  

Somewhat surprisingly, there are {\it more\/} values of $d$ for which
we get high ranks.  A natural question is to identify all pairs
$(p,d)$ such that $E(\Fpbar(t^{1/d})$ has ``new'' rank, i.e, points of
infinite order not coming from smaller values of $d$.  The exact set
of pairs $(p,d)$ for which we get high rank is mysterious.  There are
``systematic'' cases (such as $(p,p^f+1)$, as above, or $(p,2(p-1))$)
and other cases that may be sporadic.  This is the subject of ongoing
research so we will not go into more detail, except to note that the
example in Section~\ref{s:2ndexample} below is relevant to this
question.

\section{Explicit points}\label{s:explicitpoints}
The main ingredients in the rank formula of
Section~\ref{s:rankformula} are the calculation of the N\'eron-Severi
group of a product of curves in terms of homomorphisms of Jacobians
and the Shioda-Tate formula.  Tracing through the proof leads to a
homomorphism
$$\Hom(J_{\CC_d},J_{\DD_d})^{\mu_d}\cong\DivCorr(\CC_d,\DD_d)
\to L^1\NS(\EE_d)\to
\frac{L^1\NS(\EE_d)}{L^2\NS(\EE_d)}\cong E(K_d).$$

For elements of $\Hom(J_{\CC_d},J_{\DD_d})^{\mu_d}$ where we can find
an explicit representation in $\DivCorr(\CC_d,\DD_d)$, the geometry of
Berger's construction leads to explicit points in $E(K_d)$.  This
applies notably to the endomorphisms $\Fr_{p^f}\compose[\zeta_d^i]$
appearing in the analysis of the first example above.  Indeed, these
endomorphisms are represented in $\DivCorr(\CC_d,\DD_d)$ by the graphs
of Frobenius composed with the automorphisms $[\zeta_d^i]$ of $\CC_d$.  

Tracing through the geometry leads to remarkable explicit expressions
for points in $E(K_d)$.  The details of the calculation are presented
in \cite{UlmerDPCT}*{\S8} so we will just state the results here, and
only in the case $p>2$.  

\begin{thm}
Let $p>2$, $k=\Fpbar$ and $K=k(t)$.  Let $E$ be the elliptic
curve
$$y^2+xy+ty=x^3+tx^2$$
over $K$.  Let $q=p^f$, $d=q+1$, $K_d=k(t^{1/d})$, and 
$$P(u)=\left(\frac{u^q(u^q-u)}{(1+4u)^q},
  \frac{u^{2q}(1+2u+2u^q)}{2(1+4u)^{(3q-1)/2}}
  -\frac{u^{2q}}{2(1+4u)^{q-1}}\right).$$ 
Then the points $P_i=P(\zeta_d^it^{1/d})$ for $i=0,\dots,d-1$ lie in
$E(K_d)$ and they generate a finite index subgroup of $E(K_d)$, which
has rank $d-2$.  The relations among them are that
$\sum_{i=0}^{d-1}P_i$ and $\sum_{i=0}^{d-1}(-1)^iP_i$ are torsion.
\end{thm}

It is elementary to check that the points lie in $E(K_d)$.  To check
their independence and the relations by elementary means, one may
compute the height pairing on the lattice they generate.  It turns out
to be a scaling of the direct sum of two copies of the $A_{(d-2)/2}^*$
lattice.  Since we know from the previous section that $E(K_d)$ has
rank $d-2$, the explicit points generate a subgroup of finite index.
As another check that they have finite index, we could compute the
conductor of $E$---it turns out to have degree $d+2$---and apply
Corollary~\ref{cor:rankbound} of Lecture~4.  All this is explained in
detail in \cite{UlmerDPCT}*{\S8}.

\section{Another example}\label{s:2ndexample}
We keep the notation and hypotheses of Sections~\ref{s:moreBerger} and
\ref{s:rankformula}.  For another example, assume that $k=\Fpbar$ with
$p>2$.  Let $f(x)=x/(x^2-1)$ and $g(y)=y(y-1)$.  The curve
$f(x)-tg(y)=0$ has genus 1 and the change of coordinates
$x=(x'+t)/(x'-t)$, $y=-y'/2tx'$ brings it into the Weierstrass form
$$y^{\prime2}+2tx'y'=x^{\prime3}-t^2x'.$$
This curve, call it $E$, has multiplicative reduction of type $I_1$ at
the places dividing $t^2+4$, good reduction at other finite, non-zero
places, and tame reduction at $t=0$ and $t=\infty$.  We find that the
constants $c_1$ and $c_2$ are both zero and that
$$\rk E(\Fpbar(t^{1/d}))=\rk\Hom(J_{\CC_d},J_{\DD_d})^{\mu_d}.$$

Recall that the curves $\CC_d$ and $\DD_d$ are defined by the
equations
$$z^d=f(x)=\frac x{x^2-1}\quad\text{and}\quad w^d=g(y)=y(y-1).$$
Consider the morphism $\phi:\CC_d\to\DD_d$ defined by
$\phi^*(y)=1/(1-x^2)$ and $\phi^*(w)=z^2$.  It is obviously not constant
and so induces a surjective homomorphism $\phi_*:J_{\CC_d}\to J_{\DD_d}$.

The homomorphism $\phi_*$ clearly does not commute with the action of
$\mu_d$.  Indeed, if $\zeta_d$ denotes a primitive $d$-th root of
unity and $[\zeta_d]$ its action on one of the Jacobians, we have
$\phi_*\compose[\zeta_d]=[\zeta_d^2]\compose\phi_*$.  (This formula
already holds at the level of the curves $\CC_d$ and $\DD_d$.)

Now let us assume that $d$ has the form $d=2p^f-1$ and consider the
map $\phi\compose\Fr_{p^f}:\CC_d\to\DD_d$.  Then we find that
$$(\phi\compose\Fr_{p^f})_*\compose[\zeta_d]=
[\zeta_d^{2p^f}]\compose(\phi\compose\Fr_{p^f})_*=
[\zeta_d]\compose(\phi\compose\Fr_{p^f})_*$$ in
$\Hom(J_{\CC_d},J_{\DD_d})$, in other words that
$(\phi\compose\Fr_{p^f})_*$ commutes with the $\mu_d$ action.
Similarly $([\zeta_d^i]\compose\phi\compose\Fr_{p^f})_*$ commutes
with the $\mu_d$ action for all $i$.

Further analysis of the homomorphisms
$([\zeta_d^i]\compose\phi\compose\Fr_{p^f})_*$ in
$\Hom(J_{\CC_d},J_{\DD_d})^{\mu_d}$ (along the lines of
\cite{UlmerDPCT}*{7.8}) shows that they are almost independent; more
precisely, they generate a subgroup of rank $d-1$.  Thus we find (for
$d$ of the form $d=2p^f-1$) that the rank of $E(k(t^{1/d}))$ is at
least $d-1$.

The reader may find it a pleasant exercise to write down explicit
points in this situation, along the lines of the discussion in
Section~\ref{s:explicitpoints} and \cite{UlmerDPCT}*{\S8}.

\section{Further developments}
There have been further developments in the area of rational points on
curves and Jacobians over function fields.  To close, we mention three
of them.

In the examples of Sections~\ref{s:firstexample} and
\ref{s:2ndexample}, the set of $d$ that are ``interesting,'' i.e.,
for which we get high rank over $K_d$, depends very much on $p$, the
characteristic of $k$.  In his thesis (University of Arizona, 2010),
Tommy Occhipinti gives, for every $p$, remarkable examples of
elliptic curves $E$ over $\Fp(t)$ such that
for {\it all\/} $d$ prime to $p$ we have
$$\rk E(\Fpbar(t^{1/d}))\ge d.$$
The curves come from Berger's construction where $f$ and $g$ are
generic degree two rational functions.  The rank inequality comes from
the rank formula in Theorem~\ref{thm:rank-formula} and the Honda-Tate
theory of isogeny classes of abelian varieties over finite fields.

In the opposite direction, the author and Zarhin have given examples
of curves of every genus over $\C(t)$ such that their Jacobians have
bounded rank in the tower of fields $\C(t^{1/\ell^n})$ where $\ell$ is
a prime.  See \cite{UlmerZarhin10}.

Finally, after some encouragement by Dick Gross at PCMI, the author
produced explicit points on the Legendre curve over the fields
$\Fp(\mu_d)(t^{1/d})$ where $d$ has the form $p^f+1$ and proved in a
completely elementary way that they give Mordell-Weil groups of
unbounded rank.  In fact, this construction is considerably easier
than that of Tate and Shafarevich \cite{TateShafarevich67} and could
have been found in the 1960s.  See \cite{UlmerLegendre}.

It appears that this territory is rather fertile and that there is much
still to be discovered about high ranks and explicit points on curves
and Jacobians over function fields.  Happy hunting!

\bibliography{database}

\end{document}